\documentclass[letterpaper, 10pt,reqno]{amsart}

\usepackage[usenames,dvipsnames]{xcolor}
\usepackage[portrait,margin=3cm]{geometry}
\usepackage{mathrsfs}

\usepackage{amssymb,amsthm,amsfonts,amsbsy,amstext,latexsym}
\usepackage[reqno]{amsmath}

\usepackage{orcidlink}

\usepackage{caption}
\usepackage{subcaption}

\usepackage{color}
\usepackage{graphicx}
\usepackage{bbm}
\usepackage{comment}
\usepackage{a4wide,graphicx, listings,lscape, epstopdf, units, textcomp, fancyhdr, url, soul, color}
\usepackage{tikz}
\usetikzlibrary{shapes.geometric}
\usepackage[textsize=small]{todonotes}
\usepackage{enumitem}

\usepackage{mathtools}
\mathtoolsset{showonlyrefs}

\usepackage{xifthen}

\definecolor{MyDarkblue}{rgb}{0,0.08,0.50}
\definecolor{Brickred}{rgb}{0.65,0.08,0}

\usepackage{hyperref,upref}
\hypersetup{
	colorlinks=true,       
	linkcolor=MyDarkblue,          
	citecolor=Brickred,        
	filecolor=red,      
	urlcolor=cyan           
}

\usepackage{cleveref}

\newtheorem*{theorem*}{Theorem}
\newtheorem{theorem}{Theorem}[section]
\newtheorem{lemma}[theorem]{Lemma}
\newtheorem{proposition}[theorem]{Proposition}
\newtheorem{corollary}[theorem]{Corollary}

\newtheorem{claim}[theorem]{Claim}

\newtheorem{conjecture}[theorem]{Conjecture}

\theoremstyle{definition}
\newtheorem{definition}[theorem]{Definition}

\newtheorem{remark}[theorem]{Remark}

\newtheorem{question}[theorem]{Question}

\newcommand{\cB}{\mathcal B}
\newcommand{\cC}{\mathcal C}
\newcommand{\cD}{\mathcal D}
\newcommand{\cE}{\mathcal E}
\newcommand{\cG}{\mathcal G}
\newcommand{\cH}{\mathcal H}
\newcommand{\cI}{\mathcal I}

\newcommand{\cS}{\mathcal S}
\newcommand{\cT}{\mathcal T}

\newcommand{\cY}{\mathcal Y}

\newcommand{\Z}{\mathbb Z}
\newcommand{\N}{\mathbb N}

\newcommand*{\be}{\begin{equation}}
\newcommand*{\ee}{\end{equation}}
\newcommand*{\ba}{\begin{aligned}}
\newcommand*{\ea}{\end{aligned}}
\renewcommand{\P}[1]{\mathbb P\left(#1\right)}
\newcommand{\wt}{\widetilde}

\newcommand{\toinp}{\overset{\mathbb P}{\longrightarrow}}
\newcommand{\cO}{\mathcal O}
\newcommand{\eps}{\varepsilon}

\newcommand{\e}{\mathrm e}
\newcommand{\E}[1]{\mathbb E\left[#1\right]}
\newcommand{\cF}{\mathcal F}

\newcommand{\dd}{\mathrm d}
\newcommand{\cb}[1]{{\color{black}#1}}

\newcommand{\toas}{\overset{\mathrm{a.s.}}{\longrightarrow}}

\newcommand{\invisible}[1]{}

\renewcommand{\P}{\mathbb{P}}
\newcommand{\prob}{\mathbb{P}}
\newcommand{\DWT}{\mathrm{DWT}}
\newcommand{\UH}{\mathrm{UH}}
 \newcommand{\duh}{d_{\mathrm{UH}}}

\newcommand{\condP}[2]{\mathbb{P}\left(\left.#1\,\right\vert#2\right)}
\newcommand{\condE}[2]{\mathbb{E}\left[\left.#1\,\right\vert#2\right]}

\newcommand*{\barr}{\begin{array}{c}}
	\newcommand*{\earr}{\end{array}}
\def \toinp    {\buildrel {\prob}\over{\longrightarrow}}

\def \toas     {\buildrel {\mathrm{a.s.}}\over{\longrightarrow}}

\renewcommand{\P}[1]{\mathbb{P}\!\left(#1\right)}

\title{On the depth of depth-weighted trees}

    \author[Lichev]{Lyuben Lichev\orcidlink{0000-0002-3998-8920}}
    \address{Institute of Statistics and Mathematical Methods in Economics, TU Wien, A-1040 Vienna, Austria}
    \email{lyuben.lichev@tuwien.ac.at}
    
    \author{Amitai Linker}
    \address{Departamento de Matem\'aticas, Facultad de Ciencias Exactas, Universidad Andr\'es Bello,
Sazi\'e 2212, Santiago, Chile}
    \email{amitai.linker@unab.cl}
    
    \author[Lodewijks]{Bas Lodewijks\orcidlink{0000-0001-5624-2410}}
	\address{School of Mathematical and Physical Sciences, University of Sheffield, Sheffield, England}
	\email{bas.lodewijks@sheffield.ac.uk}

    \author{Dieter Mitsche}
    \address{IMC, Pontif\'{i}cia Univ. Cat\'{o}lica, Avda. Vicu\~na Mackenna 4860, Santiago, Chile}
    \email{dieter.mitsche@uc.cl}
    
        \date{\today}

\begin{document}

\begin{abstract}
		 The depth-weighted tree DWT($f$) with weight function $f:\{0,1,\ldots\}\to (0,\infty)$ is a dynamic random tree grown from a root $r$ where vertices arrive consecutively and every new vertex attaches to a parent $u$ with probability proportional to $f$(distance between $u$~and~$r$).
This work is dedicated to a systematic analysis of the depth of DWT($f$). Namely, we provide precise analytic expressions of the typical depth of DWT($f$) for convergent, periodic, slowly growing, and (super-)exponentially growing weight functions.
Furthermore, for bounded or exponentially growing $f$, we determine the typical depth up to a multiplicative constant, thus confirming and strengthening a conjecture of Leckey, Mitsche and Wormald.
\end{abstract}

\maketitle 

\noindent 
        {\bf Keywords:}  Random tree, Depth, Continuous-time branching process, Cox process
        \\\\
        {\bf MSC Subject Classifications:} Primary: 60C05, Secondary: 60J80, 60G55.

\begin{figure}[h]
    \centering
    \includegraphics[width=0.475\linewidth]{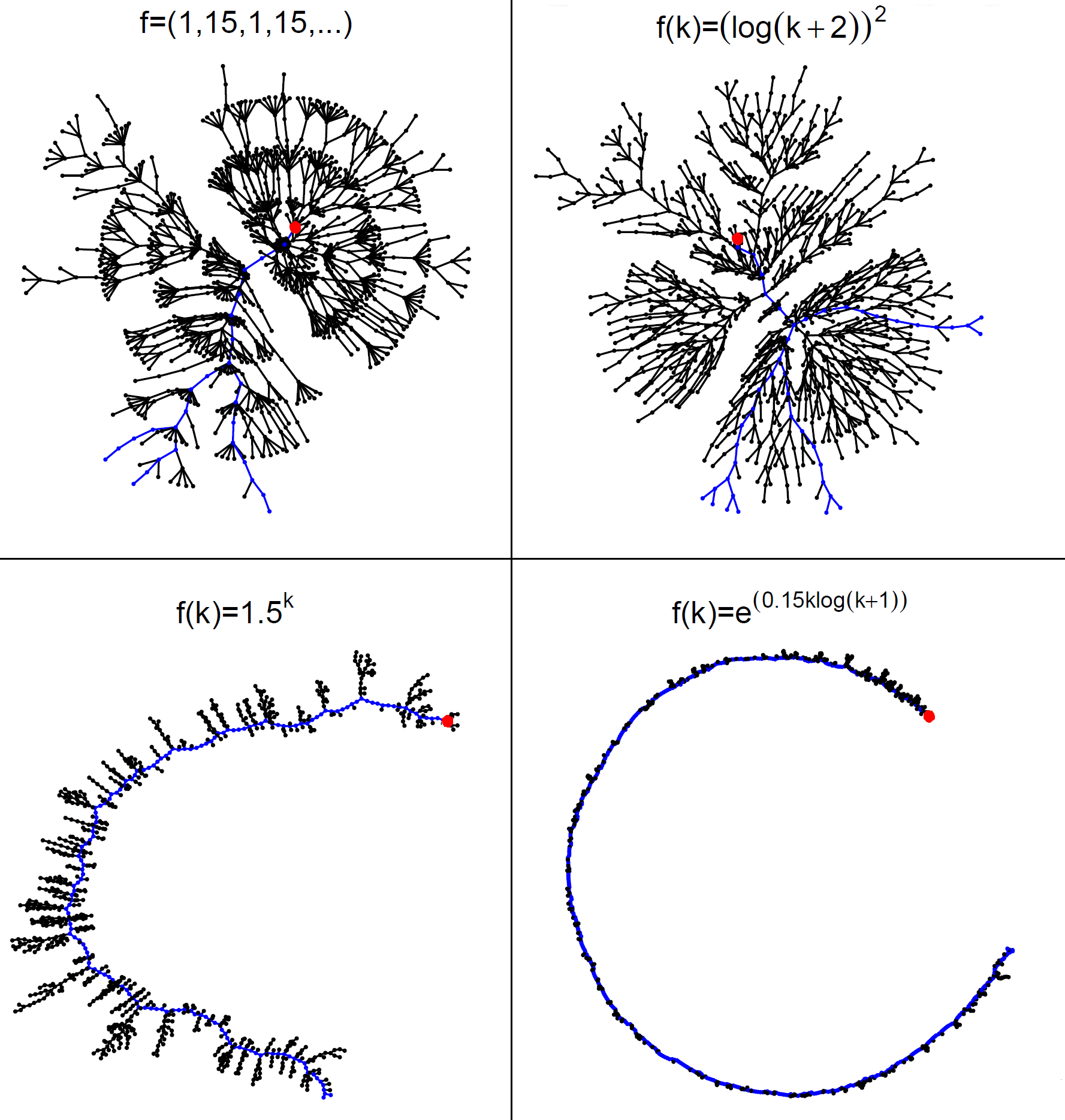}
    \caption{Simulations of DWT($f$) for different weight functions $f$. The big red dot denotes the root and the blue paths realise the depth of the tree, that is, the length of the longest path from the root to a leaf.}
    \label{fig:trees}
\end{figure}

\section{Introduction}

Random graphs have been a central topic in probability theory and network science for decades. 
Their applicability as models for real-world networks and the intrinsic mathematical interest of many classic models has led to the development of a general theoretical framework, see e.g.\ the books of Van der Hofstad~\cite{vdH17,vdH24}.
Some of the most intriguing mathematical questions in the field concern the evolution of growing random graphs and networks; here, we restrict our attention to the most widely studied setting of dynamic random (rooted) trees.

The variety of growth mechanisms for random trees which have been proposed and studied in the literature can be classified into two categories according to the power of a fixed vertex to attract new neighbours.
The first category consists of models where this power evolves with the dynamic construction of the tree.
One classic example is given by \emph{preferential attachment} models where the probability that a vertex receives a new child at a given step is proportional to a growing function of its current degree; see the books~\cite{Gol19,NBW11,vdH24} for a summary of the vast literature on the topic. We do mention a very recent paper on preferential attachment trees that studies monotonicity of the depth constant~\cite{Mon26}, which relates to the statistics of interest in this paper. Another classic model is the \emph{online nearest-neighbour tree} where vertices are embedded consecutively and randomly in a compact metric space and every new vertex sends an edge to the nearest vertex already embedded~\cite{Ald18,BBCS23,BMMS24,Cas25,LM24,PW08,Tra25,Wad09}. 

In the second category of models, the power of a vertex to attract neighbours is determined by a \emph{weight} attributed at birth: namely, the probability that a vertex receives a child at a given step is always proportional to its weight. 
The most famous example here is the \emph{random recursive tree} studied in a long line of research~\cite{BB-RKK23a,BB-RKK23b,Drm09,FM10,FMP08,Fuc08,Lod22} where all vertices receive equal weights. 
Another natural connection rule giving rise to the \emph{depth-weighted tree} where the weight attributed to a vertex is a function of its distance to the root. 
Here, the choice of a function may serve to privilege connections closer to the root (e.g.\ when it is the primary source of a signal) or to the leaves instead (e.g.\ in the mathematics genealogy tree), or also adapt to a more complex hierarchical structure over different generations.
Perhaps surprisingly, unlike the remaining models described above, depth-weighted trees were introduced by Leckey, Mitsche and Wormald~\cite{LeckMitWor20} only a few years ago (see also~\cite{Omeretal} for a very recent and somewhat related model). 

\begin{definition}[Depth-weighted tree (DWT)]\label{def:DWT}
Fix a function $f:\{0,1,\ldots\}\to (0, \infty)$. An (inclusion-wise) increasing sequence of labelled rooted trees $(T_n)_{n\in\N} = (T_n(f))_{n\in \mathbb N}$ with depth functions $(d_n)_{n\in\N}$ with $d_n:\{1,\ldots,n\}\to \{0,\ldots,n-1\}$ is constructed as follows:
\begin{itemize}
    \item The tree $T_1$ contains a single vertex labelled $1$ (the root) on depth $d_1(1)=0$.
    \item For every $n\in \mathbb N$, the tree $T_{n+1}$ is constructed conditionally on $T_n$ by adding a vertex labelled $n+1$ to $T_n$ and connecting it to a single vertex, which is $i\in [n]$ with probability
    \begin{equation*}\label{eq:connprob}\frac{f(d_n(i))}{\sum_{j=1}^n f(d_n(j))}.
    \end{equation*}
    On this event, we say that $n+1$ \emph{attaches} to vertex $i$ and to depth $d_n(i)$ (and thus vertex $n+1$ has depth $d_{n+1}(n+1) = d_n(i)+1$).
\end{itemize} 
In the described context, $f$ is called a \emph{weight function} and $T_n = T_n(f)$ is called the \emph{depth-weighted tree on $n$ vertices with weight function $f$}. 
Moreover, we define $T_{\infty}$ to be the infinite tree obtained as the union of the trees $(T_n)_{n\in\N}$.
\end{definition}

For a rooted tree $T$, the \emph{depth of the tree} $d(T)$ denotes the maximum over the depths of the vertices in $T$.
A description of the depth of DWTs was the main focus of~\cite{LeckMitWor20} where the following bounds on $\mathbb E[d(T_n)]$ for different weight functions were provided. 

\begin{theorem}[Theorem~2.3 in~\cite{LeckMitWor20}]\label{thm:LMW}
Fix a weight function $f:\{0,1,\ldots\}\to (0,\infty)$. 
\begin{enumerate}
\item[\emph{(a)}] If $0<\inf f \le \sup f < \infty$, then $\mathbb E[d(T_n)] = \Theta(\log n)$.
\item[\emph{(b)}] If $f: k\mapsto (k+1)^{\alpha}$ for some $\alpha \ge 0$, then $\mathbb E[d(T_n)] = \Theta(\log n)$.
\item[\emph{(c)}] If $f: k\mapsto \exp(k^{\beta})$ for some $\beta\in (0,1)$, then $\mathbb E[d(T_n)] = O((\log n)^{1/(1-\beta)})$.
\item[\emph{(d)}] If $f: k\mapsto c^k$ for some $c > 2$, then $\mathbb E[d(T_n)] = \Omega(n/\log n)$.
\item[\emph{(e)}] If $f: k\mapsto \exp(ak\log k)$ for some $a > 1$, then $\mathbb E[d(T_n)] = \Theta(n)$.
\end{enumerate}
\end{theorem}

We extend the analysis of DWTs and significantly strengthen parts (a), (d) and (e) of Theorem~\ref{thm:LMW} by determining the typical growth of $d(T_n)$ (sometimes up to lower order terms) for bounded, slowly growing, exponential, and super-exponential weight functions.
The next subsection is dedicated to a discussion of our main results  and open problems for further study followed by an outline of the proof ideas.

\subsection{Our results}

We separately analyse $d(T_n)$ in  each of the four previously described regimes. 

\subsection*{Bounded weight functions} 
We first study the depth of DWTs with weight functions $f$ bounded away from 0 and infinity, that is, $0 < \inf f \le \sup f < \infty$.
This regime is fundamental as it consists of models which may be considered as perturbations of the random recursive tree. Indeed, beyond the apparent similarity, our first result shows that, in this setting, the depth of DWTs behaves similarly to that of the random recursive tree. 

\begin{proposition}\label{prop:bddH}
For every weight function $f$ bounded away from $0$ and infinity, there exist constants $\ell=\ell(\inf f,\sup f),L=L(\inf f,\sup f)>0$ such that almost surely
        \begin{equation*} 
        \ell\le \liminf_{n\to\infty} \frac{d(T_n)}{\log n}\le \limsup_{n\to\infty} \frac{d(T_n)}{\log n}\le L. 
        \end{equation*} 
\end{proposition}

Under additional regularity assumptions on the weight function $f$, it is possible to strengthen Proposition~\ref{prop:bddH} to a strong law of large numbers, mirroring precisely the behaviour of $d(T_n)$ for random recursive trees.

\begin{theorem}\label{thrm:bddH}
For every convergent or periodic weight function bounded away from $0$,
\begin{equation}\label{eq:1.3}
\frac{d(T_n)}{\log n}\toas\e.
\end{equation}
\end{theorem}

\subsection*{Slowly growing weight functions}
A weight function $f$ is called \emph{slowly growing} if it is non-decreasing, tends to infinity and satisfies
\be \tag{WSV}\label{eq:widesv}
\lim_{k\to\infty} \frac{f(k/f(k))}{f(k)}=\lim_{k\to\infty} \frac{f(kf(k))}{f(k)}=1.
\ee 
Examples of slowly growing functions include $n\mapsto \log(n+1)$ or $n\mapsto \exp((\log(n+1))^{\alpha})$ with $\alpha < 1/2$. Note that any slowly growing function is slowly varying (but not the other way around). Slowly growing weight functions introduce a weak bias towards larger depths, and it is a priori unclear if this bias makes the tree significantly deeper.
Our next theorem shows that this is not the case and that~\eqref{eq:1.3} also holds for slowly growing functions.

\begin{theorem}\label{thrm:svf}
For every slowly growing weight function,
\begin{equation*} 
\frac{d(T_n)}{\log n}\toas \e.
\end{equation*} 
\end{theorem}

\begin{remark}\label{rem:lowerslow}
We note that the lower bound in \Cref{thrm:svf} can be derived by combining~\cite[Proposition~2.7]{LeckMitWor20} showing that the depth of $T_n(f)$ for any non-decreasing weight function $f$ stochastically dominates the depth of an $n$-vertex random recursive tree, and~\cite[Theorem $1$]{Pit94} showing that the depth of the latter rescaled by $\log n$ almost surely converges to $\e$. Thus, our contribution is the proof of the upper bound.
\end{remark}

\subsection*{Exponential weight functions}
Next, we focus on functions for which there is $c \ge 1$ with
\be \tag{EG}\label{eq:ExpG} 
\lim_{k\to\infty} \frac{f(k+1)}{f(k)}=c. 
\ee
When $c>1$, these are exactly the functions exhibiting exponential growth. 
Unlike previous regimes, the strong bias towards \cb{larger} depths ensures that vertices typically attach to a parent at (nearly) maximal distance from the root, causing the depth of $T_n$ to grow much faster. Our next result formalises this intuition.

\begin{theorem}\label{thm:main_exp}
Fix a weight function $f$ satisfying assumption~\eqref{eq:ExpG} for some $c\geq 1$.  
\begin{enumerate}
    \item[\emph{(i)}] If $c>1$, there is $\gamma = \gamma(\cb{f}) > 0$ such that, with high probability\footnote{That is, with probability tending to 1 as $n\to \infty$.}, $d(T_n)\ge \gamma n$.
    \item[\emph{(ii)}] If $c=1$, for every $\eps > 0$, $d(T_n)\le \eps n$ with probability $1-\e^{-\omega(n)}$. Thus, $d(T_n)/n\toas 0$.
\end{enumerate}
\end{theorem}

\begin{remark}
For every $\eta>0$, at the cost of allowing $\gamma$ in \Cref{thm:main_exp}(i) to depend on $\eta$, our proof shows that \Cref{thm:main_exp}(i) holds with probability at least $1-n^{-\eta}$.
\end{remark}

The result in Theorem~\ref{thm:main_exp}(i) is a significant strengthening compared to part~(d) in Theorem~\ref{thm:LMW} in~\cite{LeckMitWor20}, as we provide the correct asymptotic order of the depth for all constants $c>1$ and for a wider range of functions $f$, not just the purely exponential case $k\mapsto c^k$. The second part of Conjecture~$2.4$ from~\cite{LeckMitWor20} (suggesting that the expected depth of $T_n$ is linear under the assumptions of part~(d) of \Cref{thm:LMW}) follows as a trivial corollary of Theorem~\ref{thm:main_exp}(i).

\begin{remark}\label{rem:subexp}
The techniques from the proof of Theorem~\ref{thm:main_exp} can be extended to functions satisfying a slightly weaker condition compared to assumption~\eqref{eq:ExpG}. 
For example, consider the weight function $k\mapsto \e^{k/(\log(k+2))^{\alpha}}$ with $0<\alpha<1$. 
For this slightly sub-exponential function, the methods in \Cref{sec:exp_weight_fns} can be adapted to show that there is $\nu>0$ such that, with high probability,
\begin{equation}\label{eq:weakbound}
d(T_n)\ge n\e^{-\nu(\log n)^{\alpha}(\log\log n)^2}.
\end{equation}
Notice that this lower bound is sublinear but still asymptotically larger than any bound of the form $n^{1-\varepsilon}$. Interestingly, this bound becomes trivial at $\alpha=1$ where~\eqref{eq:weakbound} is of order $o(1)$, which hints at a possible phase transition of the model for \cb{weight} functions of the form $k\mapsto \e^{ck/\log(k+2)}$. See also Figures~\ref{fig:polynomial1} and~\ref{fig:polynomial2} for numerical simulations, and the remark at the end of Section~\ref{sec:exp_weight_fns} for further comments on the proof.
\end{remark}

\begin{question}\label{Q:subexp}
For the weight function $k\mapsto \e^{ck/\log(k+2)}$ with $c>0$, does there exist $\beta=\beta(c)\in (0,1)$ such that $d(T_n)=n^{\beta+o(1)}$ with high probability?
\end{question}

\begin{figure}[h]
     \centering

    \includegraphics[width=0.6\textwidth]{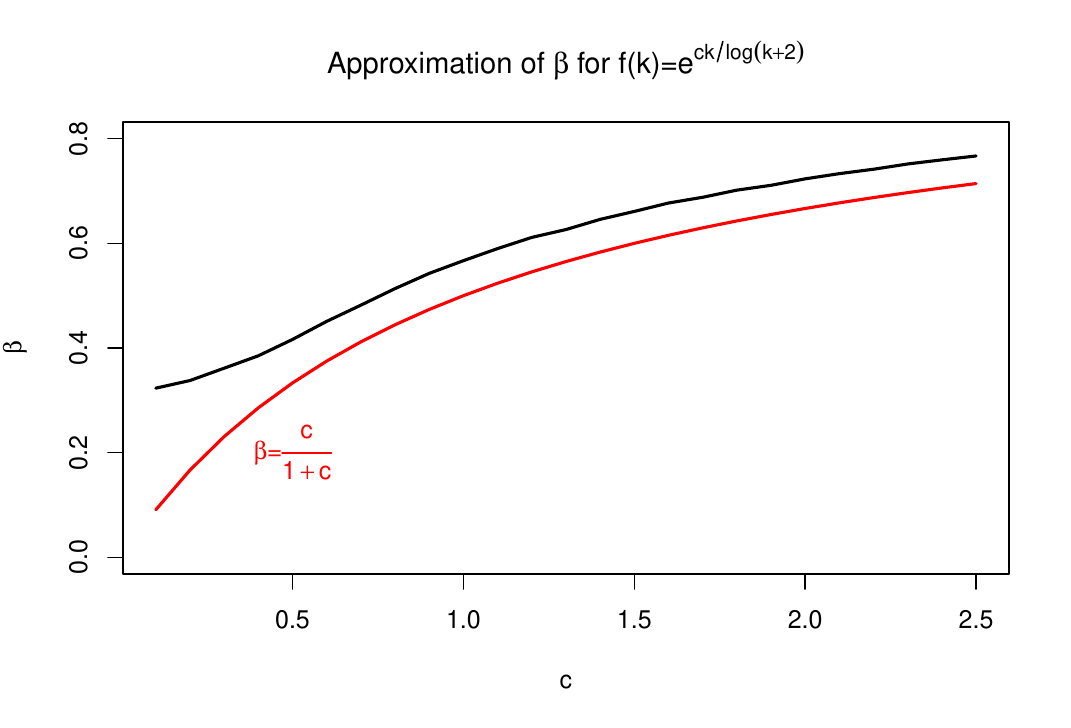}
         \caption{A numerical approximation of the constant $\beta$ predicted in Question~\ref{Q:subexp} (in black) for the tree with weight function $f(k)=\e^{ck/\log(k+2)}$ and $n=20.000$. Here, $c$ ranges between $0.1$ and $2.5$ with increments of $0.1$, where we take an average over $50$ samples. The red line serves as a comparison.}\label{fig:polynomial1}
    \end{figure}
    
     \begin{figure}[h]
         \centering
         \includegraphics[width=0.6\textwidth]{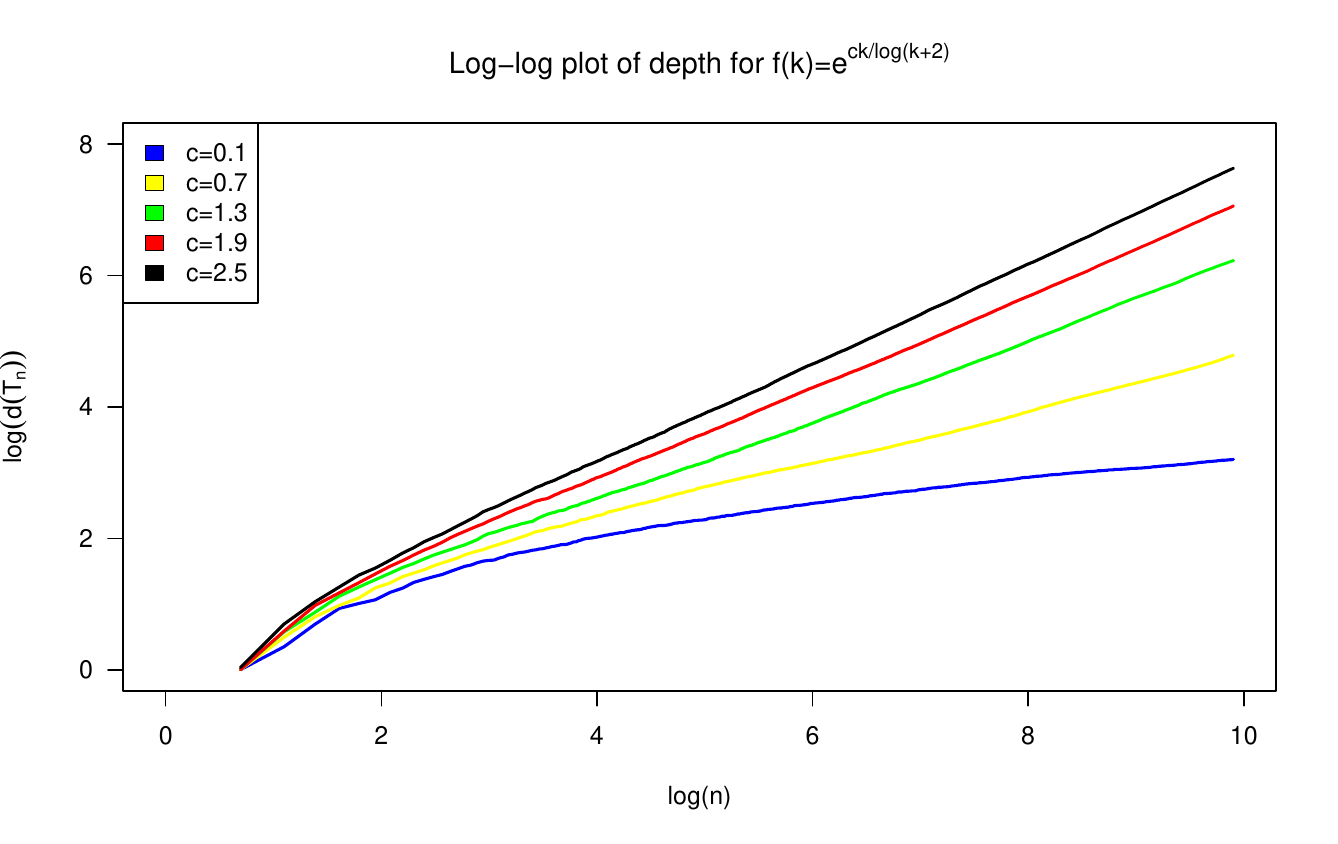}
         \caption{A log-log plot of the growth of the depth of the tree with weight function $f(k)=\e^{ck/\log(k+2)}$. We have taken an average over 50 samples for each value of $c$. }\label{fig:polynomial2}
     \end{figure}

{Under the conditions of Theorem~\ref{thm:main_exp}, $d(T_n)/n$ appears to converge in computational experiments, which opens the door to the possibility of having a.s.\ convergence of this ratio, see \Cref{fig:expsim}. We leave this as a conjecture:

\begin{conjecture}\label{conj:exp}
For every weight function $f$ satisfying assumption~\eqref{eq:ExpG} for some $c > 1$, there is $\nu=\nu(c)\in(0,1)$ such that 
\begin{equation*} 
\frac{d(T_n)}{n}\toas \nu.
\end{equation*} 
\end{conjecture}

\begin{figure}[h]
    \centering
    \includegraphics[width=0.6\linewidth]{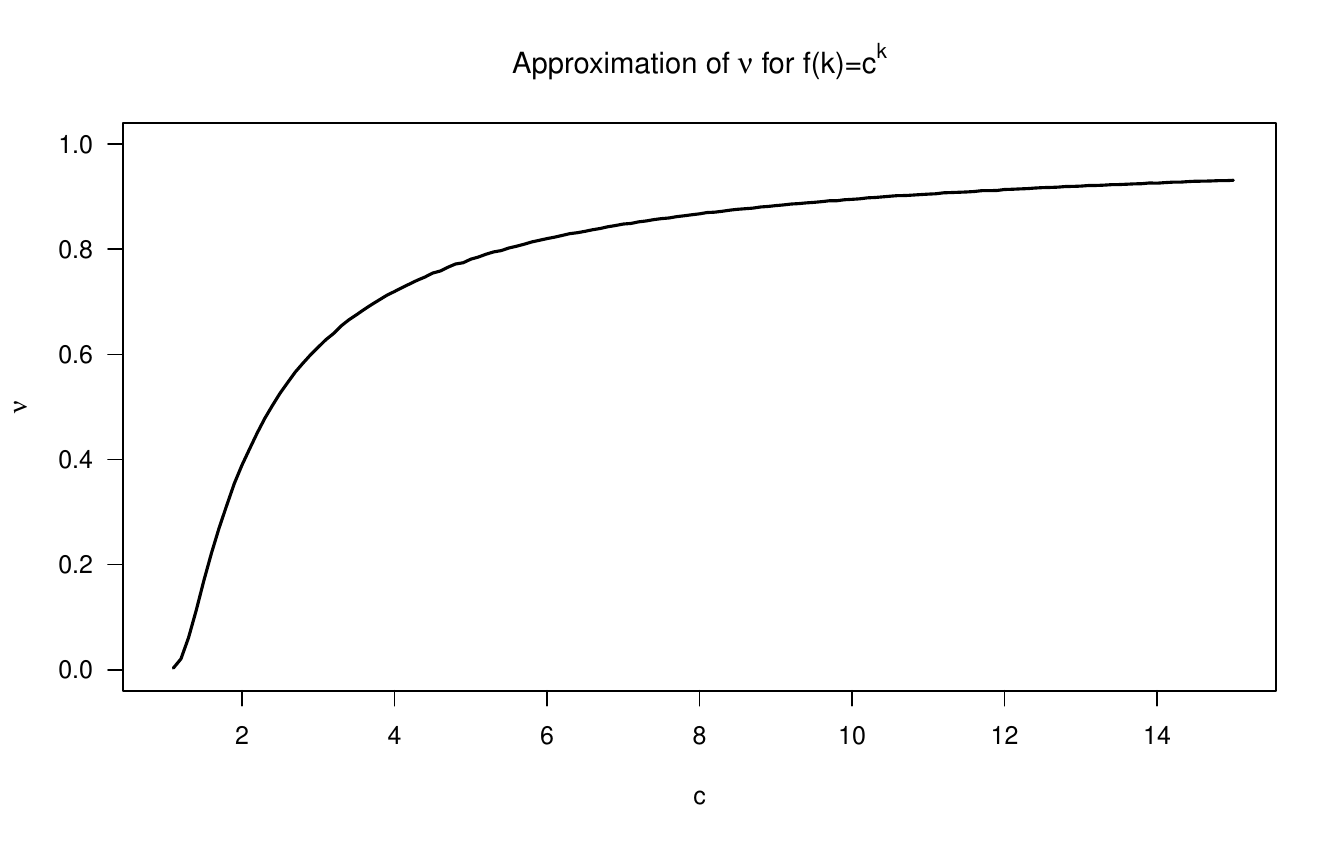}
    \caption{A numerical approximation of the predicted constant $\nu$ in Conjecture~\ref{conj:exp}. The figure shows the average value of $d(T_n)/n$ for $n=15.000$ and weight function $f(k)=c^k$. The plot is obtained by taking an average over $50$ independent samples for each $c$ between $1.1$ and $15$ with increments of $0.1$.}
    \label{fig:expsim}
\end{figure}

\subsection*{Super-exponential weight functions}

Finally, we present several results concerning certain weight functions $f$ growing \emph{faster} than exponentially. Namely, we assume that
        \be\tag{D} \label{ass:D}
        \lim_{j\to\infty}\frac{f(j+1)}{f(j)}=\infty
    \ee 	 
and define the quantity
    \begin{equation*} \label{eq:In}
    I_n\coloneq\sum_{j=0}^n \frac{f(j)}{f(j+1)}. 
    \end{equation*} 
Note that, by Assumption~\eqref{ass:D}, it follows that $I_n=o(n)$. As implied by the following theorem, when $I_n$ is large, it determines precisely the typical number of vertices outside a longest path starting from the root of $T_n$.

Our main theorem for super-exponentially growing functions is then the following:
\begin{theorem}\label{thrm:supexp}
Fix a weight function satisfying Assumption~\eqref{ass:D}.
\begin{enumerate}
\item[\emph{(i)}] If $I_n$ converges, then there is an integer-valued random variable $K$ such that
\begin{equation*}n-d(T_n)\toas K.\end{equation*}
\item[\emph{(ii)}] If $I_n = \omega(1)$, then 
\begin{equation*}
\frac{n-d(T_n)}{I_n}\toinp 1.
\end{equation*}
\end{enumerate}
\end{theorem}
\noindent
Observe that, when $I_n$ converges, part (i) of Theorem~\ref{thrm:supexp} shows that all but a finite (random) number of vertices belong to a unique infinite path from the root of the tree.

Theorem~\ref{thrm:supexp} significantly strengthens part~(e) in Theorem~\ref{thm:LMW} from~\cite{LeckMitWor20}, \cb{where we note that the assumption $a>1$ is equivalent to the convergence of $I$ for a weight function of the type $k\mapsto \exp(ak\log k)$. 
Thus, Theorem~\ref{thm:LMW}(e)} concerns a special case of (i), and its proof is based on the fact that the probability of the event that $T_n$ is a path has a non-zero limit, so that the expected value of the depth of $T_n$ is linear. We extend this result to more general functions and obtain a much stronger convergence result. Case $(ii)$ is entirely novel.

\subsection{Outline of the proof ideas} Important prerequisites for our proofs are two continuous-time processes presented in \Cref{sec:Poisson,sec:Cox}. 
The first of them is a continuous-time branching process version of the DWT which was also used in \cite{LeckMitWor20}. Here, each vertex in the process independently produces offspring according to a Poisson process with intensity depending on its depth in the tree. 
The second process simplifies the structure of the previously described branching process by recursively tracking only the profile of the process (the number of vertices on each depth).
It presents a convenient framework used in each of the proofs of our main results and is among the main technical novelties of this work.

The proofs of \Cref{prop:bddH}, \Cref{thrm:bddH,thrm:svf} follow a similar theme. 
The main point is to estimate the stopping times $\tau_n$ (when the $n$-th vertex is born) and $\tau_{1,k}$ (when the first vertex on depth $k$ is born) in the continuous-time version sufficiently precisely: then, any function $k(n)$ such that $\tau_n\approx \tau_{1,k(n)}$ actually asymptotically characterises the typical depth of $T_n$.
The short proof of \Cref{prop:bddH} makes use of known estimates of $\tau_n$ and $\tau_{1,k}$.
Achieving the more precise conclusions of \Cref{thrm:bddH,thrm:svf} requires a refined understanding of the growth of the DWT. 
Estimating $\tau_{1,k}$ is the comparatively simpler part: this requires standard first moment computations (for convergent and slowly growing functions), a subadditivity argument (for convergent functions) and an old result of Biggins~\cite{Big76} (for periodic functions).
The main complication in each of these proofs is the estimation of $\tau_n$, which requires finding a convenient way to circumvent the potential irregularities among the small values of $f$. Note that, since vertices produce offspring at different rates (depending on their depths), these results cannot be obtained directly from the well-known theory on Crump--Mode--Jagers branching processes. To estimate $\tau_n$, we analyse the second continuous-time process over an interval $[0,t]$ in two phases.
During the short first phase taking place over the interval $[0,\eps t]$ with $\eps$ small, the process typically produces a vertex on a sufficiently large depth where $f$ starts to behave regularly.
The second phase analyses the process over the interval $[\eps t,t]$ and uses the regularity of $f$ to show convenient lower bounds for the process $N(\cdot)$ which keeps track of the total number of vertices.

We turn to the proof of \Cref{thm:main_exp}. Part (ii) relies on a short renewal-process-type of argument: it uses that the time needed to increase the depth of the tree from $k$ to $k+1$ is stochastically minimised when all unsuccessful attempts result in a new vertex on depth $k$.
The proof of \Cref{thm:main_exp}(i) is much harder and uses the framework defined in \Cref{sec:Cox} in a different and novel way. 
In a nutshell, our approach estimates the profile of the depths in $T_n$ in two steps. 
First, we provide an exponential concentration estimate on the event that some depth is a sort of non-strict local maximum (in a suitably chosen constant-size window) containing exactly $s$ vertices. 
This step is the most technical part of the proof: it consists in a recursive analysis of the continuous-time trajectories of blocks of consecutive depths, and relies on several concentration bounds, stochastic comparisons between processes and a concise case analysis.
The second step analyses a deterministic algorithm showing that every depth $r$ is within a small distance from a local maximum $r'$, and one can reach $r'$ from $r$ by making jumps of constant length towards depths containing (strictly) more vertices.
Finally, to show \Cref{thm:main_exp}(i), we show that typically the interval $[0,cn]$ of depths contains fewer than $n$ vertices of $T_n$: here, the first step serves to control the number of vertices on depths which are local maxima while the second step controls the vertices on all remaining depths.

Finally, the proof of \Cref{thrm:supexp} relies on a simple but crucial coupling result and a rather delicate first-second moment argument.
The coupling result states that, in the framework of \Cref{sec:Cox}, there exists a sequence of independent exponential random variables $(E_k)_{k=0}^{\infty}$ with suitable rates satisfying the following two properties: 
first, $E_k$ almost surely dominates the difference between the birth times of the first vertices on depths $k$ and $k+1$ and, second, for every $k$, the family $E_k,E_{k+1},\ldots$ is independent of the process on depths $0,\ldots,k$.
The sequence $(E_k)_{k=0}^{\infty}$ is important for two related reasons: the first property above allows us to suitably approximate the explosion time of the continuous-time process which, 
combined with the independence encoded in the sequence, makes it possible to precisely bound the joint moments of the number of vertices on different depths at the explosion time. 
The latter task is achieved via delicate recursive computations.

\vspace{1em}
\noindent
\textbf{Structure of the paper.} In Section~\ref{sec:prelims}, we gather notation and present preliminary results. 
Then, Section~\ref{sec:bdd} is dedicated to proving the results for bounded and slowly growing weight functions, namely Proposition~\ref{prop:bddH} and Theorems~\ref{thrm:bddH} and~\ref{thrm:svf}. 
In Section~\ref{sec:exp_weight_fns}, we prove Theorem~\ref{thm:main_exp} regarding exponential weight functions. 
Finally, we deal with super-exponential weight functions in Section~\ref{sec:supexp} where we show Theorem~\ref{thrm:supexp}.

\section{Preliminaries}\label{sec:prelims}

In this section we present general notation and several concentration results that we use throughout. 
Furthermore, we introduce two continuous-time versions of the DWT model, serving as a basis of several of our proofs, in Section~\ref{sec:contembed}.

\subsection{General notation and concentration results}
We start with some standard notation. We write $\mathbb N$ for the set of positive integers and $\mathbb N_0 = \mathbb N\cup \{0\}$.
Moreover, for a positive integer $n$, we define $[n] = \{1,\ldots,n\}$. Rounding of positive expressions to the nearest integer is often ignored if irrelevant for the context.
Furthermore, we \cb{adopt} the convention that empty sums are equal to 0 while empty products are equal to 1. We use the standard asymptotic notations $o$, $O$, $\omega$ and $\Omega$ (where $f = \Omega(g)$ if $g = O(f)$ and $f = \omega(g)$ if $g = o(f)$). 
By default, the implicit constants in $O(\cdot)$ and $\Omega(\cdot)$ are absolute.

\subsection*{Concentration inequalities}

We begin with the well-known Paley-Zygmund inequality for concentration of non-negative random variables with bounded second moment.

\begin{lemma}[Paley-Zygmund inequality, see~\cite{Ste04}]\label{lem:PZ}
For every $\theta\in [0,1]$ and every non-negative random variable $Z$ with finite variance, $\mathbb P(Z\ge \theta \mathbb E[Z])\ge (1-\theta)^2 \mathbb E[Z]^2/\mathbb E[Z^2]$.
\end{lemma}

Next, we state standard Chernoff bounds for binomial and Poisson random variables, see e.g.\ \cite[Theorems A.1.12--A.1.15]{AS16} (stated for $\eps = t/\mathbb E[X]$) and \cite[Theorem~2.1]{JLR11}.

\begin{lemma}[Chernoff bounds]\label{lem:chernoff}
For every binomial or Poisson random variable $X$ and $t>0$,
\begin{equation*}\mathbb P(X-\mathbb E[X]\le -t) \le \exp(-t^2/2\mathbb E[X])\end{equation*}
and
\begin{equation*}\mathbb P(X-\mathbb E[X]\ge t) \le \exp\bigg(-\frac{t^2}{2(\mathbb E[X]+t/3)}\bigg).\end{equation*}
\end{lemma}

Next, we state an exponential concentration inequality due to Bernstein for sums of independent exponential random variables; a more general version appears as \cite[Corollary~2.9.2]{Ver18}.

\begin{lemma}[see Corollary~2.9.2 in \cite{Ver18}]\label{lem:Ber}
Fix independent exponential random variables $(X_i)_{i=1}^n$ of rate $1$ and sum $X$. Then, there are constants $c_1,c_2$ such that, for every $t\ge 0$,
\begin{equation*}\mathbb P(|X-n|\ge t)\le 2\exp\bigg(-\min\bigg\{\frac{t^2}{c_1n},\frac{t}{c_2}\bigg\}\bigg).\end{equation*}
\end{lemma}

We finish this section with slightly more advanced concentration results for Poisson processes.
Fix a measure space \((S, \mathcal{F},\Lambda)\). 
A \emph{Poisson point process $N$ with intensity measure $\Lambda$} is a family of random variables \(\{N(A) : A \in \mathcal{F}\}\) with values in $\N_0$ such that the following properties hold:
\begin{itemize}
    \item for each $A\in\mathcal{F}$, $N(A)$ is a Poisson random variable with parameter $\Lambda(A)$, and
    \item for each family $(A_i)$ of pairwise disjoint measurable sets, $(N(A_i))$ are jointly independent random variables.
\end{itemize}
Furthermore, if the intensity measure $\Lambda$ varies according to a stochastic process on the set of measures on $(S,\cF)$, $N$ is instead called a \textit{Cox process}; concrete examples of such processes will be constructed in \Cref{sec:contembed}. 

When the set $S$ in the previous definition is the real line, the function $t\to N([0,t])$ defines a c\'adl\'ag stochastic process $(N(t))$. 
In the scenario where $\Lambda$ is not random \cb{and $\Lambda[0,t]<\infty$ for all $t\geq 0$}, it is well known that the centred process $(N(t)-\Lambda[0,t])$ is a martingale. 
Among several exponential inequalities available in this setting (see for instance \cite{Kroll,Rey06,ReynaudBouret2003AdaptiveEO,LeGuvel2021}), we make extensive use of the following one due to Le Gu\'evel~\cite{LeGuvel2021}.

\begin{theorem}[Theorem 5 in \cite{LeGuvel2021}]\label{theo: LeGuvel}
Consider a Poisson point process on the real line $(N(t))$ with intensity measure $\Lambda$ such that $\Lambda[0,t]<\infty$ for all $t\ge 0$. 
Also, define the function $I:z\mapsto (1+z)\log(1+z)-z$. Then, for all $t>0$ and $x>0$,
\begin{equation*}\P{\sup_{u\in [0,t]}|N(u)-\Lambda[0,u]|\geq x}\,\leq\,2\exp{\bigg(-\Lambda[0,t]\cdot I\bigg(\frac{x}{\Lambda[0,t]}\bigg)\bigg)}.\end{equation*}
\end{theorem}

\subsection{Continuous embeddings}\label{sec:contembed} In this subsection we describe two continuous-time interpretations of the DWT model.

\subsubsection{First interpretation: via a continuous-time branching process}\label{sec:Poisson}
This continuous-time version of the model appears in the original paper of Leckey, Mitsche and Wormald~\cite{LeckMitWor20}. It has the advantage that the discrete model can be completely recovered from it (which is not the case with the second interpretation presented in \Cref{sec:Cox}).

A central object in this construction is the \emph{Ulam-Harris tree} with vertex set
\begin{equation*}
\UH\coloneq\{\varnothing\}\cup \bigg(\bigcup_{n\in\N} \mathbb N^n\bigg).
\end{equation*} 
This tree grows away from a root $\varnothing$ and the children of each vertex $v\in \UH$ are $v1,v2,\ldots$.
Moreover, we denote by $\duh\colon\UH\to\N_0$ the depth function for the Ulam-Harris tree; namely, $\duh(\varnothing)=0$ and, for every vertex $v=v_1v_2\cdots v_k$ with $k,v_1,\ldots,v_k\in \mathbb N$, $\duh(v)=k$. 
We also define a \emph{ray} to be an infinite sequence of positive integers and denote by $\mathbb N^{\infty}$ the set of rays.

To every $v\in \UH$, we then assign a Poisson process $\xi^{v}$ of rate $f(\duh(v))$ where $f$ is the weight function associated with the depth-weighted tree. 
The points $\sigma^{v}_1, \sigma^{v}_2, \ldots$ in the Poisson process $\xi^{v}$ correspond to the difference between the birth-times of the children $v1,v2,\ldots$ of $v$ and the birth-time of $v$ itself.
We can thus recursively construct the \emph{absolute} birth times $(\mathfrak{b}(v))_{v\in\UH}$ as
\begin{equation*} \label{eq:birthtime}
\mathfrak{b}(\varnothing)\coloneq0 \qquad\text{and}\qquad\ \mathfrak{b}(vj)\coloneq\mathfrak{b}(v)+\sigma^{v}_j \qquad\text{for every }v\in \UH\text{ and } j\in\N. 
\end{equation*} 
Finally, we construct the continuous-time branching process $(\cT(t))_{t\geq 0}$ by defining
\begin{equation*} 
\cT(t)\coloneq\{v\in \UH: \mathfrak{b}(v)\leq t\} \qquad \text{for all }t\geq 0.
\end{equation*} 
In particular, $d(\cT(t))$ stands for the depth of the branching process $\cT$ at time $t$. To justify the correspondence between the continuous-time branching process $\cT$ and the depth-weighted tree process, we introduce the stopping times $(\tau_n)_{n\in\N}$ where 
\begin{equation*} 
\tau_n\coloneq\inf\{t\geq 0: |\cT(t)|=n\} \qquad \text{for all }n\in\N. 
\end{equation*} 
It is then not hard to observe (and follows from~\cite[Lemma~3.1]{LeckMitWor20}) that
\be \label{eq:taun}
\{T_n: n\in\N\}\overset d=\{\cT(\tau_n): n\in\N\}. 
\ee 

\subsubsection{Second interpretation: via Cox processes}\label{sec:Cox}

Sometimes the exact structure of the tree $T_n$ is irrelevant and only the number of vertices at different depths matters. 
The continuous-time construction from \Cref{sec:Poisson} carries more information than needed in such a case, which can complicate the analysis. 
To this end, we introduce the following simplified interpretation.

\begin{itemize}
\item First, we define a Poisson point process $N_{1}(t)$ with intensity $\Lambda_{1}(\dd t)=f(0)\dd t$ (counting the vertices on depth 1).
\item Recursively, for every $r\geq 2$, we define a Cox process $N_{r}(t)$ with intensity 
\[\Lambda_r(\dd t)=f(r-1)N_{r-1}(t)\dd t\] 
(counting the vertices on depth $r$).
\end{itemize}

\noindent
For every $r\in \mathbb N$, we call $\Lambda_r$ the \emph{directing measure} of $N_r$ and additionally denote $N_0(t)=1$ for all $t\geq 0$ to account for the root on depth zero. It will be useful to further define the stopping times
\begin{equation}\label{eq:taus}
\tau_{j,r} = \inf\{t\ge 0: N_r(t) \ge j\}.
\end{equation}
Note that $\tau_{j,r}$ can be interpreted as the first time when the tree reaches $j$ vertices on depth $r$. 
Further, for every $t\ge 0$, we define 
\begin{equation*} \label{eq:Nt}
N(t) \coloneq\sum_{r=0}^\infty N_r(t)= 1+\sum_{r=1}^{\infty} N_r(t),
\end{equation*} 
interpreted as the total number of individuals at time $t$, that is, $N(t)=|\cT(t)|$. 
\cb{In particular, the stopping time $\tau_n$ has the following natural definition in terms of the Cox process:}
\begin{equation}\label{eq:taus+}
\tau_n\coloneq\inf\{t\geq 0: N(t)=n\}.
\end{equation}
The following result formalises the connection with the original model.

\begin{lemma}\label{lem:equiv}
For every weight function $f$, the distribution of the sequences $(N_r(\tau_n))_{r\in \mathbb N}$ and $(|\{v\in T_n: d_n(v) = r\}|)_{r\in \mathbb N}$ coincide.
\end{lemma}
\begin{proof}
We proceed by recursively constructing a coupling of the two sequences that identifies them almost surely.
The coupling of the roots is immediate.
Suppose that, for some $n\ge 2$, a coupling with the above description has been constructed so that \[N_r(\tau_{n-1}) = |\{v\in T_{n-1}: d(v) = r\}|\qquad \text{for every $r\in \mathbb N$,}\] 
and define $s_r = s_r(n) \coloneq N_r(\tau_{n-1})+1$.
Then, \cb{conditionally on $(N_r(t))_{r\in\N, t\in[0,\tau_{n-1}]}$}, for every $r\in \mathbb N$, $(\tau_{s_r,r} - \tau_{n-1})_{r\in \mathbb N}$ is a family of independent exponential random variables with rates $f(r-1) N_{r-1}(\tau_{n-1})$, respectively. Thus, since $N_r(\tau_n) = s_r$ holds only for the index $r$ of the minimum among $(\tau_{s_r,r} - \tau_{n-1})_{r\in \mathbb N}$,
\begin{equation*}\mathbb P(N_r(\tau_n) = s_r\,|\, \cb{(N_r(t))_{r\in\N, t\in[0,\tau_{n-1}]}}) = \frac{f(r-1) N_{r-1}(\tau_{n-1})}{\sum_{i=1}^{\infty} f(i-1) N_{i-1}(\tau_{n-1})}.\end{equation*}
However, this is precisely the probability that the $n$-th vertex attaches to a vertex on depth $r-1$ in the discrete model, which concludes the coupling until time $\tau_n$ and the proof.
\end{proof}

\subsection{(Non)-explosion of DWTs via infinite paths}

We end with a minor  result concerning the embedding of DWTs in the continuous-time branching process discussed in Section~\ref{sec:Poisson}. Define
    \be \label{eq:tauinfty}
    \tau_\infty\coloneq\lim_{n\to\infty}\tau_n , \qquad\text{and}\qquad\ \tau_{1,\infty}\coloneq\lim_{k\to\infty} \tau_{1,k} 
    \ee 
    with $\tau_{1,k}$ introduced in~\eqref{eq:taus}. While clearly $\tau_{\infty}\le \tau_{1,\infty}$ almost surely, the following lemma shows that equality holds as well; for a proof, we direct the reader to the more general \cite[Lemma 1.3]{Kom16}.
    
	\begin{lemma}\label{lemma:explray}
		Almost surely $\tau_\infty=\tau_{1,\infty}$.
	\end{lemma}

\section{Bounded and slowly-varying weight functions}\label{sec:bdd}

This section is dedicated to the proofs of \Cref{prop:bddH} and Theorem~\ref{thrm:bddH} regarding bounded weight functions, and Theorem~\ref{thrm:svf} regarding slowly-varying weight functions. 
These two settings are handled within a common framework using the continuous-time formulation described in Section~\ref{sec:Cox}.
We recall that the number of vertices at each depth $k$ is represented by a continuous-time jump process $N_k(t)$, and $\tau_{1,n}$ and $\tau_n$ correspond to the first time the tree attains depth and size $n$, respectively. 
Our proofs relate to the analysis of the $\DWT$ with constant weight function $f\equiv 1$ which, when interpreted through the continuous-time formulation described in Section~\ref{sec:Cox}, becomes a Yule process, where individuals produce children in an i.i.d.\ fashion according to a rate-one Poisson point process. 
This is a special  case of a Crump–Mode–Jagers (CMJ) branching process. Both Kingman~\cite{King75} and Biggins~\cite{Big76} study the growth rate of $\tau_{1,n}$ for such processes, and the growth rate of $\tau_n$ follows from general work of Nerman on CMJ branching processes~\cite{Ner81}.

In the case of the Yule process,  the following properties, described in our terminology, hold almost surely: 
\begin{itemize}
    \item $\log(N(t))/t\toas 1$ (exponential growth of the size, see~\cite[Section $2.5$]{Nor98}), and
    \item $\lim_{n\to\infty}\tau_{1,n}/n=1/\e$ (linear growth of the depth, see~\cite[Theorem $5$]{King75}).
\end{itemize}
Our proofs rely on extending these properties to the more general class of functions $f$ that are bounded but not constant. 
We end this preamble by giving a closed expression for $\E{N_k(t)}$, which plays a central role in our computations. Recall the convention that an empty product equals one.

\begin{lemma}\label{lemma:gensize} 
For every $k\in\N_0$ and $t\geq 0$,
\begin{equation}\label{eq:gensize}
\E{N_k(t)}=\frac{(t\ell_k)^k}{k!}
\end{equation}
with
\be \label{eq:hk}
	\ell_k\coloneq \bigg(\prod_{j=0}^{k-1}f(j)\bigg)^{1/k}.
	\ee
\end{lemma} 
\begin{proof} 
We prove the result by induction over $k$. The result holds for $k=0$ as $N_0(t)=1$ for all $t\geq0$. Assume~\eqref{eq:gensize} for some $k\in\N_0$. 
Then,
        \begin{equation*}
        \E{N_{k+1}(t)}=\E{\E{N_{k+1}(t)\,\big |\, (N_k(s))_{s\in[0,t]}}}=\E{\int_0^t f(k)  N_k( s)\,\dd s}.
        \end{equation*} 
        After switching expectation and integration by Fubini's theorem, the induction hypothesis gives 
        \be\label{eq:integrateNk}
        \E{N_{k+1}(t)} = \int_0^tf(k)\E{N_k(s)}\,\dd s=\frac{1}{k!}\prod_{i=0}^k f(i)\int_0^t s^k\,\dd s=\frac{(t\ell_{k+1})^{k+1}}{(k+1)!},  
        \ee 
        as desired. 
        \end{proof} 

\begin{remark}
    By combining Markov's inequality,~\eqref{eq:gensize} and the bound $k!\ge (k/\e)^k$ for all $k\ge 0$, we obtain the useful bound \begin{equation}\label{eq:rel_tau_N}
        \P{\tau_{1,k}\leq t}=\P{N_k(t)\geq 1}\leq \E{N_k(t)}=\frac{(t\ell_k)^k}{k!}\le \bigg(\frac{\e t\ell_k}{k}\bigg)^k.
        \end{equation}
\end{remark}

\subsection{\texorpdfstring{Weight functions bounded away from $0$ and infinity}{Weight functions bounded away from 0 and infinity}} 
In this section, we prove \Cref{prop:bddH} for functions $f$ with $0 < \inf f < \sup f < \infty$; in particular, for all $k\in \mathbb N_0$, we have 
\[0 < \cb{\inf f\le} \inf \ell_k < \sup \ell_k \cb{\le \sup f} < \infty,\]
where $\ell_k$ is defined in Lemma~\ref{lemma:gensize}.

To prove Proposition~\ref{prop:bddH}, we make use of the continuous-time formulation of the model from \Cref{sec:Poisson}. We can write the random variables $(\tau_{1,k})_{k\in\N}$, as defined in~\eqref{eq:taus}, as
\begin{equation*}
\tau_{1,k}\coloneq \inf\{t\geq 0\colon \exists v\in \cT(t)\text{ with }d_{\UH}(v)=k\},
\end{equation*}
so that the depth of the discrete-time tree $T_n$ can be expressed via $\cT(\tau_n)$   as  
\begin{equation*}
d(T_n)\overset{\dd}{=}d(\cT(\tau_n))=\sup\{k\in\N\colon \tau_{1,k}\leq \tau_n\}.
\end{equation*}
Set $I\coloneq \inf f$ and $S\coloneq \sup f$, and let $(\cY_\lambda(t))_{t\geq 0}$ denote a Yule branching process of rate $\lambda>0$ (i.e.\ a branching process $\cT$ with $f(i)=\lambda$ for all $i\in\N_0$). We can couple the original branching process $\cT$ to $\cY_I$ and $\cY_S$ such that $\cY_I(t)\subseteq \cT(t)\subseteq \cY_S(t)$ for all $t\geq 0$ almost surely. Letting $\tau_{1,k}^\lambda$ and $\tau_n^\lambda$ denote the corresponding stopping times in the Yule process $\cY_\lambda$, equivalent to the definition of $\tau_{1,k}$ and $\tau_n$, respectively, we thus obtain the inequalities
\begin{equation*} 
\tau_{1,k}^S\leq \tau_{1,k}\leq \tau_{1,k}^I\qquad\text{and}\qquad \tau_n^S\leq \tau_n\leq \tau_n^I, 
\end{equation*}
almost surely, so that under the coupling, almost surely
\be \label{eq:supbounds}
\sup\{k\in\N\colon \tau_{1,k}^I\leq \tau_n^S\}\leq d(\cT(\tau_n))\leq \sup\{k\in\N\colon \tau_{1,k}^S\leq \tau_n^I\}.
\ee 
Using Nerman~\cite[Theorem 5.4]{Ner81}, it follows that, as $t\to\infty$, 
\be 
\e^{-\lambda t}|\cY_\lambda (t)|\toas E, 
\ee 
where $E\sim \text{Exp}(1)$. As $|\cY_\lambda(\tau_n^\lambda)|=n$ by definition ($\tau_n^\lambda$ is the (stopping) time at which $\cY_\lambda$ contains $n$ individuals), it thus follows that 
\be \label{eq:taunlambdaconv}
\frac{\tau_n^\lambda}{\log n}\toas \frac{1}{\lambda},
\ee 
as $n\to\infty$. Furthermore, from Kingman~\cite[Theorem 5]{King75}, we conclude that 
\be \label{eq:tau1klambdaconv}
\frac{ \tau_{1,k}^{\lambda} }{k}\toas \frac{1}{\lambda \e},
\ee 
as $k\to\infty$. Indeed, Kingman derives the almost sure scaling limit of $\tau_{1,k}^\lambda/k$ to some constant $\gamma$ (the Yule process is  a special case within the more general set-up of Crump--Mode--Jagers branching processes that Kingman studies). The fact that $\gamma=1/(\lambda\e)$ can be derived by applying (2.2), (3.2), and (3.3) in~\cite{King75} to the special case of a rate $\lambda$ Yule process. By applying~\eqref{eq:taunlambdaconv} and~\eqref{eq:tau1klambdaconv} for $\lambda\in\{I,S\}$, the two suprema in~\eqref{eq:supbounds} can be bounded from below and above by $\ell \log n$  and $L\log n$ for all $n$ sufficiently large almost surely, for any constants $\ell\in(0,I\e/S)$ and $L>S\e/I$.

We now focus on the proof of Theorem~\ref{thrm:bddH}. 
We begin with three simple but important observations. First, for convergent functions $f$ bounded away from 0, $(\ell_k)_{k\ge 1}$ converges to the limit $\ell \coloneq \lim_{k\to \infty} f(k)$.
Second, as rescaling the weight function by a constant factor does not modify the distribution of the resulting depth-weighted tree, we henceforth assume without loss of generality that $\ell=1$. 
Third, for the case of $r-$periodic functions $f$, the same rescaling argument allows us to assume without loss of generality that $\prod_{i=0}^{r-1}f(i)=1$.

\vspace{0.5em}

We next show that, in the setting of converging weight functions, the quantity $\tau_{1,k}/k$ converges almost surely, analogous to~\eqref{eq:tau1klambdaconv}.
Our lower bound on $\tau_{1,k}/k$ comes from a first moment argument while the upper bound is more involved and adapts ideas of Kingman~\cite{King75}.

\begin{lemma}\label{lem:taukconv}
For every weight function $f(k)\to 1$ as $k\to\infty$,
\begin{equation}\label{eq:convtau1k}
\frac{\tau_{1,k}}{k}\toas \frac{1}{\e}.
\end{equation} 
\end{lemma}

\begin{remark}\label{rem:biggins}
For periodic weight functions, \eqref{eq:convtau1k} is implied by Theorem~2 in~\cite{Big76}. 
This more general result concerns Crump--Mode--Jagers branching processes satisfying the following assumptions:
\begin{itemize}
    \item supercriticality: with positive probability, $\tau_{1,k}<\infty$ for every $k\in \N$,
    \item finite number of types: the type of a vertex $v$ is given by the depth of $v$ modulo $r$
    \item irreducibility: for every pair of types $i,j\in [r]$, a vertex of type $i$ has a descendant of type $j$ with positive probability.
\end{itemize}
Under these assumptions (which are clearly satisfied in our case), Theorem~2 in~\cite{Big76} shows that the sequence $(\tau_{1,k}/k)_k$ almost surely converges to a limit $\gamma$.
The fact that $\gamma=1/\e$ in our case requires a combination of displays (2.5), (3.2), the equality for $\gamma_k/k$ after display (3.3) and Lemma~3.1 in~\cite{Big76}.
\end{remark}

\begin{proof}[Proof of~\Cref{lem:taukconv}]
For the lower bound on $\tau_{1,k}/k$, fix $k\in\N$, $x\in(0,1)$. By~\eqref{eq:rel_tau_N}, 
\begin{equation*} 
\P{\tau_{1,k}\leq xk/\e}\leq (x\ell_k)^k. 
\end{equation*} 
Since we assume that $f(k)$ converges to $\ell=1$, it follows that $\ell_k$ converges to $1$ as well. As a result, there exists $\wt x\in(x,1)$ such that $x\ell_k<\wt x<1$ for all large $k$. Then, by Borel-Cantelli's lemma, almost surely
\begin{equation*}\label{lowerboundconst}
\liminf_{k\to\infty} \frac{\tau_{1,k}}{k}\geq \frac{x}{\e}.
\end{equation*}
As $x$ is an arbitrary real number in $(0,1)$, the lower bound follows.

We turn to the upper bound on $\tau_{1,k}/k$. Fix $k\in\N$ and call $v_{1,k}$ the first vertex on depth $k$ born at time $\beta_{1,k}\coloneq\tau_{1,k}$. 
Further, define $v_{2,k}$ as the first descendant of $v_{1,k}$ on depth $2k$ (so the distance between $v_{1,k}$ and $v_{2,k}$ is $k$), and write $\beta_{2,k}$ for the difference between the birth time of $v_{2,k}$ and $\beta_{1,k}$. 
Similarly, we define a sequence of independent random variables $\{\beta_{m,k}\}_{m \ge 1}$, see \Cref{fig:placeholder}.
    \begin{figure}[h!]
        \centering
        \includegraphics[width=0.7\linewidth]{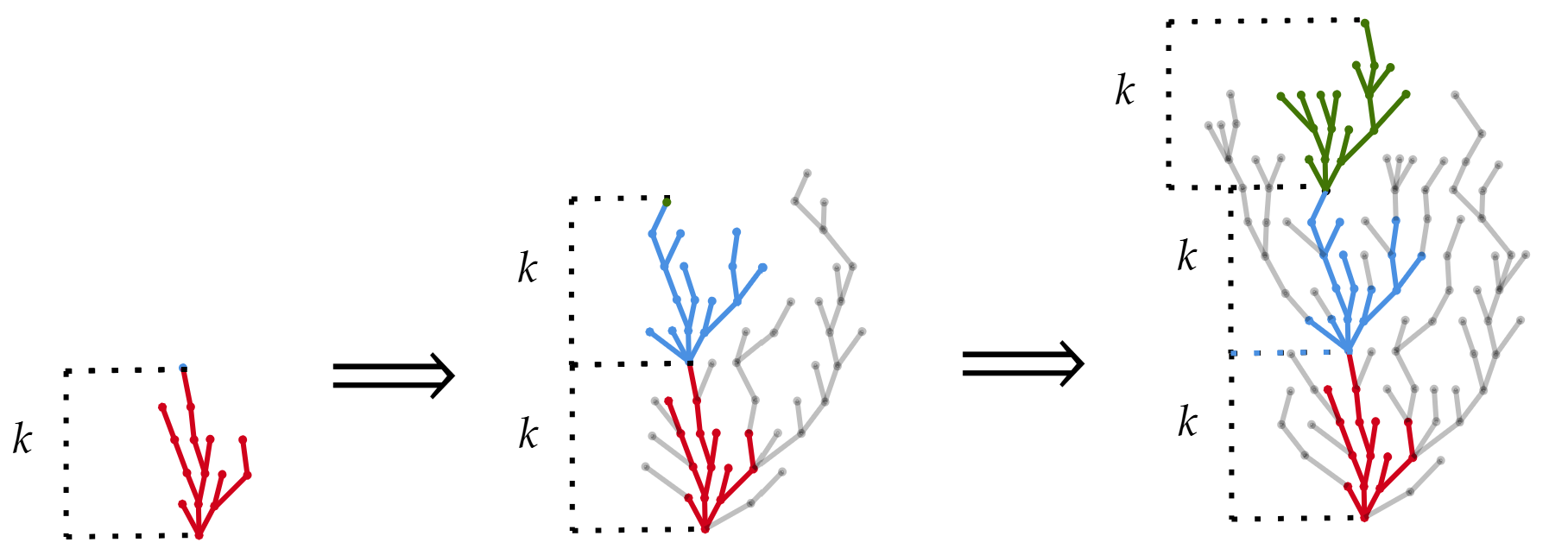}
        \caption{The construction of the times $\beta_{m,k}$. 
        In the \cb{left-most picture,}
        the process runs for $\beta_{1,k}$ units of time, at which point the \cb{red} tree 
        reaches depth $k$. 
        In the \cb{central picture,}
        we let the whole tree grow for $\cb{\beta_{1,k}+\beta_{2,k}}$ units of time, at which point the \cb{blue} subtree 
        reaches depth $k$. Notice that, by time $\cb{\beta_{1,k}+\beta_{2,k}}$, the whole tree may have already reached a greater depth.}
        \label{fig:placeholder}
    \end{figure}

Next, for every $k\in \N$, denote by $\overline\tau_{1,k}$ the birth time of the first individual on depth $k$ \cb{in the depth-weighted tree with weight function $f\equiv 1$.}

\begin{claim}\label{cl:SLLN}
For every $k\in \N$, the sequence $(\beta_{m,k})_{m\ge 1}$ converges in distribution to $\overline\tau_{1,k}$. Moreover, for every $k\in \N$, $\sup_{m\ge 1} \mathbb E[\beta_{m,k}^4] < \infty$.
\end{claim}
\begin{proof}[Proof of \Cref{cl:SLLN}]
For the first point, for every $m\in \N$ such that 
\begin{equation}\label{eq:ineq3.5}
1-\eps\le \inf\{f(n): n\ge mk\}\le \sup\{f(n):n\ge mk\}\le 1+\eps,
\end{equation} 
observe that \cb{$(1-\eps)^{-1} \overline\tau_{1,k}$ stochastically dominates $\beta_{m,k}$, and $\beta_{m,k}$ stochastically dominates $(1+\eps)^{-1} \overline\tau_{1,k}$.}
Since $f$ converges to 1, \eqref{eq:ineq3.5} holds for every $\eps > 0$ and sufficiently large $m$ and the first statement follows.
For the second statement, it suffices to observe that, for every $m\ge 1$, $\beta_{m,k}$ is stochastically dominated by the sum of $k$ independent exponential random variables with rate $\inf f$.
\end{proof}

By combining \Cref{cl:SLLN} with the strong law of large numbers, and observing that the tree might have reached depth $mk$ before time $\beta_{1,k}+\beta_{2,k}+\cdots+\beta_{m,k}$, we obtain that almost surely
\begin{equation*}\limsup_{m\to\infty} \frac{\tau_{1,mk}}{m}\le \limsup_{m\to\infty} \frac{\beta_{1,k}+\ldots+\beta_{m,k}}{m} = \mathbb E[\overline\tau_{1,k}].\end{equation*}
Moreover, using that $\lim_{k\to\infty} \mathbb E[\overline\tau_{1,k}/k]= \e^{-1}$ by~\cite[Theorems~4 and~5]{King75}, almost surely
\begin{equation}\label{eq:partial_conv}
\limsup_{k\to\infty} \limsup_{m\to\infty} \frac{\tau_{1,mk}}{mk}\le \frac{1}{\e}.
\end{equation}
To deduce the desired upper bound assuming \eqref{eq:partial_conv}, fix any $\eps>0$ and large enough $k$ such that $\limsup_{m\to\infty} \tau_{1,mk}/mk\le \e^{-1}+\eps$.
Then, for every large enough $n$ (in terms of $\eps$ and $k$), write $n=mk+r$ for integers $m$ and $r\in [0,k-1]$ and note that
\begin{equation*}\frac{\tau_{1,n}}{n}\le \frac{\tau_{1,(m+1)k}}{mk}\le \frac{m+1}{m} \bigg(\frac{1}{\e}+\eps\bigg)\le \frac{1}{\e}+2\eps\qquad\Rightarrow\qquad \limsup_{n\to\infty} \frac{\tau_{1,n}}{n}\le \frac{1}{\e},\end{equation*}
which finishes the proof.
\end{proof}

The second step in the proof of \Cref{thrm:bddH} consists in estimating the time $\tau_n$ when the branching process reaches $n$ vertices. 
The necessary control on $\tau_n$ is given by the next propositions.

\begin{proposition}\label{prop:taun}
For every weight function $f(k)\to 1$ as $k\to\infty$,
    \begin{equation*} 
    \frac{\tau_n}{\log n}\toas 1.
    \end{equation*} 
\end{proposition}

\begin{proposition}\label{prop:taun_periodic}
For every integer $r\in \N$ and every $r$-periodic weight function $f$ with $f(0)\cdots f(r-1)=1$,
    \begin{equation*} 
    \frac{\tau_n}{\log n}\toas 1.
    \end{equation*} 
\end{proposition}
    
Before proving \Cref{prop:taun,prop:taun_periodic}, we first complete the proof of \Cref{thrm:bddH}. 

\begin{proof}[Proof of Theorem~\ref{thrm:bddH} assuming \Cref{prop:taun,prop:taun_periodic}]
By \Cref{lem:taukconv} and \Cref{prop:taun,prop:taun_periodic}, almost surely, for every $\eps > 0$ and sufficiently large $n$, we have that 
\begin{equation*}\frac{\tau_{1,(\e-2\eps)\log n}}{\log n}\le \frac{\e-\eps}{\e}\le \frac{\tau_n}{\log n} \le \frac{\e+\eps}{\e}\le \frac{\tau_{1,(\e+2\eps)\log n}}{\log n},\end{equation*}
implying that $\tau_n\in [\tau_{1,(\e-2\eps)\log n}, \tau_{1,(\e+2\eps)\log n}]$.
Since $\cT(\tau_n)$ and $T_n$ have the same distribution by~\eqref{eq:taun}, \Cref{thrm:bddH} follows.
\end{proof} 

We finish this section by proving \Cref{prop:taun,prop:taun_periodic}.

\begin{proof}[Proof of \Cref{prop:taun}] 
To prove the result, it is enough to show that
        \be\label{eq:conv_N}
        \frac{\log(N(t))}{t}\toas 1. 
        \ee 
    Indeed, as $\tau_n\to\infty$ with $n$, and $N(\tau_n)=n$ for every $n\in \N$, \eqref{eq:conv_N} implies the desired result. 

\vspace{0.5em}
\noindent
\textbf{Upper bound.}
Recall that as $f(k)\to 1$, we also have $\ell_k\to 1$ with $\ell_k$ defined in~\eqref{eq:hk}. Then, for every $\eps > 0$, $\ell_k\le 1+\eps$ for all but finitely many $k$.
Combining this with Lemma~\ref{lemma:gensize}, we deduce that, for every suitably large $t$, we have that 
\begin{equation}\label{eq:expN-estimate}
\mathbb E[N(t)]=\sum_{k=0}^{\infty} \frac{(t\ell_k)^k}{k!}\le t^{\max\{k: \ell_k > 1+\eps\}+1}+\sum_{k=0}^{\infty} \frac{((1+\eps)t)^k}{k!}\le \e^{(1+2\eps)t}.
\end{equation}
As a consequence, for every such $t$, Markov's inequality implies that
\begin{equation*}
\P{\frac{\log(N(t))}{t}\geq 1+3\eps} = \mathbb P(N(t)\ge \e^{(1+3\eps)t})\leq \e^{-\eps t}.
\end{equation*}
By the monotonicity of $(N(t))_{t\ge 0}$, a union bound and the \cb{previous} inequality, for every large enough~$t$,
\begin{align}
\P{\sup_{u\ge t}\frac{\log(N(u))}{u}\ge 1+4\eps}
&= \P{\sup_{k\ge 1}\sup_{u\in [t+\eps (k-1), t+\eps k]}\frac{\log(N(u))}{u}\ge 1+4\eps}\nonumber\\
&\le \P{\sup_{k\ge 1}\frac{\log(N(t+\eps k))}{t+\eps(k-1)}\ge 1+4\eps}\label{eq:LBN_l2}\\
&\le \P{\sup_{k\ge 1}\frac{\log(N(t+\eps k))}{t+\eps k}\ge 1+3\eps}\nonumber\\
&\le \sum_{k=1}^{\infty} \e^{-\eps t - \eps^2 k} \le \frac{\exp(-\eps t)}{1-\exp(-\eps^2)}.\nonumber
\end{align}
Thus, by Borel-Cantelli's lemma for the events in \eqref{eq:LBN_l2} with $t\in \N$, for every $\eps > 0$, almost surely
\begin{equation}\label{eq:UBtau_n1}
\limsup_{t\to\infty}\frac{\log(N(t))}{t} \leq 1+4\eps.
\end{equation}

\vspace{0.5em}
\noindent
\textbf{Lower bound.} To obtain lower bounds for $N(t)$ with large $t$, our approach consists in finding a depth $r_0$ such that, from that depth onwards, the weight function $f$ is sufficiently close to $1$. 
We then follow the growth of $(N(s))_{s\in [0,\eps t]}$ for some small $\eps>0$ and show that, by time $\varepsilon t$, there is a large number of vertices on depth $r_0$. 
From this time onwards, we use concentration bounds to show that the evolution of the tree (from depth $r_0$ onwards) is close to that of a tree with weight function $f \equiv 1$. To formalise this idea, fix $\varepsilon \in (0,1/2)$ small and let $r_0 = r_0(\eps)$ be the smallest integer in the interval $[8,\infty)$ such that $f(r)\ge 1-\eps/4$ for every $r\ge r_0$. We then choose $t>0$ large enough so that
\begin{equation}\label{eq:tlargelowerbound}\frac{\varepsilon t\min_{k\leq r_0}f(k)}{r_0}\geq 2t^{1-\eps}.\end{equation}
We split the interval $[0,t]$ into $[0,\varepsilon t]$ and $[\varepsilon t,t]$ and analyse the process on these time intervals separately.

\vspace{0.5em}

For the growth of $N$ throughout $[0,\varepsilon t]$, define a sequence of times of the form $s_i=i\varepsilon t/r_0$ for $i\in \{0,\ldots,r_0\}$. Notice that $s_{r_0}=\varepsilon t$ and that, by our assumption on $t$, $s_{i}-s_{i-1}$ is large. 
Then, we define the events $\cD_i\coloneq\{N_i(s_i)\ge t^{(1-\eps)i}\}$; note that, at each time $s_i$, the event $\cD_i$ corresponds to a different depth $i$. 
Since $N_1(s_1)$ is a Poisson random variable with mean $\eps tf(0)/r_0\geq 2t^{1-\varepsilon}$ by \eqref{eq:tlargelowerbound}, Chernoff's bound for Poisson random variables (\Cref{lem:chernoff}) yields
\begin{equation*}\P{\cD_1^{c}}\leq \P{N_1(s_1)\leq \E{N_1(s_1)}/2}\leq \exp\left(-\frac{\eps t\min_{k\leq r_0}f(k)}{8r_0}\right).\end{equation*}
Furthermore, on the event $\mathcal{D}_j$, there are at least $t^{(1-\eps)j}$ vertices on depth $j$ at time $s_j$. 
The offspring produced by these vertices \cb{in the interval} $[s_j,s_{j+1}]$ is a Poisson random variable with mean $(\eps tf(j)/r_0)t^{(1-\varepsilon)j}\geq 2t^{(1-\varepsilon)(j+1)}$. 
Thus, by repeating the \cb{former} argument, we obtain
\begin{equation*}\condP{\cD_{j+1}^{c}}{\cD_j}\leq \exp\left(-\frac{\eps t\min_{k\leq r_0}f(k)}{8r_0}\right),\end{equation*}
and hence conclude that
\begin{equation}\label{eq:firsthalflbound}\mathbb P(\cD_{r_0}^c)\leq \bigg(\sum_{i=1}^{r_0-1} \mathbb P(\cD_{i+1}^c\mid \cD_i)\bigg)+\mathbb P(\cD_1^c)\le r_0\exp\bigg(-\frac{\eps t \min_{k\leq r_0}f(k)}{8r_0}\bigg).
\end{equation}
\cb{On the event $\cD_{r_0}$, there are at least $n_t\coloneq t^{(1-\eps)r_0}$ many vertices at depth $r_0$ at time $s_{r_0}=\eps t$. Select $n_t$ of these vertices in an arbitrary manner and label them $v_1,\ldots,v_{n_t}$. For each $i\in[n_t]$, let $Y_i(t)$ denote the size of the subtree grown from $v_i$ in the time interval $[\eps t,t]$ (that is, we ignore children and their descendants of $v_i$ that were born before time $\eps t$). 
Since $f(r)\geq 1-\eps/4$ for all $r\geq r_0$, it follows that $(Y_i(t))_{i\in [n_t]}$ stochastically dominate the sizes of independent Yule processes with rates $1-\eps/4$ grown within a time interval of length $(1-\eps)t$. 
We denote these sizes by $(G_i(t))_{i\in [n_t]}$. Then, 
\begin{equation*}
 \mathbb P\bigg(\bigg\{\sum_{i=r_0}^\infty N_i(t)\leq \e^{(1-2\eps)t}\bigg\}\cap \cD_{r_0}\!\bigg)
\leq \mathbb P\bigg(\sum_{i=1}^{n_t}Y_i(t)\leq \e^{(1-2\eps)t}\bigg)\leq \mathbb P\bigg(\sum_{i=1}^{n_t}G_i(t)\leq \e^{(1-2\eps)t}\bigg).
\end{equation*} 
As each $G_i(t)$ has a geometric distribution with parameter $p\coloneq \exp(-(1-\eps/4)(1-\eps)t)$, see \cite[(4.3.4)]{vdHof17}, we obtain via a union bound that 
\begin{equation*}\ba 
\mathbb P\bigg(\sum_{i=1}^{n_t}G_i(t)\leq \e^{(1-2\eps)t}\bigg)&\leq n_t\mathbb P\bigg(G_1(t)\leq \frac{1}{n_t}\e^{(1-2\eps)t}\bigg)=n_t\Big(1-(1-p)^{\e^{(1-2\eps)t}/n_t}\Big)\\
&\leq p\e^{(1-2\eps)t}=\e^{-(3\eps/4+\eps^2/4)t}.
\ea\end{equation*} 
By taking $\eps>0$ sufficiently small, this upper bound is at most $\e^{-\eps t/2}$. Combined with~\eqref{eq:firsthalflbound}, we thus arrive at 
\begin{equation*}\ba 
\P{\frac{\log(N(t))}{t}\leq 1-2\eps}&\leq \P{\cD_{r_0}^c}+\mathbb P\bigg(\bigg\{\sum_{i=r_0}^\infty N_i(t)\leq \e^{(1-2\eps)t}\bigg\}\cap \cD_{r_0}\!\bigg)\\
&\leq r_0\exp\bigg(-\frac{\eps t \min_{k\leq r_0}f(k)}{8r_0}\bigg)+\e^{-\eps t/2}\leq \e^{-ct} 
\ea\end{equation*}
for some small constant $c>0$.
}

A similar argument to \eqref{eq:LBN_l2} allows to strengthen this bound to 
\[\mathbb P\big(\inf_{u\geq t}\; \log(N(u))/u\le 1-3\eps\big)=\e^{-\eps't}\qquad \text{for some $\eps'=\eps'(\eps,f) > 0$.}\] 
By Borel-Cantelli's lemma for $t\in \N$, for every sufficiently small $\eps\in (0,1/2)$, almost surely
\begin{equation}\label{eq:LBtau_n}
\liminf_{t\to\infty}\frac{\log(N(t))}{t} \geq 1-3\eps.
\end{equation}
The fact that each of~\eqref{eq:UBtau_n1} and~\eqref{eq:LBtau_n} holds for every $\eps > 0$ completes the proof.
\end{proof}

\begin{proof}[Proof of \Cref{prop:taun_periodic}]
Recall our assumption that $f$ is an $r$-periodic function with $\prod_{i=0}^{r-1}f(i)=1$ and extend $f$ periodically to the entire $\mathbb Z$. 
Our proof of \Cref{prop:taun_periodic} is very similar to that of \Cref{prop:taun}: we show that $\log (N(t))/t \toas 1$ by obtaining upper and lower bounds. 
For the upper bound, the argument remains unchanged except that~\eqref{eq:expN-estimate} is replaced by
\begin{equation*}\mathbb E[N(t)]=\sum_{k=0}^{\infty} \frac{(t\ell_k)^k}{k!}\le \bigg(\max_{i\in [r]} \prod_{j=0}^{i-1} f(j)\bigg)\sum_{k=0}^{\infty} \frac{t^k}{k!} = \bigg(\max_{i\in [r]} \prod_{j=0}^{i-1} f(j)\bigg)\e^t.\end{equation*}
For the lower bound, we again analyse separately the growth of $N$ on two intervals, $[0,\varepsilon t]$ and $[\varepsilon t, t]$, where $\varepsilon\in (0,1/2)$ is a fixed small value and $t > 0$ is large enough to satisfy \eqref{eq:tlargelowerbound} with $r+8$ in place of $r_0$. 
For the growth of the process on $[0,\eps t]$ and for every $i\in \{0,\ldots,r+8\}$, define the time $s_i = \frac{i \varepsilon t}{r+8}$ and the event $\mathcal{D}_i = \{ N_i(s_i) \ge t^{(1-\varepsilon)i} \}$ as in the previous proof. 
Observe that, for each $i \in [r+8]$, we have $N_i(s_i) \le N_i(\eps t)$ and that, for each $i \ge 8$, we have $t^{(1-\varepsilon)i} \ge t^4$. Setting $c_1 = (\min_{k \ge 0} f(k)) / (8r+64)$ (which neither depends on $\eps$ nor on $t$), 
we can repeat verbatim the argument from the previous proof to dominate each $\P{\cD_i^c}$ as in \eqref{eq:firsthalflbound}, and thus obtain
\begin{equation}\label{eq:prob_periodic}
\P{\min_{i\in [9,r+8]} N_i(\eps t)\ge  t^4}\ge 1-(r+8)\e^{-c_1\eps t}.
\end{equation}
Next, we study the growth of $N$ over the interval $[\varepsilon t, t]$ by again considering the sequence of times $s_i' = \varepsilon t + i\alpha$ where $0 \le i \le L \coloneq \lceil (1 - \varepsilon)t / \varepsilon^2 \rceil$ and $\alpha \coloneq (1 - \varepsilon)t / L \approx \varepsilon^2$. For every $j \in [r]$, let $\widehat{N}_j(u)$ denote the number of vertices on depths congruent to $j$ modulo $r$ at time $u$, that is,
\begin{equation*}
\widehat{N}_j(u) = \sum_{k \in \mathbb{N}_0} N_{j + kr}(u).
\end{equation*}
Using this notation, for every $i\in \{0,\ldots,L\}$ and $j\in [r]$, we define the events
\begin{equation*}
\mathcal{E}_{i,j} = \bigg\{ \widehat{N}_j(s_i') \ge Q_{i,j} t^4 \bigg\}
\qquad \text{and} \qquad
\mathcal{E}_i = \bigcap_{j=1}^r \mathcal{E}_{i,j}
\end{equation*} 
where the quantities $Q_{i,j}$ are given by
\begin{equation*}
Q_{i,j} \coloneq \sum_{k=0}^{i} \binom{i}{k} ((1 - \varepsilon)\alpha)^k \prod_{\ell=0}^{k-1} f(j - 1 - \ell)
\end{equation*}
(recall the convention that empty products are equal to $1$, \cb{and that we have extended $f$ periodically to $\Z$}). 
Three key observations about these quantities are:
\begin{itemize}
    \item For all $i\in \N$, the identity $\binom{i-1}{k-1}+\binom{i-1}{k}=\binom{i}{k}$ and a simple change of variable give
    \begin{equation}\label{eq:recQ}
    Q_{i,j}=Q_{i-1,j}+(1-\varepsilon)\alpha f(j-1)Q_{i-1,j-1}.
    \end{equation}
    \item For all $j\in [r]$, we have $Q_{0,j}= 1$. Combined with \eqref{eq:recQ}, this yields $Q_{i,j}\geq 1$ for all $i\geq0$.
    \item Since $\prod_{i=0}^{r-1}f(i)=1$, there is some $\hat{c}>0$ such that $\prod_{\ell=0}^{k-1} f(j - 1 - \ell)\geq\hat{c}$ for all $k\geq 0$ and $j\in [r]$. Consequently, for all $i\geq 0$ and $j\in[r]$,
    \begin{equation}\label{eq:lowerQ}
    Q_{i,j} \geq \hat{c}\sum_{k=0}^{i} \binom{i}{k} ((1 - \varepsilon)\alpha)^k =\hat{c}(1+(1-\eps)\alpha)^i\geq \hat{c}e^{(1-2\eps)\alpha i}
    \end{equation}
    provided $\eps$ (and thus also $\alpha$) is sufficiently small. 
    In particular, using the definition of $\alpha$ and $L$, we obtain $Q_{L,j}\geq \hat{c}\e^{(1-2\eps)(1-\eps)t}.$
\end{itemize}
As in the previous proof, we observe that the event in \eqref{eq:prob_periodic} already implies $\cE_0$, and our strategy again focusses on bounding $\condP{\cE_{i+1}^c}{\cE_i}$. 
To do so, fix some $j\in [r]$ and observe that, on the event $\cE_{i}$, there are at least $Q_{i,j-1}t^4$ vertices born by time $s_i'$ on levels congruent to $j-1$ modulo $r$. 
Thus, on the event $\cE_i$, the number of children of such vertices throughout the interval $[s'_i,s'_{i+1}]$ stochastically dominates a Poisson random variable $\cO_{i,j-1}$ with mean $\alpha f(j-1)Q_{i,j-1}t^4$. 
Observing that $\widehat{N}_j(s'_{i+1})\geq \widehat{N}_j(s'_i)+\cO_{i,j-1}$ and using~\eqref{eq:recQ} and Chernoff's bound (\Cref{lem:chernoff}), we obtain
\begin{align*}
    \mathbb P(&\widehat{N}_j(s'_{i+1})<Q_{i+1,j}t^4\mid \cE_{i})\\
    &\leq\condP{\widehat{N}_j(s'_i)+\cO_{i,j-1}<Q_{i,j}t^4+(1-\eps)\alpha f(j-1)Q_{i,j-1}t^4}{\cE_{i}}\\
    &\leq\condP{\cO_{i,j-1}<(1-\eps)\alpha f(j-1)Q_{i,j-1}t^4}{\cE_{i}}\leq \e^{-\eps^2\alpha f(j-1)Q_{i,j-1}t^4/8}\leq \e^{-\eps^5 t^4}.
\end{align*}
It follows as before that
\begin{equation*}\P{\cE^c_L}\leq \P{\cE_0^c}+ \sum_{i=1}^L\sum_{j\in [r]}\condP{\cE_{i,j}^c}{\cE_{i-1}}\leq (r+8)\e^{-c_1\eps t}+rL\e^{-\eps^5t^4}.\end{equation*}
From our observation~\eqref{eq:lowerQ} applied to $i=L$, we observe that $\cE_L$ implies
\begin{equation*}N(t)=N(s_L')\geq r\hat{c}\e^{(1-2\eps)(1-\eps)t}t^4\geq \e^{(1-2\eps)^2t}\end{equation*}
provided $t$ is sufficiently large. This implies that, for every $\eps\in (0,1/2)$ small,
\begin{equation*}\P{\frac{\log(N(t))}{t}\le (1-2\eps)^2} = \mathbb P(N(t)\le \e^{(1-2\eps)^2t})\le (r+8)\e^{-c_1\eps t}+rL\e^{-\eps^5t^4},\end{equation*}
which suffices to conclude the proof of the lower bound as in the proof of \Cref{prop:taun}.
\end{proof}

\cb{\textbf{Remark.}  We observe that our analysis of $\tau_{1,k}$ from Lemma~\ref{lem:taukconv} does not go through for general bounded functions $f$ due to the use of the strong law of large numbers, which requires certain asymptotic regularity. 
The proofs of Propositions~\ref{prop:taun} and~\ref{prop:taun_periodic} are more robust in this regard, and the convergence/periodicity hypotheses mostly serve for convenience. 
Unfortunately, our absence of a natural analogue of Lemma~\ref{lem:taukconv} and Remark~\ref{rem:biggins} does not allow us to make any general conclusion concerning the growth rate of the depth.

After having submitted this paper and before preparing the final revision of this paper, an answer to Question 1.5 has appeared in the recent work of Kammerer~\cite{Kam26}, however. The author shows in Theorem~1.1 of~\cite{Kam26} that weight functions $f$ with $c \le f(k) < 1/c$ for any $k$ and some $c > 0$, and that additionally satisfy $\lim_{k\to\infty}\log(\ell_k)=0$ (where we recall $\ell_k$ from~\eqref{eq:hk}) yield $d(T_n)/\log n\toas \e$; in fact, the author shows more: he gives the scaling limit of the depth profile for such trees. To prove his main result, the author makes use of martingale arguments based on a suitably defined Laplace transform of the depth profile combined with general Edgeworth expansion for   profiles of trees, thereby circumventing the difficulties we encountered in generalising Propositions~\ref{prop:taun} and~\ref{prop:taun_periodic}. Moreover, the author shows in Proposition 1.4 of the same paper~\cite{Kam26} that there also exist bounded weight functions having only two different values of $f$ for different depths for which $d(T_n)/\log n$ does not converge.}

\subsection{Slowly growing weight functions} 
We now turn to the proof of~\Cref{thrm:svf} for weight functions satisfying~\eqref{eq:widesv}. In the light of Remark~\ref{rem:lowerslow}, it suffices to show that
\begin{equation*}\label{eq:limsupslow}
  \limsup_{n\to\infty}\frac{d(T_n)}{\log n}\leq\e  
\end{equation*}
almost surely. As in the case of bounded weight functions, our strategy consists in controlling the growth of $\tau_{1,k}$ and $\tau_n$, from which we draw the desired conclusion.

\begin{lemma}\label{lem:Bksvf}
For every weight function $f$ satisfying~\eqref{eq:widesv}, almost surely
\begin{equation*} 
\liminf_{k\to\infty}\frac{\tau_{1,k}}{k/ f(k)}\geq \frac{1}{\e}.
\end{equation*} 
\end{lemma}
\begin{proof} 
Fix $\eps > 0$ and $t_k \coloneq (1-\eps)k/(\e f(k))$. 
\cb{Noting that $\ell_k\le f(k)$ for all $k\in \mathbb N_0$ since $f$ is non-decreasing and using~\eqref{eq:rel_tau_N},} we obtain
\begin{equation*}\label{eq:liminf_grow}
\P{\tau_{1,k}\leq t_k}\leq \Big(\frac{\e t_k\ell_k}{k}\Big)^k\leq (1-\eps)^k. 
\end{equation*} 
Hence, by Borel-Cantelli's lemma, almost surely $\liminf_{k\to\infty} \tau_{1,k}/(k/f(k))\geq(1-\eps)/\e$. As the latter holds for every $\eps > 0$, the desired conclusion follows.
\end{proof} 

Next, we state a one-sided analogue of \Cref{prop:taun,prop:taun_periodic} regarding the growth rate of $\tau_n$ for slowly growing weight functions. 
For the purposes of the following proposition, we extend $f$ to the interval $[0,\infty)$ by interpolating linearly between consecutive integers. 
We slightly abuse notation and keep denoting the extension $f$.

\begin{proposition}\label{prop:taunsv}
For every weight function $f$ satisfying~\eqref{eq:widesv}, 
\begin{equation*}\limsup_{n\to\infty}\frac{\tau_n f(\tau_n)}{\log n}\le 1.\end{equation*}
\end{proposition}

We postpone the proof of \Cref{prop:taunsv} and first derive the upper bound in Theorem~\ref{thrm:svf}. 
    
\begin{proof}[Proof of Theorem~\ref{thrm:svf} assuming \Cref{prop:taunsv}]
The lower bound of Theorem~\ref{thrm:svf} already follows from Remark~\ref{rem:lowerslow}. 
We focus on the upper bound. Fix $\eps>0$ small. By~\eqref{eq:widesv}, for all sufficiently large $n$, we have
\begin{equation*}\frac{1+\eps/2}{1+\eps}\le \frac{1}{f(\log n)} f\bigg(\frac{\log n}{f(\log n)}\bigg).\end{equation*}
By rearranging this inequality and using the monotonicity of $f$, we obtain
\begin{equation*}(1+\eps/2)\log n\le \frac{(1+\eps)\log n}{f(\log n)}f\bigg(\frac{(1+\eps)\log n}{f(\log n)}\bigg).\end{equation*}
Now, from~\Cref{prop:taunsv} we also know that, for all $n$ large, $\tau_n f(\tau_n)\leq (1+\eps/2)\log n$, and thus
\begin{equation*}
\tau_n f(\tau_n) \le \frac{(1+\eps)\log n}{f(\log n)}f\bigg(\frac{(1+\eps)\log n}{f(\log n)}\bigg).
\end{equation*}
Using that the function $xf(x)$ is strictly increasing, we conclude that, for all such $n$, we have 
\begin{equation*}\tau_n\leq \frac{(1+\eps)\log n}{f(\log n)}.\end{equation*}
On the other hand, taking $k=(1+3\eps)\e\log n$ in~\Cref{lem:Bksvf} gives \begin{equation*}\tau_{1,(1+3\eps)\e\log n}\geq\frac{1+2\eps}{1+3\eps}\frac{1}{\e}\frac{(1+3\eps)\e\log n}{f((1+3\eps)\e\log n)}=\frac{(1+2\eps)\log n}{f((1+3\eps)\e\log n)}.\end{equation*}
As $f$ satisfies~\eqref{eq:widesv}, $f(\log n)/f((1+3\eps)\e \log n)\to 1$. Thus, almost surely, for all large~$n$,
\begin{equation*}\tau_n\le \frac{(1+\eps)\log n}{f(\log n)}\le \frac{(1+2\eps)\log n}{f((1+3\eps)\e\log n)} \le \tau_{1,(1+3\eps)\e\log n}.\end{equation*}
Since $\eps > 0$ is arbitrarily small and $\cT(\tau_n)$ and $T_n$ have the same distribution by~\eqref{eq:taun}, the latter display proves the upper bound in \Cref{thrm:svf}.
\end{proof} 

    We conclude the section by proving \Cref{prop:taunsv}.

    \begin{proof}[Proof of \Cref{prop:taunsv}]
    As before, we extend the function $f$ to $[0,\infty)$ by linear interpolation between consecutive integers but keep denoting the extension $f$. Instead of proving bounds on $\tau_n$, it will be enough to show that almost surely
    \begin{equation}\label{eq:equiv3.9}
    \liminf_{t\to\infty} \frac{\log(N(t))}{tf(t)} \ge 1.
    \end{equation}
Indeed, as $N(\tau_n)=n$ for every $n\in \N$, \eqref{eq:equiv3.9} implies the desired result. 
\vspace{0.5em}

Fix $\eps\in (0,1)$ small and $t$ large so that $\eps t$ is also large. Set also $K = K(t) \coloneq \eps t / 4$. 
\cb{We assume that $f(0)=1$: note that this leads to no loss of generality since rescaling the function $f$ by a factor $c>0$ does not change the attachment probabilities and only modifies the speed of the growth process by a factor of $c^{-1}$ (and, in particular, rescales the stopping time $\tau_n$ by the same amount, so using the hypothesis of slow growth shows that the limit of $(\tau_n/c) (cf(\tau_n/c)) = \tau_n f(\tau_n/c)$ remains unchanged). 
Using our assumption that $f$ is non-decreasing, we have $f(k) \ge 1$ for all $k \ge 1$.}

We begin by showing that with high probability the tree attains depth $K$ quickly. 
To do so, we observe that, for each $j\in \N$, $\tau_{1,j}-\tau_{1,j-1}$ is stochastically dominated by an exponential random variable $E_j$ with rate $f(j-1)\geq 1$, which is the time it takes for the first vertex on depth $j-1$ to produce its first child. 
Even further, the variables $\{E_j\}_{j\geq1}$ are independent, so each $\tau_{1,j}$ \cb{is stochastically dominated by} the sum of $j$ independent exponential variables with rate $1$. Taking $j=K$ and using Bernstein's inequality (\Cref{lem:Ber}) gives
\begin{equation}\label{eq:bound_exp}
\begin{split}
\mathbb P(N_{K}(2K)=0)
&=\mathbb P(\tau_{1,K}>2K)\\
&\le \cb{\mathbb P\bigg(\bigg(\sum_{i=1}^k E_i\bigg)-K>K\bigg)}\le 2\exp\bigg(-\min\bigg\{\frac{K^2}{c_1K},\frac{K}{c_2}\bigg\}\bigg).
\end{split}
\end{equation}
Next, denote by $v_K$ the vertex born on depth $K$ at time $\tau_{1,K}$. 
From this time onwards, we focus on the subtree rooted at $v_K$. 
Recalling that $\min_{i\ge K} f(i) = f(K)$ by monotonicity of $f$, we can (and do) couple this tree with the branching process $\overline{\cT}(\cdot)$ given by the DWT with constant weight function equal to $f(K)$ so that $\overline{\cT}(s) \subseteq \cT(\tau_{1,K} + s)$ for all $s \ge 0$, with $\cT$ being the original DWT. 
Call $Y(s)=|\overline{\cT}(s)|$ and observe that, since the weight function of $\overline\cT$ is constant, $Y(s)$ is a Yule process with rate $f(K)$.
By~\cite[(4.3.4)]{vdHof17} and simple rescaling, $Y(s)$ is a geometric random variable with parameter $p=p(s)\coloneq\exp(-sf(K))$.
Therefore, setting $s=t-2K=(1-\eps/2)t$ and using~\eqref{eq:widesv}, we obtain that 
\begin{equation*}\mathbb P(Y(s)\le \e^{(1-\eps)sf(K)}) =1-(1-p)^{\exp((1-\eps)sf(K))}\le \e^{-\eps sf(K)}\le \e^{-\eps(1-\eps/2)tf(t)/2}.\end{equation*}
By combining the latter with~\eqref{eq:widesv} and~\eqref{eq:bound_exp}, we obtain that

\begin{align*}
\mathbb P(N(t)\le \e^{(1-3\eps)tf(t)})
&\le \mathbb P(N(t)\le \e^{(1-2\eps)tf(K)})\\
&\le \mathbb P(N_{K}(2K)=0)+\mathbb P(\{N(t)\le \e^{(1-\eps)sf(K)}\}\cap \{N_{K}(2K)\ge 1\})\\
&\le 2\e^{-K/\max\{c_1,c_2\}}+\mathbb P(Y(s)\le \e^{(1-\eps)sf(K)})\le 3\e^{-K/\max\{c_1,c_2\}},
\end{align*}
where the third inequality uses that $N_K(2K)\geq 1$ implies $\tau_{1,K}<2K$ and hence our coupling gives $Y(s)\leq N(s+\tau_{1,K})\leq N(t)$, and the last inequality uses that $t$ is large.
Combined with the monotonicity of each of $N(\cdot)$ and $x\mapsto xf(x)$ and Borel-Cantelli's lemma for $t\in \N$, this is enough to conclude the proof of \eqref{eq:equiv3.9} as in the proof of \Cref{prop:taun}.
\end{proof}

\section{Exponential weight functions}\label{sec:exp_weight_fns}

We turn to Theorem~\ref{thm:main_exp}. 
First, we prove part (ii) for subexponential functions.

\begin{proof}[Proof of \Cref{thm:main_exp}(ii)]
Recall the discrete model from Definition~\ref{def:DWT} and define the function
\begin{equation*}g:n\in \mathbb N\mapsto \frac{f(n)}{\sum_{i=0}^n f(i)}.\end{equation*}
Then, our assumption on the weight function implies that $g(n) = o(1)$. 
Set $\zeta_0 = 0$ and, for every $k\in \N$, denote by $\zeta_k$ the smallest $n\in \N$ such that $T_n$ has depth $k$; observe that $\zeta_k$ is a discrete analogue of $\tau_{1,k}$ from previous sections.
Recalling the convention that irrelevant rounding is ignored, our task comes down to bounding from above the probability of the event $\zeta_{\eps n}\le n$. 

To begin with, for every $i\in \mathbb N$, the probability that the vertex born at step $\zeta_{k}+i$ has depth $k+1$ is bounded from above by $p_{k,i} \coloneq \min(ig(k),1)$. 
Indeed, at this step, there is at least one vertex on each of the depths $0,1,\ldots,k$ and at most $i$ vertices on depth $k$.
For all $k\in \N_0$ and $i\in \N$, define the family of independent Bernoulli random variables $(I_{k,i})$ with respective parameters $(p_{k,i})$.
Then, for every $k\in \mathbb N_0$, $\zeta_{k+1} - \zeta_k$ stochastically dominates $\min\{i:I_{k,i} = 1\}$.
Let $\cI = \{\eps n/2,\ldots,\eps n-1\}$ and $\cS$ be the family of subsets of $\cI$ of size $\eps n/4$. Then,
\begin{align*}
\mathbb P(\zeta_{\eps n}\le n)
&\le \mathbb P(\exists S\in \cS: \forall k\in S,\, \zeta_{k+1}-\zeta_{k}\le 5/\eps)\\
&\le \binom{\eps n/2}{\eps n/4} \max_{S\in \cS} \bigg\{\prod_{k\in S} \mathbb P(\exists i\le 5/\eps: I_{k,i}=1)\bigg\}\le 2^{\eps n/2} \bigg(\max_{k\in \cI} \bigg\{\sum_{i=1}^{5/\eps} p_{k,i}\bigg\}\bigg)^{\eps n/4}\\
&\le 2^{\eps n/2} \bigg(\sum_{i=1}^{5/\eps} i \max_{k\in \cI} g(k)\bigg)^{\eps n/4} \le \bigg(\frac{100}{\eps^2} \max_{k\in \cI} g(k)\bigg)^{\eps n/4} = \e^{-\omega(n)},
\end{align*}
where the last inequality used that $\max_{k\in \cI} g(k) = o(1)$.
\end{proof}

We now turn to the proof of Theorem~\ref{thm:main_exp}(i) for $c > 1$. 
Even though the results presented here apply to all functions satisfying assumption~\eqref{eq:ExpG}, the reader might benefit from keeping in mind the case of pure exponential functions $k\mapsto c^k$ for some $c>1$, for which some of the definitions and proofs become cleaner. 

The proof consists of several steps. The first one is Proposition~\ref{prop:small_tower_ex} below, which is arguably the most technical component. 
Before stating the proposition formally, we first give an intuitive explanation of its meaning. 
Imagine that a large mass of vertices in the tree $T_n$ accumulates at a certain depth $k$ away from the deepest vertices.
Then, this accumulation is likely to result in a rapid gain of mass on depths $k+1, k+2, \ldots$. 
This `wave' of vertices eventually reaches and overtakes the vertices at maximal depth and significantly delays the growth of $d(T_n)$.
Our analysis of exponentially growing functions shows that the above phenomenon is atypical to happen on a large scale: accumulations of mass are rare and well-spaced, which results in a regularly paced growth of the height of the tree.
To properly state this intuition, define the function
\begin{align}\label{eq:Psi}
\Psi: k\in \mathbb N
&\mapsto \min\left\{\ell\in \mathbb N:\frac{\prod_{j=1}^{\ell}f(k+j-1)}{(2f(k-1))^{\ell} (\ell+1)!} \ge 8\right\},
\end{align}
with the convention $\Psi(k)=\infty$ if there is no such $\ell$. This function controls the accumulation on depth $k$: in fact, $k+\Psi(k)$ can be interpreted as a depth which, once reached, the vertices on depth $k$ lose their impact with probability close to 1.
Note that, for the particular choice of $f(k)=c^k$,
\begin{equation*}\Psi(k)=\min\left\{\ell\in \mathbb N:\frac{c^{\frac{\ell(\ell+1)}{2}}}{2^{\ell} (\ell+1)!} \ge 8\right\}.\end{equation*}
While generally $\Psi$ depends on $k$, we show $\Psi$ is always bounded thanks to assumption~\eqref{eq:ExpG}.

\begin{lemma}\label{lem:mfkf}
For every weight function $f$ satisfying assumption~\eqref{eq:ExpG} \cb{with $c>1$}, $\Psi$ is a bounded function.
\end{lemma}
\begin{proof}
Let $k_0$ be the smallest integer such that, for all $k\ge k_0$, both $f(k)\ge \tfrac{c+1}{2} f(k-1)$ and $f(k-1)\ge \max\{f(i): i\in \{0,1,\ldots,k-1\}\}$ hold.
Then, for every $k\ge k_0$, we have that $\Psi(k)$ is bounded from above by the minimum integer $\ell$ such that 
\begin{equation*}\bigg(\frac{c+1}{2}\bigg)^{\ell(\ell+1)/2}\ge 8\cdot 2^{\ell}(\ell+1)!,\end{equation*}
which exists since $n! = 2^{o(n^2)}$.
For $k\in [k_0-1]$, by distinguishing the cases $k+j-1 < k_0$ and $k+j-1\ge k_0$, the inequality in~\eqref{eq:Psi} is satisfied when
\begin{equation*}\bigg(\prod_{j=1}^{k_0-k} \frac{f(k+j-1)}{f(k-1)}\bigg)\cdot \bigg(\frac{f(k_0-1)}{f(k-1)}\bigg)^{\ell-k_0+k} \bigg(\prod_{j=k_0-k+1}^{\ell} \frac{f(k+j-1)}{f(k_0-1)}\bigg)\ge 8\cdot 2^\ell (\ell+1)!.\end{equation*}
Since the second term above is at least 1 and the third term dominates $c^{(\ell-k_0+k)(\ell-k_0+k+1)/2}$, the last inequality holds for every large enough $\ell$. Thus, $\Psi(k)<\infty$ for all $k<k_0$, as desired.
\end{proof}

Next, we define a family of important events appearing in Proposition~\ref{prop:small_tower_ex}. 
To this end, we work in the continuous-time setting introduced in Section~\ref{sec:Cox}. 
Recall that $N_k(t)$ stands for the number of vertices on depth $k$ at time $t$, $\tau_{s,k}$ stands for the first time at which there are $s$ vertices on depth $k$, and $\tau_n$ stands for the first time when the tree attains $n$ vertices.

\cb{\begin{definition}\label{def:G_sr}
For integers $r,s\in \mathbb N$, define the event
\[\cG_{s,r} \coloneq \{\tau_{s,r}<\min\{\tau_{s+1,r-1},\tau_{s+1,r+1},\ldots,\tau_{s+1,r+\Psi(r)}\}\}\]
\end{definition}

In words, when constructing the Cox process from \Cref{sec:Cox}, $\cG_{s,r}$ is the event that at some time depth $r$ contains at least $s$ vertices while each of the depths $r-1,r+1,r+2,\ldots,r+\Psi(r)$ contains at most $s$ vertices; see \Cref{fig:Es}. Next, we show that this is a rather improbable event when $s$ is large: the reason is that a large mass of vertices on depth $r$ typically results in a faster growth on depths $r+1,r+2,\ldots$ whilst the upper bound on depth $r-1$ limits the growth on depth $r$ itself. 
The following proposition formalises this intuition by providing an upper bound on $\mathbb P(\cG_{s,r})$ which is uniform over $r$ and decreases exponentially with $s$.}  

\begin{center}
    \begin{figure}[h]
    \includegraphics[scale=1.2]{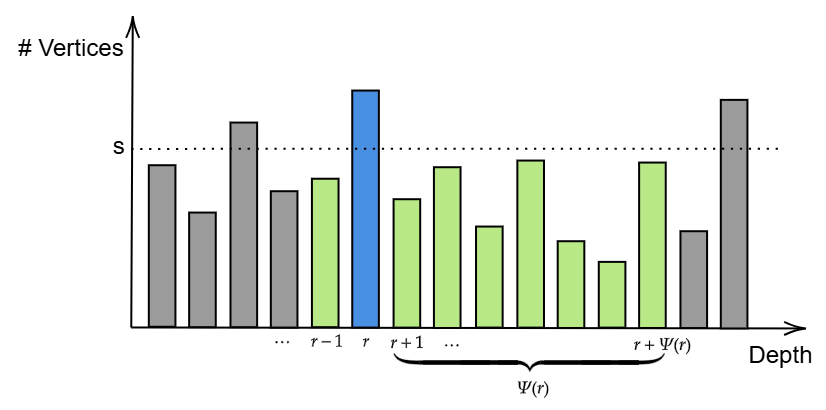}
    \caption{An illustration of the event \cb{$\cG_{s,r}$}. 
    In blue, the load of vertices on depth $r$, which is above the threshold $s$. In green, the previous depth and the subsequent $\Psi(r)$ depths remaining below the threshold $s$.}
    \label{fig:Es}
    \end{figure}
\end{center}

\begin{proposition}\label{prop:small_tower_ex}
Fix a weight function $f$ satisfying assumption~\eqref{eq:ExpG} \cb{with $c>1$ and, for any $\ell\in\N$, define $\mathcal{F}_{\ell}$ as the $\sigma$-algebra generated by the $N_0(\cdot),N_1(\cdot),\ldots,N_\ell(\cdot)$ processes, with the convention that $\mathcal{F}_{-1}$ is the trivial $\sigma$-algebra}. Then, there is $\delta = \delta(f) > 0$ such that, for all $r,s\in \mathbb N$ with $s\geq2$,
\begin{equation*}\cb{\condP{\cG_{s,r}}{\mathcal{F}_{r-2}}} \le 2(\sup\Psi+2) \e^{-\delta s}\le 6(\sup\Psi) \e^{-\delta s}.\end{equation*}
\end{proposition}
\begin{proof}

\cb{Fix $r,s\in\N$ with $s\geq 2$.} For the purpose of the analysis, we introduce an auxiliary process $(M_{\ell}(\cdot))_{\ell\in \mathbb N_0}$ defined in the same way as $(N_{\ell}(\cdot))_{\ell\in \mathbb N_0}$ in Section~\ref{sec:Cox}, 
except for the fact that $M_{r-1}(\cdot)$ is stopped once it reaches size $s$, that is, it satisfies $M_{r-1}(\cdot) = \min\{N_{r-1}(\cdot), s\}$. 
Observe that the restriction on $M_{r-1}$ influences the distribution of all subsequent depths. This auxiliary process is interpreted similarly to $(N_\ell(\cdot))_{\ell\in \N_0}$ but the uniformly bounded load on depth $r-1$ is convenient for consequent computations.
Note that, by restricting the number of vertices on depth $r-1$, the probability of the event $\cG_{s,r}$ (with the stopping times $\tau_{j,\ell}$ redefined for the process $(M_{\ell}(\cdot))_{\ell\in \mathbb N_0}$) can only increase since we exclude the possibility of having $\tau_{s+1,r-1}\leq \tau_{s,r}$. 
From now on, we work with $(M_{\ell}(\cdot))_{\ell\in \mathbb N_0}$ instead of $(N_{\ell}(\cdot))_{\ell\in \mathbb N_0}$, that is, we adapt the definition of the event $\cG_{s,r}$, the directing measures $(\Lambda_\ell)_{\ell\in \N}$ from \Cref{sec:Cox} and the random variables $\tau_{j,\ell}$ and $\tau_j$ to the process $(M_{\ell}(\cdot))_{\ell\in \mathbb N_0}$. \cb{Observe that $M_{\ell}(\cdot)$ coincides with $N_{\ell}(\cdot)$ for all $\ell\leq r-2$, hence $\mathcal{F}_{r-2}$ is unaltered in this new setting.}

To begin with, we note that the distribution of the processes $(M_\ell(\cdot))_{\ell\in [r, \infty)}$ is independent of $(M_\ell(\cdot))_{\ell\in [0,r-2]}$ conditionally on a realisation of $(M_{r-1}(t))_{t\geq 0}$.
In particular,
\begin{equation*}\condP{\cG_{s,r}}{\mathcal{F}_{r-2}}=\condE{\condP{\cG_{s,r}}{(M_{r-1}(t))_{t\geq 0}}}{\mathcal{F}_{r-2}}.\end{equation*}
Thus, our goal is to obtain a suitable upper bound (which will turn out to be uniform in $r$) on the conditional probability $\condP{\cG_{s,r}}{(M_{r-1}(t))_{t\geq 0}}$. A heuristic version of our argument is as follows: given any trajectory of $M_{r-1}$ (outside of a null probability set), the trajectories of the stochastic processes $M_{r}(\cdot),\,M_{r+1}(\cdot),\ldots,\,M_{r+\Psi(r)}(\cdot)$ are typically concentrated around some deterministic functions $F_r(\cdot)$, $F_{r+1}(\cdot),\ldots,F_{r+\Psi(r)}(\cdot)$ which are constructed recursively from $M_{r-1}(\cdot)$. Using $F_r$ (which approximates $M_r$), we can identify the time $T$ when $F_r$ is close to the target value $s$.

To complete the proof, it suffices to show that at least one of the functions $F_{r+1},\ldots,F_{r+\Psi(r)}$ attains a value significantly larger than $s$ before time $T$. 
In this case, the event $\cG_{s,r}$ can only occur if at least one of the stochastic processes deviates substantially from its deterministic counterpart, an event which will turn out to have an exponentially small probability. 
Note that the exponential growth of $f$ in the interval $[r-1,r+\Psi(r)]$ is essential to prove that this strategy works irrespectively of the chosen trajectory of $M_{r-1}(\cdot)$.

\vspace{0.5em}

As a first step, we address the trajectory of $M_r$. Fix a realisation of $(M_{r-1}(t))_{t\geq 0}$ and call it $F_{r-1}(t)$. 
By the construction of the process, each of $M_{r}(\cdot),\,M_{r+1}(\cdot),\ldots,\,M_{r+\Psi(r)}(\cdot)$ must be equal to 0 before time $\tau_{1,r-1}$ when generation $r-1$ receives its first vertex. 
To simplify subsequent computations, we (somewhat abusively) identify $\tau_{1,r-1}$ with 0 so that $F_{r-1}(0)=1$. By construction,
\begin{equation}\label{eq:Lambdar}
\Lambda_r[0,t]\coloneq\int_0^tf(r-1)F_{r-1}(u)\dd u
\end{equation}
is deterministic and therefore $M_r$ is a Poisson process whose expected trajectory follows $\Lambda_r[0,\cdot]$. 
To unify notation, we call $F_r(t)=\Lambda_r[0,t]$: this is a continuous and increasing function since $F_{r-1}$ is the trajectory of a positive jump process. 
Keeping in mind that our aim is to control the probability of $\cG_{s,r}$, we define $T=T(F_{r-1},\varepsilon)$ as the unique value such that
\begin{equation}\label{def:defT_poisson}F_r(T)=(1-2\varepsilon)s,\end{equation}
for some $\varepsilon\in(0,0.1)$ depending on $f$ to be fixed later. 
The key idea behind this choice is that $[0,T]$ is not long enough for $M_r$ to reach the value $s$, whereas at least one of the subsequent depths contains at least $s$ vertices before time $T$.
To show that the process $M_{r}$ stays close to $F_r$, we make use of Theorem~\ref{theo: LeGuvel} with $M_r$ and $\dd F_r$ taking the role of the Poisson process and its directing measure, respectively, and with $x=\varepsilon s$, thus obtaining
\begin{equation*}\mathbb P\bigg(\sup_{u\in [0, T]} \left|M_{r}(u)-F_{r}(u)\right|\geq \varepsilon s\,\bigg|\,M_{r-1}=F_{r-1}\bigg)\,\leq\,2\exp\bigg(-F_r(T)\cdot I\bigg(\frac{\varepsilon s}{F_{r}(T)}\bigg)\bigg),\end{equation*}
where $\varepsilon$ is taken the same as in the definition of $T$ and $I(z)=(1+z)\log(1+z)-z$. 
Using the definition of $T$ in \eqref{def:defT_poisson}, we thus arrive at
\begin{equation*}\mathbb P\bigg(\sup_{u\in [0, T]} \left|M_{r}(u)-F_{r}(u)\right|\geq \varepsilon s\,\bigg|\,M_{r-1}=F_{r-1}\bigg)\,\leq\,2\e^{-I_1 s},\end{equation*}
where $I_1\coloneq(1-2\varepsilon)I(\varepsilon/(1-2\varepsilon))$ can be checked to be bounded from below by $\delta\coloneq\varepsilon^2/8$ from our choice of $\eps\in (0,0.1)$. 
Now, we define the event $\mathcal{B}_r$ that $M_r$ concentrates around the deterministic trajectory $F_r$, that is,
\begin{equation*}\mathcal{B}_r\coloneq\left\{\sup_{u\in [0, T]} \left|M_{r}(u)-F_{r}(u)\right|\leq \varepsilon s\right\}.\end{equation*}
Then, on $\mathcal{B}_r$, we necessarily have $M_r(T) \leq F_r(T) + \varepsilon s \leq (1 - \varepsilon)s$, which means that by time $T$ the tree has not yet produced $s$ vertices on depth $r$. 
We can therefore bound $\mathbb P(\mathcal{G}_{s,r})$ by the probability of the slightly more general event that none of the subsequent depths has reached $s+1$ vertices by time $T$:
\begin{align*}
\mathbb P(\cG_{s,r}&\mid M_{r-1}=F_{r-1})\leq \condP{\cG_{s,r}\cap\mathcal{B}_r}{M_{r-1}=F_{r-1}}+2\e^{-\delta s}\\
&\leq \condP{\{T<\min\{\tau_{s+1,r+1},\ldots,\tau_{s+1,r+\Psi(r)}\}\}\cap\mathcal{B}_r}{M_{r-1}=F_{r-1}}+2\e^{-\delta s}.
\end{align*}
Now, define the event $\mathcal{C}_{r+1}\coloneq\{T<\min\{\tau_{s+1,r+1},\ldots,\tau_{s+1,r+\Psi(r)}\}\}$ and observe that once $M_{r-1}$ and $\varepsilon$ (and thus $T$) are fixed, 
$\mathcal{C}_{r+1}$ depends solely on the trajectories of $M_{r+1},M_{r+2},\ldots,M_{r+\Psi(r)}$ on $[0,T]$, which can be constructed recursively if $\Lambda_{r+1}$ is given on this interval.
Moreover, an immediate coupling argument shows that $\mathcal{C}_{r+1}$ is monotone on $\Lambda_{r+1}$, meaning that for all increasing non-negative functions $G_1, G_2: [0,\infty)\to (0, \infty)$ satisfying $G_1\le G_2$ on $[0,T]$, we have 
\begin{equation}\label{eq:ineq_monotonicity}\condP{\mathcal{C}_{r+1}}{\Lambda_{r+1}[0,\cdot]=G_2(\cdot)}
\leq\condP{\mathcal{C}_{r+1}}{\Lambda_{r+1}[0,\cdot]=G_1(\cdot)}.
\end{equation}
Indeed, if $G_1\leq G_2$, then the trajectory of $M_{r+1}$ constructed using $\dd G_1$ as directing measure is stochastically bounded from above by the trajectory of the same process using $\dd G_2$. 
Since the directing measure of $M_{r+i+1}$ is given by an integral over $M_{r+i}$ for every $i\ge 1$, repeating the argument allows to show that the collection of trajectories $M_{r+1},\ldots,M_{r+\Psi(r)}$ constructed based on $\dd G_1$ is stochastically bounded from above by the same collection using $\dd G_2$, and thus the event $\mathcal{C}_{r+1}$ is less likely for the latter.

\vspace{0.5em}

Recall that $\Lambda_{r+1}$ is constructed from $M_r$ as
\begin{equation*}\Lambda_{r+1}[0,t]=\int_0^tf(r)M_r(u)\dd u,\end{equation*}
where $M_r$ is random. 
Nevertheless, conditionally on $\mathcal{B}_r$, the trajectory of $M_r$ is at distance at most $\varepsilon s$ from $F_r$, and hence
\begin{equation*}\Lambda_{r+1}[0,t]\geq F_{r+1}(t)\coloneq f(r)\int_0^t \max\{F_r(u)-\varepsilon s,0\} \dd u\end{equation*}
for all $t\in[0,T]$.
Using the stochastic domination stated in \eqref{eq:ineq_monotonicity} with $G_2=\Lambda_{r+1}[0,\cdot]$ and $G_1=F_{r+1}$, and the independence between $M_{r+1}$ and $M_{r-1}$ conditionally on $\Lambda_{r+1}$, we obtain
\begin{align}
\mathbb P(\cG_{s,r}\mid M_{r-1}=F_{r-1}) 
&\leq \condP{\mathcal{C}_{r+1}\cap\mathcal{B}_r}{M_{r-1}=F_{r-1}}+2\e^{-\delta s}\nonumber\\
&\leq \condP{\mathcal{C}_{r+1}}{\Lambda_{r+1}[0,\cdot]=F_{r+1}(\cdot)}+2\e^{-\delta s}.\label{eq:B_0}
\end{align}
To bound the probability term on the right-hand side of~\eqref{eq:B_0}, we consider two possible scenarios for $F_{r+1}(T)$.

\vspace{0.5em}
\noindent
\textbf{Case 1: $F_{r+1}(T)\geq 2s$.}
In this case, since the trajectory of $M_{r+1}$ typically dominates $F_{r+1}$ over the entire interval $[0,T]$, $M_{r+1}$ reaches values significantly larger than $s$ before time $T$ with high probability. 
To formalise this statement, we use Theorem~\ref{theo: LeGuvel} again to obtain that
\begin{align*}\condP{\sup_{u\in [0, T]}\left|M_{r+1}(u)-F_{r+1}(u)\right|\geq \varepsilon F_{r+1}(T)}{\Lambda_{r+1}[0,\cdot]=F_{r+1}(\cdot)}\,&\leq\,2\e^{-F_{r+1}(T)I(\varepsilon)}\\
&\leq\, 2\e^{-I_2s},\end{align*}
where $I_2\coloneq2I(\varepsilon)$ can be checked to be bounded from below by $\delta=\varepsilon^2/8$ (again from our choice $\varepsilon\in(0,0.1)$). Notice however that $\mathcal{C}_{r+1}$ must be contained in the event inside the probability above, since if each of the events $\sup_{u\in [0, T]} |M_{r+1}(u)-F_{r+1}(u)| \le \varepsilon F_{r+1}(T)$ and $F_{r+1}(T)\geq 2s$ holds, then for every $s\geq 2$,
\begin{equation*}M_{r+1}(T)\geq (1-\varepsilon)F_{r+1}(T)\geq (1-\varepsilon)\cdot 2s > s+1.\end{equation*}
It follows that if $F_{r+1}(T)\geq 2s$ holds, then
\begin{equation*}
\condP{\mathcal{C}_{r+1}}{\Lambda_{r+1}[0,\cdot]=F_{r+1}(\cdot)}\leq 2\e^{-\delta s}.
\end{equation*}

\noindent
\textbf{Case 2: $F_{r+1}(T) < 2s$.} In this case, the deterministic trajectory $F_{r+1}$ does not reach values large enough to infer that $M_{r+1}$ reaches $s$ vertices before time $T$. 
However, concentration of $M_{r+1}$ around $F_{r+1}$ still holds with high probability, which we use to derive an upper bound for $\condP{\mathcal{C}_{r+1}}{\Lambda_{r+1}[0,\cdot]=F_{r+1}(\cdot)}$ which is identical in structure but only takes into account the trajectories of $M_{r+2},\ldots,M_{r+\Psi(r)}$. 
Analogously to $\mathcal{B}_r$, we define
\begin{equation*}\mathcal{B}_{r+1}\coloneq\left\{\sup_{u\in [0, T]} \left|M_{r+1}(u)-F_{r+1}(u)\right|\leq \varepsilon s\right\}\end{equation*}
and use Theorem~\ref{theo: LeGuvel} again to obtain
\begin{equation*}\condP{\cB_{r+1}^c}{\Lambda_{r+1}[0,\cdot]=F_{r+1}(\cdot)}\,\leq\,2\e^{-F_{r+1}(T)I(\varepsilon s/F_{r+1}(T))}.\end{equation*}
To analyse the bound obtained on the right-hand side, we observe that, for every $\alpha>0$, the function $x\in (0, \infty)\mapsto xI(\alpha/x)$ has derivative $\log(1+\alpha/x) - \alpha/x < 0$, so it is decreasing. In our case $\alpha=\varepsilon s$ and $x=F_{r+1}(T)<2s$. Using monotonicity to replace $F_{r+1}(T)$ by $2s$ in the exponent, we arrive at
\begin{equation*}\condP{\cB_{r+1}^c}{\Lambda_{r+1}[0,\cdot]=F_{r+1}(\cdot)}\,\leq\,2\e^{-I_3 s}\end{equation*}
where $I_3 = 2I(\eps/2)$ is bounded from below by $\delta=\eps^2/8$ from our assumptions on $\varepsilon$. It follows that
\begin{equation*}\condP{\mathcal{C}_{r+1}}{\Lambda_{r+1}[0,\cdot]=F_{r+1}(\cdot)}\,\leq\,\condP{\mathcal{C}_{r+1}\cap\mathcal{B}_{r+1}}{\Lambda_{r+1}[0,\cdot]=F_{r+1}(\cdot)}+2\e^{-\delta s}.\end{equation*}
\cb{Furthermore,} on the event $\mathcal{B}_{r+1}$, the trajectory of $M_{r+1}$ is within distance $\varepsilon s$ from $F_{r+1}$, and hence
\begin{equation*}\Lambda_{r+2}[0,t]\geq F_{r+2}(t)\coloneq f(r+1)\int_0^t\max\{F_{r+1}(u)-\varepsilon s,0\}\dd u,\end{equation*}
for all $t\in[0,T]$. We can thus repeat the argument used to derive \eqref{eq:B_0} and obtain
\begin{equation*}\condP{\mathcal{C}_{r+1}}{\Lambda_{r+1}[0,\cdot]=F_{r+1}(\cdot)}\,\leq\, \condP{\mathcal{C}_{r+2}}{\Lambda_{r+2}[0,\cdot]=F_{r+2}(\cdot)}+2\e^{-\delta s},\end{equation*}
where $\mathcal{C}_{r+2}$ is defined analogously to $\mathcal{C}_{r+1}$ as $\mathcal{C}_{r+2}\coloneq\{T<\min\{\tau_{s+1,r+2},\ldots,\tau_{s+1,r+\Psi(r)}\}\}$.

Thus far, we know $\condP{\mathcal{C}_{r+1}}{\Lambda_{r+1}[0,\cdot]=F_{r+1}(\cdot)}$ is bounded either by $2\e^{-\delta s}$, in the case when $F_{r+1}(T)\geq 2s$, or by
\begin{equation*}\condP{\mathcal{C}_{r+2}}{\Lambda_{r+2}[0,\cdot]=F_{r+2}(\cdot)}+2\e^{-\delta s},\end{equation*}
if $F_{r+1}(T)<2s$. 
For every $i\in [\Psi(r)]$, defining $\cC_{r+i} \coloneq \{T<\min\{\tau_{s+1,r+i},\ldots,\tau_{s+1,r+\Psi(r)}\}\}$ and 
\begin{equation}\label{def:iterative_F}
F_{r+i}(t)\coloneq f(r+i-1)\int_0^t\max\{F_{r+i-1}(u)-\varepsilon s,0\} \dd u\qquad\text{for all }t\in[0,T]
\end{equation}
allows us to iterate the argument to show that, for any $i<\Psi(r)$, $\condP{\mathcal{C}_{r+i}}{\Lambda_{r+i}[0,\cdot]=F_{r+i}(\cdot)}$ is bounded either by $2\e^{-\delta s}$, if $F_{r+i}(T)\geq 2s$, or by
\begin{equation*}\condP{\mathcal{C}_{r+i+1}}{\Lambda_{r+i+1}[0,\cdot]=F_{r+i+1}(\cdot)}+2\e^{-\delta s}\end{equation*}
if $F_{r+i}(T)<2s$.
Iterating this reasoning at most $\Psi(r)$ times and using \eqref{eq:B_0}, we deduce the bound
\begin{equation*}\label{eq:small_tower_ex_almostfinalebound}
\mathbb P(\cG_{s,r}\mid M_{r-1}=F_{r-1})\leq 2(\Psi(r)+1)\e^{-\delta s}+\condP{\mathcal{C}_{r+\Psi(r)}}{\Lambda_{r+\Psi(r)}[0,\cdot]=F_{r+\Psi(r)}(\cdot)}.
\end{equation*}

This bound almost matches the one in the statement, except for the probability term on the right. 
However, notice that, from our previous analysis \cb{in Case 1,} this term would be bounded by $2\e^{-\delta s}$ as soon as $F_{r+\Psi(r)}(T)\geq 2s$, thus giving the result.
The rest of the proof consists in showing that this is indeed the case.

In fact, we will prove something stronger: namely, that we can find $\varepsilon>0$ small, depending on $f$, such that, for all $r$, any choice of $F_{r-1}$, and any $\ell\in \{0,\ldots,\Psi(r)\}$, we have
\begin{equation}\label{eq:Lower_bounds_F}F_{r+\ell}(T)\ge \frac{s}{2} \bigg(\prod_{j=0}^{\ell} \frac{f(r+j-1)}{f(r-1)}\bigg) \frac{(1-2\eps)^{\ell+1}}{(\ell+1)!}.\end{equation} 
If this inequality holds, then the proof of the proposition follows since, in particular,
\begin{equation*}F_{r+\Psi(r)}(T)\ge \frac{s}{2} \bigg(\prod_{j=0}^{\Psi(r)} \frac{f(r+j-1)}{f(r-1)}\bigg) \frac{(1-2\eps)^{\Psi(r)+1}}{(\Psi(r)+1)!}\geq s2^{\Psi(r)+2}(1-2\varepsilon)^{\Psi(r)+1}\geq 2s,\end{equation*}
where for the second inequality we used the definition of $\Psi(r)$, and for the last one we used the fact that $2-4\varepsilon>1$.

To prove \eqref{eq:Lower_bounds_F}, we define $T_0=T-\frac{1-2\varepsilon}{f(r-1)}$ and the function $L_r(t)$ on $[0,T]$ as
\begin{equation*}L_r(t)
=
\begin{cases}
0,
&\mbox{ if } t\in [0, T_0],\\[2pt]
f(r-1)s\cdot (t-T_0), 
&\mbox{ if }t\in (T_0, T].
\end{cases}\end{equation*} 
We argue that $L_r(t)\leq F_r(t)$ on $[0,T]$. This statement is trivial on the interval $[0,T_0]$ whereas on the interval $(T_0,T]$, $L_r$ is a straight line whose slope $f(r-1)s$ dominates the derivative $F'_r(t)=f(r-1)F_{r-1}(t)$ since every realisation $F_{r-1}$ of $M_{r-1}$ is bounded from above by $s$.

\begin{figure}[h]
       \centering 
\includegraphics[scale=0.75]{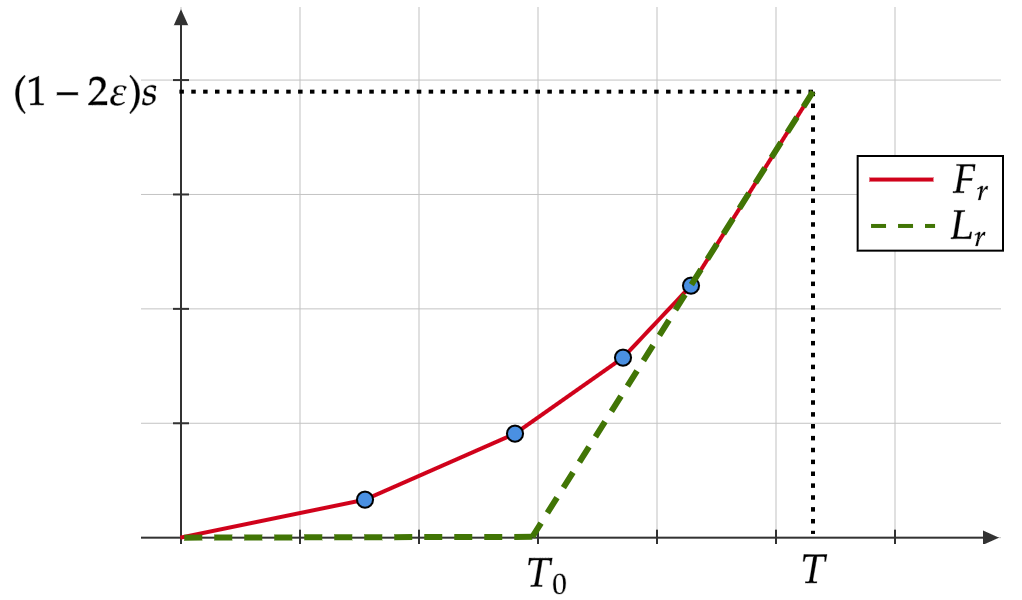}
\caption{A plot of the functions $F_r(t)$ (solid red), and $L_r(t)$ (dashed green). The blue circles represent the times at which $F_{r-1}$ jumps, and hence $F_r$ increases its slope by $f(r-1)$, attaining a maximum slope of $f(r-1)s$. $L_r$, on the other hand, is constructed with one line segment with slope $0$, and another with slope $f(r-1)s$.}\label{fig:Lrfr}
\end{figure}

Given $L_r$, for each $\ell\in [\Psi(r)]$, we recursively define a function $L_{r+\ell}$ on $[0,T]$ by setting
\begin{equation*}L_{r+\ell}(t)=
\begin{cases}
0,
&\mbox{ if } t\in [0, T_0],\\[4pt]
f(r+\ell-1)\int_{T_0}^{t}(L_{r+\ell-1}(u)-\varepsilon s)\dd u, 
&\mbox{ if }t\in (T_0, T].
\end{cases}\end{equation*}
The purpose of $L_{r+\ell}$ is to serve as a simple and convenient lower bound for $F_{r+\ell}$: indeed, from the inequality $L_r\leq F_r$ and the definition \eqref{def:iterative_F}, it can be checked by immediate induction that, for each $\ell$, we have $L_{r+\ell}\leq F_{r+\ell}$ on $[0,T]$. 
Hence, it is enough to prove that $L_{r+\Psi(r)}(T)\ge 2s$. To do so, observe that, for every $t\in(T_0,T]$ and $\ell\in [\Psi(r)]$, we can explicitly compute 

\begin{equation*}L_{r+\ell}(t)\;=\;s \bigg(\prod_{j=0}^{\ell} f(r+j-1)\bigg) \frac{(t-T_0)^{\ell+1}}{(\ell+1)!} -  \eps s \sum_{i=1}^{\ell} \bigg(\prod_{j=i}^{\ell} f(r+j-1)\bigg) \frac{ (t-T_0)^{\ell-i+1}}{(\ell-i+1)!}.\end{equation*}

Taking $t=T$, and using the definition of $T_0$ we arrive at
\begin{equation*}L_{r+\ell}(T)\;=\;s \bigg(\prod_{j=0}^{\ell} \frac{f(r+j-1)}{f(r-1)}\bigg) \frac{(1-2\eps)^{\ell+1}}{(\ell+1)!} - \eps s \sum_{i=1}^{\ell} \bigg(\prod_{j=i}^{\ell} \frac{f(r+j-1)}{f(r-1)}\bigg) \frac{ (1-2\eps)^{\ell-i+1}}{(\ell-i+1)!}.\end{equation*}

Moreover, the second term above is at most
\begin{equation*}
\eps s \max_{i\in [\ell]} \bigg\{\prod_{j=i}^{\ell} \frac{f(r+j-1)}{f(r-1)}\bigg\} \sum_{i=1}^\ell \frac{(1-2\eps)^{\ell-i+1}}{(\ell-i+1)!}\le 3\eps s  \max_{i\in [\ell]} \bigg\{\prod_{j=i}^{\ell} \frac{f(r+j-1)}{f(r-1)}\bigg\}.\end{equation*}
Hence, \eqref{eq:Lower_bounds_F} follows if we show that there exists $\varepsilon$ small such that, for every $r\in \N$ and $\ell\in [\Psi(r)]$,
\begin{equation}\label{eq:max_L}3\eps  \max_{i\in [\ell]} \bigg\{\prod_{j=i}^{\ell} \frac{f(r+j-1)}{f(r-1)}\bigg\}\leq \frac{1}{2} \bigg(\prod_{j=0}^{\ell} \frac{f(r+j-1)}{f(r-1)}\bigg) \frac{(1-2\eps)^{\ell+1}}{(\ell+1)!}.\end{equation}
To prove \eqref{eq:max_L}, we observe that, from assumption~\eqref{eq:ExpG}, there are constants $m$ and $M$ (depending on $f$) such that, for all $r\in \N$, we have $m\leq f(r)/f(r-1)\leq M$, and hence
\begin{equation*}m^j\leq \frac{f(r+j-1)}{f(r-1)}\leq M^j.\end{equation*}
It follows that the left-hand side of \eqref{eq:max_L} is bounded from above by $3\eps M^{\ell(\ell+1)/2}$, while the right-hand side is bounded from below by $\frac{m^{\ell(\ell+1)/2}}{2^{\ell+2}(\ell+1)!}$. Noticing that $\ell\leq\Psi(r)\leq\sup\Psi$ \cb{(where $\sup\Psi<\infty$ by Lemma~\ref{lem:mfkf})}, we can find $\varepsilon$ small enough (depending on $m$, $M$, and $\sup\Psi$, all of which depend solely on $f$) so that \eqref{eq:max_L} holds, and the result then follows.
\end{proof}

\cb{We now apply Proposition~\ref{prop:small_tower_ex} as follows. For $r,s$ as in the statement, and $n\in\N$, define the events $\cE_{s,r,n}$ and $\cE_{s,r}$ as
\begin{equation*}\cE_{s,r,n} \coloneq \{N_{r-1}(\tau_n)\le s\}\cap \{N_r(\tau_n)\ge s\}\cap \bigg(\bigcap_{i=1}^{\Psi(r)} \{N_{r+i}(\tau_n)\le s\}\bigg)\quad \text{and}\quad \cE_{s,r}\coloneq\bigcup_{n=1}^{\infty}\cE_{s,r,n}.\end{equation*}
In words, when constructing the DWT, $\cE_{s,r,n}$ is the event that, at the $n$-th step, depth $r$ contains at least $s$ vertices, while each of the depths $r-1,r+1,r+2,\ldots,r+\Psi(r)$ contains at most $s$ vertices. In that case, we say that there is an \textit{unusual accumulation of size $s$ at depth $r$}. Moreover, we say that the unusual accumulation has size \emph{exactly} $s$ if $N_r(\tau_n)=s$. Observe that any global maximum of the function $r\mapsto N_r(\tau_n)$ is attained at a depth with an unusual accumulation.

The event $\cE_{s,r}$, on the other hand, drops the requirement that this occurs precisely at the $n$-th step, allowing it instead to occur at any point during the construction of the DWT. Notice that, when constructing the DWT using the Cox processes of Section~\ref{sec:Cox}, for all $s,r,n$ we have the inclusions $\cE_{s,r,n}\subseteq\cE_{s,r}\subseteq\cG_{s,r}$ (where we recall that $\cG_{s,r}$ allows the unusual accumulation to occur beyond time $\tau_{\infty}$), so that Proposition~\ref{prop:small_tower_ex} yields
\begin{equation}\label{eq:boundEsr}\P{\cE_{s,r,n}}\leq\P{\cE_{s,r}}\leq\P{\cG_{s,r}}\leq 6(\sup\Psi)\, \e^{-\delta s}.\end{equation}
In particular, we have the following corollary regarding the largest unusual accumulation in the tree $T_n$.} 

\begin{corollary}\label{cor:small levels_ex}
Fix a weight function $f$ satisfying assumption~\eqref{eq:ExpG} \cb{with $c>1$}, and $\delta=\delta(f) > 0$ defined as in Proposition~\ref{prop:small_tower_ex}.
Then, with high probability, there are no more than $(2\log n)/\delta$ vertices on any depth in $T_n$.
\end{corollary}
\begin{proof}
For all $r\in [n]$ and $s\geq 2$, conditionally on the event $N_r(\tau_n) = \max_{k\in [n]} N_k(\tau_n) = s$, 
the event $\cE_{s,r,n}$ is necessarily satisfied.
Hence, by \cb{\eqref{eq:boundEsr}} and a union bound,
\begin{equation*}\label{eq:maxXr}
\P{\max_{k\in [n]} N_k(\tau_n)=s}\leq \P{\exists r\in [n]:\cE_{s,r}}\leq 6(\sup\Psi) \e^{-\delta s}n.
\end{equation*}
Define $s_0\coloneq (2\log n)/\delta$. Summing over all $s\geq s_0$, we obtain
\begin{equation*}\mathbb P\bigg(\max_{k\in [n]} N_k(\tau_n)\ge s_0\bigg)\le 6(\sup\Psi) \frac{\e^{-\delta s_0} n}{1-\e^{-\delta}}\le \frac{6(\sup\Psi)}{(1-\e^{-\delta})n},\end{equation*}
as desired.
\end{proof}

Corollary~\ref{cor:small levels_ex} shows that the largest number of vertices found on a single depth of $T_n$ is typically of order $\log n$, thus showing that the height of $T_n$ is of order at least $n/\log n$.
To strengthen this weaker lower bound to a linear one, we need to show that depths containing many vertices are rare.  

Our strategy is to use Proposition~\ref{prop:small_tower_ex} to control the number of depths $r' \in [n]$ on which the function $r \mapsto N_r(\tau_n)$ attains a sort of `local maximum', and then to show that $N_r(\tau_n)$ reaches comparable values only within a small neighbourhood of each such depth. 
We begin with the former. Fix integers $n,m,s\in \mathbb N$ with $s\geq 2$ and $m\leq n$ and, for each $r\in[m]$, define
\begin{equation*}\cH_{s,r,n}\coloneq\{N_r(\tau_n)=s\}\cap\cE_{s,r,n}\quad\text{and}\quad Y_n(m,s)\coloneq|\{r\in[m]:\cH_{s,r,n}\text{ holds}\}|.\end{equation*}
In words, $\cH_{s,r,n}$ is the event that there is an unusual accumulation of size exactly $s$ at depth $r$, and $Y_n(m,s)$ corresponds to the number of depths with unusual accumulation of size exactly $s$ among the first $m$ depths of $T_n$. Notice that a direct application of \cb{\eqref{eq:boundEsr}} gives
\begin{equation*}\E{Y_n(m,s)}\leq\sum_{r=1}^{m}\P{\cE_{s,r,n}}\leq 6m(\sup\Psi)\e^{-\delta s},\end{equation*}
which is a small fraction of the considered $m$ depths when $s$ is large. The next lemma estimates the probability that $Y_n(m,s)$ exceeds this bound.

\begin{lemma}\label{lem:conc Y_s_ex}
Fix a weight function $f$ satisfying assumption~\eqref{eq:ExpG} \cb{with $c>1$}, and $\delta$ as in Proposition~\ref{prop:small_tower_ex}.
Then, there is $s_0=s_0(f)$ such that, for every $s\ge s_0$, $x > 0$ and $n\geq m\ge 4s$, we have
\begin{equation}\label{eq:bound_Yn}\P{Y_n(m,s) - 6m(\sup\Psi)\e^{-\delta s}\ge x}\leq \widetilde\Psi\cdot \exp\left(-\frac{x^2}{2\widetilde\Psi^2(6m\e^{-\delta s}+x/(3\widetilde\Psi))}\right),\end{equation}
where $\widetilde \Psi: = \sup\Psi + 3.$ 
\end{lemma}
\begin{proof}
\cb{Construct the DWT using the Cox processes of Section~\ref{sec:Cox}, and define $\widetilde{Y}(m,s)\coloneq|\{r\in[m]:\cG_{s,r}\text{ holds}\}|$. It follows directly from the inclusions of events discussed before~\eqref{eq:boundEsr} that
\be 
Y_n(m,s)\leq \widetilde{Y}(m,s),
\ee
so it is enough to prove~\eqref{eq:bound_Yn} with $\wt Y$ instead of $Y_n$.}
Partition the set $\{1,2,\ldots,m\}$ into $\tilde \Psi$ subsets $\{S_i\}_{i=1}^{\tilde \Psi}$ of the form 
\begin{equation*}S_i = \{i,i+\widetilde\Psi,i+2\widetilde\Psi,i+3\widetilde\Psi,\ldots\}.\end{equation*} 
Define \cb{$Z_i(m,s)$} as the number of integers $r\in S_i$ such that \cb{$\cG_{s,r}$} holds. It follows that \cb{$\widetilde{Y}(m,s)$} is the sum of $(Z_i(m,s))_{i=1}^{\widetilde\Psi}$, and hence 
\cb{\begin{equation}\label{eq:prob Z_i}
\P{\widetilde{Y}(m,s) - 6m(\sup\Psi)\e^{-\delta s}\geq x}\leq \sum_{i=1}^{\tilde \Psi}\mathbb P\big(Z_i(m,s) - 6m(\sup\Psi)\e^{-\delta s}/\widetilde\Psi \geq x/\widetilde\Psi\big).
\end{equation}}

\begin{figure}[h]
        \centering 
\includegraphics[scale=1.2]{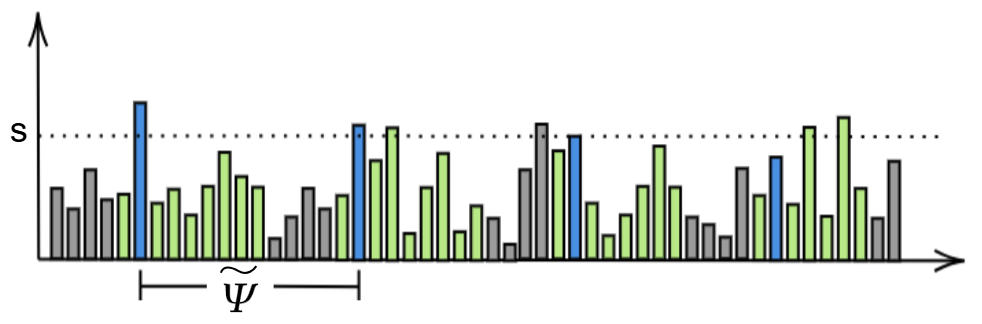}
\caption{In this figure, we see the profile of vertices on each depth of $T_n$. The depths used to construct $Z_i(m,s)$ are in blue. For each of them, one needs to check if the number of vertices on this depth is exactly $s$, and if $\cE_{s,r,n}$ is satisfied. In this example, only the third blue depth (from left to right) satisfies both requirements.}
\end{figure}

Now, fix $i\in\{1,\ldots,\widetilde\Psi\}$ and $r\in S_i$ with $r>\tilde{\Psi}$. For every other $r'\in S_i$ with $r'<r$, we have $r'+\tilde{\Psi}\leq r$ and hence the event $\cG_{s,r'}$ is $\mathcal{F}_{r'+\widetilde\Psi-2}\subseteq\mathcal{F}_{r-2}$-measurable (recall that $\cF_{r'+\widetilde\Psi-2}$ is the $\sigma$-algebra generated by the depths $0,1,\ldots, r'+\widetilde\Psi-2$). It then follows from Proposition~\ref{prop:small_tower_ex} that, for any $S'\subseteq S_i\cap\{1,2,\ldots,r-1\}$, we must have
\begin{equation*}\cb{\P{\cG_{s,r}\,\bigg|\bigcap_{r'\in S'}\cG_{s,r'}}}\leq 2(\sup\Psi+2)\e^{-\delta s}\le 6(\sup\Psi)\e^{-\delta s}.\end{equation*}
In particular, a repeated application of the latter inequality gives that, for every $S'\subseteq S_i$,
\begin{equation*}\P{\bigcap_{r\in S'}\cG_{s,r}}\leq (6(\sup\Psi)\e^{-\delta s})^{|S'|}.\end{equation*}
As a result, \cb{$Z_i(m,s)$} is stochastically dominated by a binomial random variable $Z_i'$ with parameters $m/\widetilde \Psi$ and $\min\{6(\sup\Psi)\e^{-\delta s},1\}$. 
Thus, for every $s$ such that $6(\sup\Psi)\e^{-\delta s}<1$, by Chernoff's bound (\Cref{lem:chernoff}), each of the probabilities on the right-hand side of~\eqref{eq:prob Z_i} is at most
\begin{align*}
\exp\bigg(-\frac{(x/\widetilde\Psi)^2}{2(6m\e^{-\delta s}+x/(3\widetilde\Psi))}\bigg).
\end{align*}
The lemma follows by using~\eqref{eq:prob Z_i}.
\end{proof}

Lemma~\ref{lem:conc Y_s_ex} provides the control needed on the number of unusual accumulations of $T_n$.
To control the effect of these unusual accumulations on other depths, we provide a simple algorithm which, given a realisation of $T_n$ and a starting value $r$, outputs a close position $r'$ in $T_n$ with an unusual accumulation.

\begin{enumerate}
    \item[(1)] Place a walker at position $\textbf{r}=r$ and set $\textbf{s}=N_{\textbf{r}}(\tau_n)$.
    \item[(2)] The walker checks if $\cE_{\textbf{s},\textbf{r},n}$ holds; if it does, we terminate the algorithm and return the position $\textbf{r}$ of the walker.
    \item[(3)] If $\cE_{\textbf{s},\textbf{r},n}$ does not hold, then there is some depth among $\textbf{r}-1$, $\textbf{r}+1, \ldots, \textbf{r}+\Psi(\textbf{r})$ with at least $\textbf{s}+1$ vertices.
    Find the depth $r'\in \{\textbf{r}-1, \textbf{r}+1, \ldots, \textbf{r}+\Psi(\textbf{r})\}$ which maximises $N_{(\cdot)}(\tau_n)$ (in case of a tie, choose a minimal $r'$). Then, update $\textbf{r}=r'$ and $\textbf{s}=N_{r'}(\tau_n)$.
    \item[(4)] Repeat steps (2)--(3) until the algorithm terminates.
\end{enumerate}
A key property of the above algorithm is presented in the next lemma.

\begin{lemma}\label{lem:cover_ex} 
For every $r\in \mathbb N$ and every realisation of $T_n$, the algorithm starting from depth $r$ terminates and its final output $r'$ satisfies $|r'-r|\leq \sup\Psi\cdot (N_{r'}(\tau_n)-N_r(\tau_n))$.
\end{lemma}
\begin{proof}
The algorithm clearly terminates since $T_n$ is finite and every step increases the value of $N_\textbf{r}(\tau_n)$ by at least $1$.
Moreover, at every step, the position of the walker changes by at most $\sup\Psi$, which implies the inequality.
\end{proof}

We say that a depth $r$ with $s\in \mathbb N$ vertices is \emph{covered} by $r'$ if, when initiated at $\textbf{r} = r$, the previous algorithm outputs $\textbf{r} = r'$ upon termination. From the previous lemma, we deduce the following conclusions:
\begin{itemize}
    \item Since the algorithm always terminates, every depth $r\in[n]$ is covered by some $r'\in[n]$ (\cb{and $r=r'$, i.e.\ $r$ is covered by itself, if and only if there is an unusual accumulation on depth $r$}).
    \item If $r$ is covered by $r'$, then $N_r(\tau_n)\leq N_{r'}(\tau_n)$.
    \item An unusual accumulation of size $s$ on depth $r$ covers itself, at most $(\sup\Psi)s$ depths before it, and at most $(\sup\Psi)s$ depths after it, so in total it covers at most $2(\sup\Psi) s+1$ depths. 
\end{itemize}

\begin{figure}[h]
        \centering 
\includegraphics[scale=1.4]{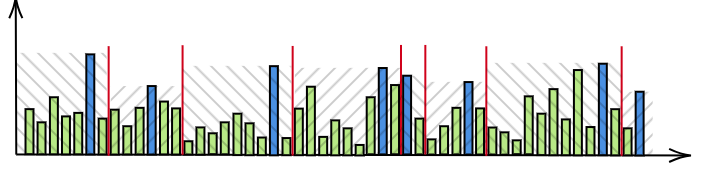}
\caption{A realisation of $T_n$ with $\sup\Psi = 7$, where all depths with an unusual accumulation are shown in blue. Each such depth defines a covering region (dashed in grey), and bounded by red vertical lines.}
\end{figure}

Notice that Lemma~\ref{lem:conc Y_s_ex} controls the number and size of unusual accumulations, while Lemma~\ref{lem:cover_ex} allows to control the amount of vertices at other depths. With these results in hand, we are ready to prove Theorem~\ref{thm:main_exp}.

\begin{proof}[Proof of Theorem~\ref{thm:main_exp}]
Set $m = c_1n$ for a small constant $c_1 > 0$ to be fixed later. 
We construct an event which holds with high probability and implies $d(T_n) > m$. To this end, define 
\begin{equation*}\cC(m)\coloneq\{r\in[n]:\exists r'\in[m],r\text{ is covered by }r'\},\end{equation*}
the set of depths covered by unusual accumulations located within $[m]$. 
Moreover, define $C(m)\coloneq\sum_{r\in\cC(m)}N_r(\tau_n)$, the total number of vertices within this set. 
Observe that the event $C(m)<n$ means that there is a non-empty depth $r$ covered by some $r'\notin [m]$, thus implying $d(T_n)>m$. 
Recall from the discussion following Lemma~\ref{lem:cover_ex} that each unusual accumulation of size exactly $s$ covers at most $2(\sup\Psi)s+1$ depths (including itself), each having at most $s$ vertices. Summing over all unusual accumulations within $[m]$ of different sizes gives the upper bound
\begin{equation*}
C(m)\leq\sum_{s = 1}^{\infty} Y_n(m,s) s(2(\sup\Psi)s+1)\le 3\sum_{s = 1}^{\infty} s^2 (\sup\Psi)\cdot Y_n(m,s),\end{equation*}
where $Y_n(m,s)$ is as in Lemma~\ref{lem:conc Y_s_ex}. We can further improve this bound by observing that Corollary~\ref{cor:small levels_ex} states that, with high probability, 
no depth contains more than $(2\log n)/\delta$ vertices at time $n$. Thus, with high probability,
\begin{equation*}\label{eq:sum:Y-s}
C(m)\leq\ 3\sum_{s = 1}^{(2\log n)/\delta} s^2 (\sup\Psi)\cdot Y_n(m,s).    
\end{equation*}

We now bound the values of $Y_n(m,s)$ by using Lemma~\ref{lem:conc Y_s_ex}. \cb{Call $\widetilde\Psi:=\sup \Psi+3$.}
To this end, recall that there is $s_0=s_0(f)$ such that, for any $s\geq s_0$, the inequality in \eqref{eq:bound_Yn} holds for all $x>0$. For $s\geq s_0$ fixed, we choose $x=14\widetilde\Psi\max\{m\e^{-\delta s},\log(n)\}$ so that the probability on the right-hand side of \eqref{eq:bound_Yn} can be bounded as 
\begin{equation*}\widetilde\Psi\cdot \exp\left(-\frac{x^2}{2\widetilde\Psi^2(6m\e^{-\delta s}+x/(3\widetilde\Psi))}\right)\leq \widetilde\Psi\cdot\exp\left(-\frac{x}{2\widetilde\Psi}\right) = o(n^{-2}),\end{equation*}
where the first inequality used that $2\widetilde\Psi^2(6m\e^{-\delta s}+x/(3\widetilde\Psi))\le 2\widetilde\Psi x$.
Taking a union bound over all $s\in [s_0,2\log(n)/\delta]$, we obtain that, with probability $1-o(n^{-1})$, for all such $s$, we have 
\begin{equation*}Y_n(m,s)\leq\begin{cases}20\widetilde\Psi m\e^{-\delta s}, &\mbox{ if }s_0\leq s\leq \delta^{-1}\log\left(m/\log n\right),\\[2pt]20\widetilde\Psi \log(n), &\mbox{ if }\delta^{-1}\log\left(m/\log n\right)\leq s\leq 2\delta^{-1}\log n.\end{cases}\end{equation*}

Calling $s_1\coloneq\delta^{-1}\log\left(m/\log n\right)$ and $s_2\coloneq2\delta^{-1}\log n$ and observing that $Y_n(m,s)$ is always bounded by $m$, we arrive at

\begin{align*}
C(m)&\leq\ 3\widetilde\Psi\left[\sum_{s = 1}^{s_0} s^2m+\sum_{s = s_0+1}^{s_1} s^2\cdot 20\widetilde\Psi m\e^{-\delta s}+\sum_{s = s_1+1}^{s_2} s^2\cdot  20\widetilde\Psi\log(n)\right]\\
&\leq3\widetilde\Psi\left[ s_0^3m+20\widetilde\Psi\left(\sum_{s = 1}^{\infty} s^2\e^{-\delta s}\right)m+ 20\widetilde\Psi \left(\frac{2\log n}{\delta}\right)^3\log(n)\right]\\
&=C_1m+C_2(\log n)^4,
\end{align*}
where the constants $C_1$ and $C_2$ are absolute and independent of $m$ or $n$. 
Thus, taking $m=c_1n$ with $c_1 C_1 < 1$ gives $C(m) < n$ for all large enough $n$, thus concluding the proof.
\end{proof}

We conclude this section with a brief discussion of Remark~\ref{rem:subexp} on how to extend our results to certain growth functions $f$ that do not satisfy Assumption~\eqref{eq:ExpG}. Notice that these functions only play a role in the proof of Proposition~\ref{prop:small_tower_ex} and are not used afterwards, so that one can try to weaken Assumption~\eqref{eq:ExpG} such that Proposition~\ref{prop:small_tower_ex} still holds. It turns out that for this proposition to work, we only need the function $\Psi(r)$ to be finite for all $r$, even if it is not bounded. In this case, the structure of the proof still applies, although it becomes heavier notation-wise, and more care is needed at certain points. The resulting bound is of the form
\[\P{\cG_{r,s}|\mathcal{F}_{r-2}} \le 5\Psi(r)\e^{-\delta(r) s},
\]
where $\delta$ is no longer a constant depending on $f$, but rather an explicit function of $r$ (and $f$). All intermediate results following Proposition~\ref {prop:small_tower_ex} as well as the structure of the proof of Theorem~\ref{thm:main_exp} can be adapted to include the new dependency of $\delta$ on $r$, although the statements become increasingly more obscure at each step, until concluding the main theorem, which in its adapted version states that with probability $1-n^{-\Omega(1)}$ it holds that
\begin{equation*}d(T_n)\geq m(n),\end{equation*}
for some function $m(n)$ satisfying an intricate set of conditions. As pointed out in \eqref{eq:weakbound}, for the choice $f(k)=\e^{k(\log(k+2))^{-\alpha}}$ we have that one can choose $\nu>0$ such that $m(n)=n\e^{-\nu\log^{\alpha}(n)\log^2\log(n)}$ satisfies these conditions. In particular showing that for $0<\alpha<1$ the growth of the depth is at least close to linear. The particular choice of $m(n)$ in this case is optimal for the set of conditions, and we also believe that our techniques cannot be improved to obtain much better lower bounds, thus suggesting some possibly interesting behavior at $\alpha=1$. 

\section{Super-exponential weight functions}\label{sec:supexp}

In this section, we prove \Cref{thrm:supexp}(i)--(ii). The proof of both parts uses a coupling result (\Cref{cl:couple}) and a delicate moment computation presented in \Cref{lem:tausv2} and \Cref{prop:sec5_moments}. We start with a simple but very useful claim that allows us to relate the number of vertices at a given level with exponential random variables of suitable rates:

\begin{claim}\label{cl:couple}
There is a coupling of the processes $(N_n(\cdot))_{n=1}^{\infty}$ and a sequence of independent exponential random variables $(E_n)_{n=0}^{\infty}$ with rates $(f(n))_{n=0}^{\infty}$, respectively, such that each of the following holds:
\begin{enumerate}
    \item[\emph{(i)}] almost surely, for every $n\ge 0$, $E_n\ge \tau_{1,n+1}-\tau_{1,n}$,
    \item[\emph{(ii)}] for every $n\ge 1$, $(N_j(\cdot))_{j\le n}$ is independent from the collection $(E_i)_{i=n}^{\infty}$. 
\end{enumerate}
\end{claim}
\begin{proof}
Denote by $E_n$ the time needed by the vertex born on depth $n$ at time $\tau_{1,n}$ to produce its first child. We justify that the random variables $(E_n)_{n=0}^{\infty}$ satisfy the properties required in the claim. 
They clearly have the desired distributions, and their independence comes from the fact that different vertices in the DWT tree produce offspring independently of each other. 
The latter argument also justifies (ii).
Finally, as the child of every vertex on depth $n$ is a vertex on depth $n+1$, (i) also holds.
\end{proof}

Next, we work with $(E_n)_{n=0}^{\infty}$ as defined in \Cref{cl:couple}.
For every integer $n\ge 0$, we define $F_n \coloneq \sum_{i=n}^{\infty} E_i$. Observe that Claim~\ref{cl:couple} yields the inequality $F_n \geq \tau_{1,\infty}-\tau_{1,n}$ almost surely for every $n$. 
It then follows that these random variables can be used to control the number of vertices at each level of the tree at time $\tau_{\infty}$ (which equals $\tau_{1,\infty}$ by Lemma~\ref{lemma:explray}). 
Moreover, since each $E_k$ follows an exponential distribution with an explicit rate, expressions involving the collection of random variables $(F_i)_{i\ge 0}$ become tractable from a computational standpoint. 
This can be seen in the next lemma where we derive bounds for the joint moments of $(F_i)_{i\ge 0}$. The proof involves lengthy combinatorial arguments and is therefore deferred to \Cref{sec:appendix}.

\begin{lemma}\label{lem:tausv2}
Fix a weight function $f$ satisfying Assumption~\eqref{ass:D}, and recall the random variables $(E_n)_{n=0}^{\infty}$ and $(F_n)_{n=0}^{\infty}$ from the last paragraph.
For every $n\in\N$, define the random variable $G_n\coloneq f(n-1)F_n$.  

\cb{ 
\begin{enumerate} 
\item[\emph{(A)}] There exists $C(f)>0$ such that for all $\ell>k\geq 1$, \begin{equation*}\E{G_kG_{\ell}}\leq \frac{f(k-1)}{f(k)}\frac{f(\ell-1)}{f(\ell)}\left(1+C\frac{f(k)}{f(k+1)}+C\frac{f(\ell)}{f(\ell+1)}\right).\end{equation*}
\item[\emph{(B)}] For every integer $d\in\N$, there exists a constant $C_d=C_d(f)>0$ such that, for every finite integer set $B\subseteq\N$ with $|B| \geq d$, every sequence $(a_j)_{j\in B}$ with values in $\{1,2\}$, and every subset $I\subseteq B$ with $|I|=d$, we have 
    \begin{equation*}\E{\prod_{j\in B} G_j^{a_j}}\leq C_d2^{-|B|}\prod_{j\in I}\left(\frac{f(j-1)}{f(j)}\right)^{a_{j}}. \end{equation*}
  \end{enumerate} 
  }
\end{lemma}
\begin{remark}
    Part (B) of \Cref{lem:tausv2} allows to find tractable upper bounds for arbitrary (but finite) products of variables of the form $G_k$ and $G_k^2$. By ignoring correlations, a naive approach would be to assume an approximation of the form
    \begin{equation*}\E{\prod_{j\in B} G_j^{a_j}}\leq C\prod_{j\in B} \E{G_j^{a_j}},\end{equation*}
    for some constant $C$ independent of the set $B$. Unfortunately, the correlations between the random variables do not allow for such an inequality, giving instead an inequality of the form
    \begin{equation*}\E{\prod_{j\in B} G_j^{a_j}}\leq M(B)\prod_{j\in B} \E{G_j^{a_j}},\end{equation*}
    where $M$ is some unbounded function. The key idea of the lemma is to find a bound that holds uniformly over $B$ by keeping only a given number of factors on the right-hand side, which is enough for our applications. 
\end{remark}

We finally state the following minor technical lemma regarding the asymptotic growth of partial sums, whose proof we leave to the reader.

\begin{lemma}\label{lemma:sums}
    Let $(c_j)_{j\in\N}$ and $(\eps_j)_{j\in\N}$ be two sequences of non-negative real numbers and let $M_1,M_2=M_1(m),M_2(m)\in\N$ with $M_1\leq M_2$ for all $m$ large be such that 
    \begin{equation*} 
    \lim_{j\to\infty}\eps_j=0\qquad\text{and}\qquad \lim_{m\to\infty} \sum_{j=M_1(m)}^{M_2(m)} c_j=\infty. 
    \end{equation*} 
    Then, as $m\to\infty$, 
    \begin{equation*} 
    \sum_{j=M_1(m)}^{M_2(m)} c_j\eps_j = o\Bigg(\sum_{j=M_1(m)}^{M_2(m)} c_j\Bigg).
    \end{equation*} 
\end{lemma}

We are now ready to prove \Cref{thrm:supexp}. 

\begin{proof}[Proof of \Cref{thrm:supexp} assuming \Cref{lem:tausv2}] We start with part (i). First, define $Z_n=n-d(T_n)$ for every $n\in \N$ and observe that the sequence $(Z_n)_{n\in \N}$ is (deterministically) non-decreasing.
In particular, $(Z_n)_{n\in \N}$ converges almost surely to $Z_{\infty} \coloneq \sup_{n\in \N} Z_n$, which can be written as
\begin{equation*}Z_{\infty}=\sum_{k=1}^{\infty}(N_{k}(\tau_{\infty})-1).\end{equation*}
Observe that \Cref{thrm:supexp}(i) follows as long as $Z_{\infty}$ is a.s.\ finite.
We use the continuous-time interpretation of the DWT from \Cref{sec:Cox} to show that $Z_{\infty}$ is an integrable random variable, \cb{that is,} that the sum of $(N_k(\tau_\infty)-1)_{k\in\N}$ has finite mean. Since almost surely we have $\tau_\infty\leq \tau_{1,k}+F_k$ for all~$k\in \N$,
\be \label{eq:Nkeq}
\E{N_k(\tau_{\infty})-1}=\E{N_k(\tau_{\infty})-N_k{(\tau_{1,k})}}\leq \E{N_k(\tau_{1,k}+F_k)-N_k{(\tau_{1,k})}}.
\ee 
The above expression on the right-hand side has the advantage that, thanks to \Cref{cl:couple}, $F_k$ is independent of $N_k$ and $\tau_{1,k}$. Hence, we are interested in studying \cb{$\condE{H_k}{\mathcal{F}_{k}^{\infty}}$, where
\begin{equation}\label{eq:sec5_defh}
H_k\coloneq N_k(\tau_{1,k}+F_k)-N_k(\tau_{1,k})=N_k(\tau_{1,k}+F_k)-1\quad \text{for all $k\in \mathbb N$,}\quad H_0\coloneq 0,
\end{equation}
and} $\mathcal{F}_{k}^{\infty}$ denotes the $\sigma$-algebra generated by the random variables $(E_i)_{i=k}^{\infty}$. 

We begin by observing that, given the trajectory of $N_{k-1}$, $N_k$ is a Poisson point process with directing measure $N_{k-1}(s)\dd s$, so a direct use of the strong Markov property yields
\be\label{eq:HkasNk}
\condE{H_k}{\mathcal{F}_{k}^{\infty}}=\condE{\int_{\tau_{1,k}}^{\tau_{1,k}+F_k}f(k-1)N_{k-1}(s)\dd s}{\mathcal{F}_{k}^{\infty}}.
\ee 
Recall that almost surely $\tau_{1,k}\leq \tau_{1,k-1}+E_{k-1}$ by Lemma~\ref{cl:couple}. Since $N_{k-1}$ is increasing, shifting the integration interval in \eqref{eq:HkasNk} to the right on the real line can only increase the value of the integral. Hence, 
\begin{equation}\label{eq:new5.5}
\condE{H_k}{\mathcal{F}_{k}^{\infty}}\leq\condE{\condE{\int_{\tau_{1,k-1}+E_{k-1}}^{\tau_{1,k-1}+E_{k-1}+F_k}f(k-1)N_{k-1}(s)\dd s}{\mathcal{F}_{k-1}^{\infty}}}{\mathcal{F}_{k}^{\infty}}.
\end{equation}
Using that $N_{k-1}$ is increasing and that $F_k+E_{k-1}=F_{k-1}$ is $\cF_{k-1}^{\infty}$-measurable, we have that~\eqref{eq:new5.5} is at most
\begin{equation}\label{eq:sec5_fix1}
\begin{split}
&\condE{\condE{f(k-1)F_{k}(H_{k-1}+1)}{\mathcal{F}_{k-1}^{\infty}}}{\mathcal{F}_{k}^{\infty}}\\
&\hspace{10em}\leq \condE{f(k-1)F_{k}\condE{H_{k-1}}{\mathcal{F}_{k-1}^{\infty}}}{\mathcal{F}_{k}^{\infty}}+f(k-1)F_k.
\end{split}
\end{equation}
At this point, we recall the notation $G_k=f(k-1)F_k$ from \Cref{lem:tausv2} and, \cb{for each $i\in [k]$, define}
\begin{equation*}\label{eq:sec5_defPi}
\Pi_{i,k} \coloneq f(k-1)F_kf(k-2)F_{k-1}\cdots f(i-1)F_i=G_kG_{k-1}\cdots G_i,
\end{equation*}
so that iterating \eqref{eq:HkasNk}--\eqref{eq:sec5_fix1} until reaching $H_0=0$ and taking expected values leads to
\begin{equation}\label{ec:sec5_finalh}\E{H_k}\leq\sum_{i=1}^{k}\E{\Pi_{i,k}}.\end{equation}
\cb{In particular, $\E{H_1}=\E{\Pi_{1,1}}=\E{G_1}$.} An application of \Cref{lem:tausv2}(B) with $d=1$, $B=\{i,\ldots,k\}$, $I=\{k\}$, and $a_i=\cdots =a_k=1$ yields
\begin{equation*}\E{\Pi_{i,k}}\leq C_12^{i-k-1}\frac{f(k-1)}{f(k)}.\end{equation*}
Combined with~\eqref{ec:sec5_finalh},  we thus get \begin{equation*}\mathbb E[H_k] \le C_1 \frac{f(k-1)}{f(k)}\sum_{i=1}^k2^{i-k-1} \le C_1\frac{f(k-1)}{f(k)}.\end{equation*}
Since $\E{N_k(\tau_\infty)-1}\leq \E{H_k}$ by~\eqref{eq:Nkeq}, the assumption that $I_n=\sum_{j=0}^n \frac{f(j)}{f(j+1)}$ converges shows that the sum of $(N_k(\tau_{\infty})-1)_{k=1}^{\infty}$ has finite mean, as desired.

\vspace{1em}

We now move  to the proof of \Cref{thrm:supexp}(ii), that is, we want to show that if $I_n=\omega(1)$, then $\frac{n-d(T_n)}{I_n}\toinp 1$. Our proof relies on obtaining lower and upper bounds on $d(T_n)$ using concentration and second moment arguments. 
The lower and the upper bounds are derived separately.

\vspace{0.5em}
\noindent
\textbf{Upper bound.} We show that, for any $\varepsilon>0$, with high probability, 
\[d(T_n)<m_n \coloneq n-(1-\varepsilon)I_n.\] 
Note that $m_n=(1-o(1))n$, by applying Lemma~\ref{lemma:sums} to the sum $I_n$, with parameters $c_j=1$, $\eps_j=f(j-1)/f(j)$, and $M_1=1, M_2=n+1$. We work with the Cox process interpretation of the model introduced in Section~\ref{sec:Cox}, and define
\begin{equation*}d(t)\coloneq\sup\{k\in\N,\,N_k(t)>0\},\end{equation*}
which in this context corresponds to the depth of the tree, that is, $d(\tau_n)\overset d= d(T_n)$. Then,
\begin{equation}\label{eq:sec5_equivalence}n-d(\tau_n)=1+\sum_{k=1}^{d(\tau_n)}(N_k(\tau_n)-1).\end{equation}
To control the sum on the right-hand side, for each $k \ge 1$, we introduce the random variable $J_k$ defined as the indicator of the event ${\tau_{2,k}<\tau_{1,k+1}}$. 
Since $N_k(\cdot)$ is increasing, for each $k\leq d(\tau_n)-1$ such that $J_k=1$, it must hold that $N_k(\tau_n)\geq 2$, so in particular
\begin{equation}\label{eq:new5.10}
\sum_{k=1}^{d(\tau_n)}(N_k(\tau_n)-1)\,\geq\,\sum_{k=1}^{d(\tau_n)-1}J_k\ge \bigg(\sum_{k=1}^{d(\tau_n)}J_k\bigg)-1.
\end{equation}
Now, fix any $\varepsilon>0$. By combining \eqref{eq:sec5_equivalence} and \eqref{eq:new5.10} and recalling $m_n=n-(1-\eps)I_n$,
we obtain
\begin{equation}\label{eq:sec5_Jk}
\P{d(\tau_n)\geq n-(1-\varepsilon)I_n}\le \P{\sum_{k=1}^{m_n}J_k\leq (1-\varepsilon)I_n}.
\end{equation}
We now address the probability on the right-hand side by making two observations. 
First, given any realisation of the processes $N_1,\ldots,N_{k-1}$ until time $\tau_{1,k}$, within the time interval $[\tau_{1,k},\min\{\tau_{2,k},\tau_{1,k+1}\}]$, it holds that $N_{k-1}\geq 1$ and $N_k\equiv1$. Thus, it directly follows that, almost surely,
\begin{align*}
\condP{J_k=1}{(N_1(t),\ldots,N_{k-1}(t))_{t=0}^{\tau_{1,k}}}
&\geq \frac{f(k-1)}{f(k-1)+f(k)},\\
\Longrightarrow\quad \condP{J_k=1}{J_1,\ldots,J_{k-1}}&\geq \frac{f(k-1)}{f(k-1)+f(k)},
\end{align*}
where we used that $J_1,\ldots,J_{k-1}$ are measurable in terms of the mentioned trajectories. As a result, the sequence $(J_k)_{k=1}^{\infty}$ stochastically dominates a sequence $(J'_k)_{k=1}^{\infty}$ of independent Bernoulli random variables with $\P{J_k'=1}=\frac{f(k-1)}{f(k-1)+f(k)}$ for each $k$.
In particular, the latter sequence satisfies
\begin{equation}\label{eq:expJ'}
\mathbb{E}\left[\sum_{k=1}^{m_n} J'_k\right]=\sum_{k=1}^{m_n} \frac{f(k-1)}{f(k-1)+f(k)}=\sum_{k=1}^{m_n}\frac{f(k-1)}{f(k)}-\sum_{k=1}^{m_n}\frac{f(k-1)}{f(k)}\frac{f(k-1)}{f(k-1)+f(k)}.
\end{equation} 
The first sum on the right-hand side equals $I_{m_n-1}$. By Assumption~\eqref{ass:D}, we can apply Lemma~\ref{lemma:sums} to the second sum on the right-hand side, with parameters $c_j=f(j-1)/f(j)$, $\eps_j=f(j-1)/(f(j-1)+f(j))$, and $M_1=1$ and $M_2=m_n$. This yields 
\begin{equation*} 
\mathbb{E}\left[\sum_{k=1}^{m_n} J'_k\right]=(1-o(1))I_{m_n-1}.
\end{equation*} 
We then write 
\begin{equation*} 
I_{m_n-1}=I_n-\sum_{k=m_n+1}^{n+1} \frac{f(k-1)}{f(k)},
\end{equation*} 
and we again use Lemma~\ref{lemma:sums}, this time with $c_j=1$, $\eps_j=f(j-1)/f(j)$, $M_1=m_n+1$ and $M_2=n+1$ (where $M_1\leq M_2$ for all $n$ large since $M_2-M_1=n-m_n=(1-\eps)I_n=\omega(1)$) to finally obtain 
\begin{equation*} 
\mathbb{E}\left[\sum_{k=1}^{m_n} J'_k\right]=(1-o(1))I_n.
\end{equation*} 
For $\eps>0$ fixed, we define 
\be \label{eq:k0}
k_0=k_0(\eps)\coloneq \inf\big\{j\in\N\colon f(k-1)\leq \tfrac\eps4f(k)\text{ for all }k\geq j\big\},
\ee 
which is well-defined due to Assumption~\eqref{ass:D}. This yields
\begin{equation*}\frac{f(k-1)}{f(k)}\leq \bigg(1+\frac{\varepsilon}{4}\bigg)\frac{f(k-1)}{f(k-1)+f(k)}\qquad\text{for all}\qquad k\geq k_0.\end{equation*} 
Moreover, by Assumption~\eqref{ass:D} we also have $m_n=\omega(1)$ and thus $I_{m_n}=\omega(1)$. Hence, for large $n$,
\begin{align*}
I_n = \sum_{k=1}^{n+1} \frac{f(k-1)}{f(k)}
&\le \sum_{k=m_n+1}^{n+1} \frac{f(k-1)}{f(k)}
+\bigg(1+\frac{\eps}{4}\bigg)\sum_{k=1}^{m_n}\frac{f(k-1)}{f(k-1)+f(k)}\\
&\leq\frac{(n+1-m_n)\eps}{4}+\bigg(1+\frac{\varepsilon}{4}\bigg) I_{m_n-1}\le \frac{\eps}{4} I_n + \bigg(1+\frac{\varepsilon}{4}\bigg) I_{m_n-1}.
\end{align*}
Furthermore, using again that $m_n=\omega(1)$ and the latter chain of inequalities,
\begin{align*}
\sum_{k=1}^{m_n}\frac{f(k-1)}{f(k-1)+f(k)}
&\geq \bigg(1+\frac{\varepsilon}{3}\bigg)^{-1}I_{m_n-1}\\
&\ge 
\bigg(1+\frac{\varepsilon}{3}\bigg)^{-1}\bigg(1+\frac{\varepsilon}{4}\bigg)^{-1}\bigg(1-\frac{\eps}{4}\bigg)I_{n} \ge \bigg(1-\frac{6\eps}{7}\bigg)I_n.
\end{align*}
Thus, by Chernoff's bound, 
\begin{equation*}
\mathbb P\bigg(\sum_{k=1}^{m_n} J_k\le (1-\eps)I_n\bigg) \le \mathbb P\bigg(\sum_{k=1}^{m_n} J'_k\le (1-\eps)I_n\bigg) = \exp(-\Omega(\eps^2 I_n)) = o(1).
\end{equation*}
As the latter holds for every $\eps > 0$, this ensures the desired upper bound on $d(T_n)$.

\vspace{0.5em}
\noindent
\textbf{Lower bound.} We show that, for any $\varepsilon>0$, with high probability $d(T_n)>n-(1+\varepsilon)I_n$. We work again with the Cox process interpretation of the model introduced in Section~\ref{sec:Cox}, and begin the proof with a computation similar to the one in \eqref{eq:sec5_Jk}, namely
\begin{align*}
\P{d(\tau_n)\leq n-(1+\varepsilon)I_n}
&=\P{1+\sum_{k=1}^{d(\tau_n)}(N_k(\tau_n)-1)\geq (1+\varepsilon)I_n}\\
&\leq \P{1+\sum_{k=1}^{n}(N_k(\tau_{1,k}+F_k)-1)\geq (1+\varepsilon)I_n}\\&= \P{\sum_{k=1}^{n}H_k\geq (1+\varepsilon)I_n-1},
\end{align*}
where the first equality follows from \eqref{eq:sec5_equivalence}, the inequality follows from the almost sure bounds $d(\tau_n)\leq n$ and $\tau_n\leq\tau_{\infty}\leq \tau_{1,k}+F_k$, and the second equality uses the definition of the variables $H_k$ in~\eqref{eq:sec5_defh}. 
The result will follow from Chebyshev's inequality together with the following proposition.

\begin{proposition}\label{prop:sec5_moments}
\cb{Under} assumption~\eqref{ass:D} and $I_n=\omega(1)$, each of the following holds:
\begin{enumerate}
    \item[\emph{(i)}] $\displaystyle \E{\sum_{k=1}^n H_k} \geq (1+o(1)) I_n$.
    \item[\emph{(ii)}] $\displaystyle \E{\bigg(\sum_{k=1}^n H_k\bigg)^2} \leq (1+o(1)) I_n^2$.
\end{enumerate}
\end{proposition}
We delay the proof of Proposition \ref{prop:sec5_moments} and use it to conclude the proof of \Cref{thrm:supexp}(ii) first. 
Given $\varepsilon>0$, Proposition \ref{prop:sec5_moments}(ii) implies that $(1+\eps/2)I_n-1\geq \E{\sum_{k=1}^nH_k}$ for all large $n$, so 
\begin{align*}
\P{d(\tau_n)\leq n-(1+\varepsilon)I_n }
&\leq \P{\sum_{k=1}^nH_k\geq (1+\varepsilon)I_n-1}\\
&\leq\P{\sum_{k=1}^nH_k\geq \E{\sum_{k=1}^nH_k}+\frac\eps2I_n},
\end{align*}
and the probability on the right-hand side converges to zero by combining Chebyshev's inequality and the two parts of Proposition \ref{prop:sec5_moments}. As the latter holds for every $\eps > 0$, this ensures the desired lower bound on $d(\tau_n)$, or equivalently on $d(T_n)$. 
\end{proof}

We conclude this section with the proof of Proposition~\ref{prop:sec5_moments}.

\begin{proof}[Proof of Proposition \ref{prop:sec5_moments} assuming \Cref{lem:tausv2}] 
Recall the variables $J_k$ from the proof of the upper bound on $d(T_n)$, for which we also have
\begin{equation*}\sum_{k=1}^{n}H_k\,\geq\,\sum_{k=1}^{n}J_k.\end{equation*}
Using the previous analysis of the $J_k$ variables we then arrive at
\begin{equation*}\mathbb E\bigg[\sum_{k=1}^{n} H_k\bigg]\ge \mathbb E\bigg[\sum_{k=1}^{n} J_k\bigg]\ge \sum_{k=1}^{n} \frac{f(k-1)}{f(k-1)+f(k)} = (1+o(1))I_n,\end{equation*}
where the last equality uses Assumption~\eqref{ass:D}, Lemma~\ref{lemma:sums} (as applied  to~\eqref{eq:expJ'}, but with $M_2=n$),  and that $I_n=\omega(1)$. This gives the bound in (i). We now move to the proof of (ii), that is, we want to show that the expression
\begin{equation}\label{ec:sec5_opensum}\E{\bigg(\sum_{k=1}^nH_k\bigg)^2}=\sum_{k=1}^n\E{ H_k^2}\,+\,2\sum_{\ell=2}^n\sum_{k=1}^{\ell-1}\E{ H_kH_{\ell}}
\end{equation}
is at most $(1+o(1)) I_n^2$. 
We first bound $\E{H_k^2}$ for a fixed value of $k$, similarly to what was done for $\E{H_k}$ in~\eqref{eq:sec5_fix1}--\eqref{ec:sec5_finalh}. Similarly to that case, we analyse $\condE{H_k^2}{\mathcal{F}_k^{\infty}}$, where again $\mathcal{F}_k^{\infty}$ denotes the $\sigma$-algebra generated by the random variables $(E_i)_{i=k}^{\infty}$. As in \Cref{thrm:supexp}(i), conditioning on $N_{k-1}$ (and on $F_k$) and using the strong Markov property gives that $H_k$ is Poisson distributed with parameter $\lambda\coloneq\int_{\tau_{1,k}}^{\tau_{1,k}+F_k}f(k-1)N_{k-1}(s)\dd s$, so that 
\begin{equation*}\condE{H_k^2}{\mathcal{F}_k^{\infty}}=\condE{\lambda^2}{\mathcal{F}_k^{\infty}}+\condE{\lambda}{\mathcal{F}_k^{\infty}}.\end{equation*}
For $\condE{\lambda^2}{\mathcal{F}_k^{\infty}}$, we use that $\tau_{1,k}\leq \tau_{1,k-1}+E_{k-1}$ by Lemma~\ref{cl:couple} and the fact that $N_{k-1}$ is increasing in the same way as in the analysis of $\condE{H_k}{\mathcal{F}_k^{\infty}}$ to obtain
\begin{align*}
\condE{\lambda^2}{\mathcal{F}_k^{\infty}}
&\leq \condE{\condE{\left(\int_{\tau_{1,k-1}+E_{k-1}}^{\tau_{1,k-1}+E_{k-1}+F_k}f(k-1)N_{k-1}(s)\dd s\right)^2}{\mathcal{F}_{k-1}^{\infty}}}{\mathcal{F}_k^{\infty}}\\
&\leq \condE{\condE{(f(k-1) F_k N_{k-1}(\tau_{1,k-1}+F_{k-1}))^2}{\mathcal{F}_{k-1}^{\infty}}}{\mathcal{F}_k^{\infty}}.
\end{align*}
Using the identity $x^2=(x-1)^2+2(x-1)+1$ with $x=N_{k-1}(\tau_{1,k-1}+F_{k-1})$ \cb{(so that $x-1=H_{k-1}$)} and the fact that $N_{k-1}(\tau_{1,k-1})=1$ shows that
\[N_{k-1}(\tau_{1,k-1}+F_{k-1})^2=H_{k-1}^2+2H_{k-1}+1,\end{equation*}
which further yields
\begin{equation*}\condE{\lambda^2}{\mathcal{F}_k^{\infty}}\leq\condE{\condE{(f(k-1)F_k)^2(H_{k-1}^2+2H_{k-1}+1)}{\mathcal{F}_{k-1}^{\infty}}}{\mathcal{F}_k^{\infty}}.\end{equation*}
Further, since $\condE{\lambda}{\mathcal{F}_k^{\infty}}=\condE{H_k}{\mathcal{F}_{k}^{\infty}}$ and $F_k$ is measurable with respect to $\mathcal{F}_{k-1}^{\infty}\supseteq\mathcal{F}_{k}^{\infty}$,
\begin{equation}\label{eq:sec5_fix2}
\condE{H_k^2}{\mathcal{F}_{k}^{\infty}}\leq \condE{(f(k-1)F_k)^2\condE{H_{k-1}^2}{\mathcal{F}_{k-1}^{\infty}}}{\mathcal{F}_{k}^{\infty}}+\condE{\condE{b_{k}}{\mathcal{F}_{k-1}^{\infty}}}{\mathcal{F}_k^{\infty}} 
\end{equation}
where 
\begin{equation*}b_{k}\coloneq (f(k-1)F_k)^2(2H_{k-1}+1)+H_k.\end{equation*}
Iterating \eqref{eq:sec5_fix2} and recalling the inclusion $\mathcal{F}_k^{\infty} \subseteq \mathcal{F}_{k-1}^{\infty}\subseteq \ldots\subseteq \mathcal{F}_0^{\infty}$, the fact that $H_0\equiv0$ and the notation $\Pi_{i,k}=f(k-1)F_kf(k-2)F_{k-1}\cdots f(i-1)F_i$ (with $\Pi_{k+1,k}=1$ by convention), 
we obtain 
\begin{equation}\label{eq:sec5_prefinalg}\condE{H_k^2}{\mathcal{F}_{k}^{{\infty}}}\leq \sum_{i=1}^{k}   \condE{b_i\Pi_{i+1,k}^2}{\mathcal{F}_{k}^{{\infty}}},\end{equation}
and taking expectations on both sides yields
\begin{equation*}
\E{H_k^2}\leq \sum_{i=1}^{k}   \E{b_i\Pi_{i+1,k}^2} = \sum_{i=1}^{k}   \E{2H_{i-1}\Pi_{i,k}^2}+\sum_{i=1}^{k}   \E{\Pi_{i,k}^2}+\sum_{i=1}^{k}   \E{H_i\Pi_{i+1,k}^2}.
\end{equation*}
Further, using~\eqref{eq:new5.5}--\eqref{eq:sec5_fix1} and that $\Pi_{i+1,k}$ is $\cF_i^{\infty}$-measurable, we obtain that
\begin{equation*}\label{eq:new5.16}
\begin{split}
\mathbb E[H_i\Pi_{i+1,k}^2] 
&= \mathbb E[\mathbb E[H_i\mid \cF_i^{\infty}] \Pi_{i+1,k}^2]\\
&\le \mathbb E[\mathbb E[\mathbb E[H_{i-1}\mid \cF_{i-1}^{\infty}]\mid \cF_i^{\infty}]\Pi_{i,i}\Pi_{i+1,k}^2]+\mathbb E[\Pi_{i,i}\Pi_{i+1,k}^2]\\
&= \mathbb E[H_{i-1} \Pi_{i,i}\Pi_{i+1,k}^2]+\mathbb E[\Pi_{i,i}\Pi_{i+1,k}^2].
\end{split}
\end{equation*}
Iterating the latter argument, and using that $H_0\equiv 0$ and $\Pi_{j,i}\Pi_{i+1,k}=\Pi_{j,k}$ for all $j\le i<k$ gives
\begin{equation*}\label{eq:new5.17}
\mathbb E[H_i\Pi_{i+1,k}^2]\le \sum_{j=1}^i \mathbb E[\Pi_{j,i}\Pi_{i+1,k}^2] \cb{=\sum_{j=1}^i \mathbb E[\Pi_{j,k}\Pi_{i+1,k}].}    
\end{equation*}
Substituting the previous inequality into \eqref{eq:sec5_prefinalg} shows that
\begin{equation}\label{leq:sec5_finalH2}\E{H_k^2}\leq \cb{2\sum_{i=2}^{k}\sum_{j=1}^{i-1}\E{\Pi_{j,k}\Pi_{i,k}}+\sum_{i=1}^{k}\E{\Pi_{i,k}^2}+\sum_{i=1}^{k}\sum_{j=1}^{i}\E{\Pi_{j,k}\Pi_{i+1,k}}.}
\end{equation}
Recalling the notation $G_k=f(k-1)F_k$ introduced in \Cref{lem:tausv2}, we observe that all these terms correspond to expressions of the form
$\prod_{j=i}^k G_j^{a_j}$, with $a_j\in\{1,2\}$. Moreover, the last factor in each term is $G_k^2$, the sole exception being the last sum in \eqref{leq:sec5_finalH2} which,
in the case $i=k$, yields a term of the form
$\sum_{j=1}^k \E{\Pi_{j,k}}$, whose last factor is $G_k$. Hence, applying \Cref{lem:tausv2}(B) with $d=1$ and $I=\{k\}$, gives
\begin{align*}
\E{H_k^2}\leq C_1\bigg(\left(\frac{f(k-1)}{f(k)}\right)^2+&\left(\frac{f(k-1)}{f(k)}\right)\bigg)\\
&\cdot\left(2\sum_{i=2}^{k}\sum_{j=1}^{i-1}2^{j-k-1}+\sum_{i=1}^{k}2^{i-k-1}+\sum_{i=1}^{k}\sum_{j=1}^{i}2^{j-k-1}\right).\end{align*}
As each of the three sums on the right is uniformly bounded as a function of $k$, it follows that there is an absolute constant $C>0$ 
such that
\begin{equation*}\E{H_k^2}\leq C\left[\left(\frac{f(k-1)}{f(k)}\right)^2+\left(\frac{f(k-1)}{f(k)}\right)\right]\end{equation*}
and, in particular, \begin{equation}\label{eq:boundsquared}
\sum_{k=1}^n\E{H_k^2}=O\left(\left(\sum_{k=1}^n\left(\frac{f(\cb{k-1})}{f(\cb{k})}\right)^2\right)+I_n\right)=O(I_n)=o\left(I_n^2\right),
\end{equation}
where we use Lemma~\ref{lemma:sums} with $c_j=\eps_j=f(j)/f(j+1)$, and $M_1=1$ and $M_2=n$ to bound the sum from above. 

We now turn to estimating the terms $\E{H_kH_{\ell}}$ with $1\leq k<\ell$ appearing in~\eqref{ec:sec5_opensum}. 
We first analyse $\condE{H_kH_{\ell}}{\mathcal{F}_{\ell}^{\infty}}$ by following arguments already used in the study of $\condE{H_k}{\mathcal{F}_k^{\infty}}$. 
More precisely, conditioning on $N_{k},N_{k+1},\ldots,N_{\ell-1}$, we use the strong Markov property as well as the fact that $N_{\ell-1}$ is increasing as in~\eqref{eq:HkasNk}--\eqref{eq:new5.5} to obtain

\begin{align*}
    \condE{H_kH_{\ell}}{\mathcal{F}_{\ell}^{\infty}}&=\condE{ H_k\int_{\tau_{1,\ell}}^{\tau_{1,\ell}+F_{\ell}}f(\ell-1)N_{\ell-1}(s)ds}{\mathcal{F}_{\ell}^{\infty}}\\&\leq \condE{ H_kf(\ell-1)F_{\ell}}{\mathcal{F}_{\ell}^{\infty}}+\condE{ H_kf(\ell-1)F_{\ell}H_{\ell-1}}{\mathcal{F}_{\ell}^{\infty}}.
\end{align*}
By iterating this procedure, we arrive at the upper bound
\begin{align*}
\condE{H_kH_{\ell}}{\mathcal{F}_{\ell}^{\infty}}
&\leq \bigg(\sum_{j=k+1}^{\ell}\condE{H_k\Pi_{j,\ell}}{\mathcal{F}_{\ell}^{\infty}}\bigg)+\condE{H_k^2\Pi_{k+1,\ell}}{\mathcal{F}_{\ell}^{\infty}}\\
&= \bigg(\sum_{j=k+1}^{\ell}\condE{\condE{H_k}{\mathcal{F}_{k}^{\infty}}\Pi_{j,\ell}}{\mathcal{F}_{\ell}^{\infty}}\bigg)+\condE{\condE{H_k^2}{\mathcal{F}_{k}^{\infty}}{\Pi_{k+1,\ell}}}{\mathcal{F}_{\ell}^{\infty}},
\end{align*}
where the last equality used that $(\Pi_{j,\ell})_{j=k+1}^{\ell}$ are all $\cF_k^{\infty}$-measurable and $\cF_\ell^\infty \subseteq \cF_k^\infty$ as $k<\ell$. 
Now, repeating the argument from~\eqref{eq:HkasNk}--\eqref{eq:sec5_fix1} for the first term, the argument from~\eqref{eq:sec5_fix2}--\eqref{leq:sec5_finalH2} for the second term and taking expectations on both sides leads to
\begin{align}\label{eq:appendix_crossed}\E{H_kH_\ell}\leq{}& \sum_{i=1}^k\sum_{j=k+1}^{\ell}\E{\Pi_{i,k}\Pi_{j,\ell}}+2\sum_{i=2}^{k}\sum_{j=1}^{i-1}\E{\Pi_{j,i-1}\Pi_{i,k}^2\Pi_{k+1,\ell}} \\ \nonumber&+\sum_{i=1}^{k}\E{\Pi_{i,k}^2\Pi_{k+1,\ell}}+\sum_{i=1}^{k}\sum_{j=1}^{i}\E{\Pi_{j,i}\Pi_{i+1,k}^2\Pi_{k+1,\ell}}.\end{align}
Taking the sum over $\ell$ from $2$ to $n$ and then over $k$ from $1$ to $\ell-1$ gives
\be\label{eq:doubleH}
\sum_{\ell=2}^n\sum_{k=1}^{\ell-1}\E{H_kH_\ell}\leq \sum_{\ell=2}^n\sum_{k=1}^{\ell-1} \left(\frac{f(k-1)}{f(k)}\frac{f(\ell-1)}{f(\ell)}+\delta_{k,\ell}\right),
\ee 
where $\delta_{k,\ell}$ is defined as
\begin{align*}
\delta_{k,\ell}&\coloneq\underbrace{\E{G_kG_{\ell}}-\frac{f(k-1)}{f(k)}\frac{f(\ell-1)}{f(\ell)}}_{\coloneq\delta^{(1)}_{k,\ell}}+\underbrace{\sum_{i=1}^{k-1}\E{\Pi_{i,k}G_{\ell}}+\sum_{i=1}^k\sum_{j=k+1}^{\ell-1}\E{\Pi_{i,k}\Pi_{j,\ell}}+\sum_{j=1}^{k}\E{\Pi_{j,\ell}}}_{\coloneq\delta^{(2)}_{k,\ell}}\\
&+\underbrace{2\sum_{i=2}^{k}\sum_{j=1}^{i-1}\E{\Pi_{j,i-1}\Pi_{i,k}^2\Pi_{k+1,\ell}}+\sum_{i=1}^{k}\E{\Pi_{i,k}^2\Pi_{k+1,\ell}}+\sum_{i=1}^{k-1}\sum_{j=1}^{i}\E{\Pi_{j,i}\Pi_{i+1,k}^2\Pi_{k+1,\ell}}}_{\coloneq\delta^{(3)}_{k,\ell}}.
\end{align*}
Note that the term $\delta^{(3)}_{k,\ell}$ represents the last three sums in~\eqref{eq:appendix_crossed} except for the part for $i=k$ and $j\in [k]$ of the last double sum, which appears as the last term in $\delta^{(2)}_{k,\ell}$, and the rest of $\delta^{(1)}_{k,\ell}+\delta^{(2)}_{k,\ell}$ corresponds to a decomposition of the first double sum. We observe that all terms in the definition of $\delta_{k,\ell}$ can be expressed as products of the type
$\prod_{j\in B} G_j^{a_j}$ for suitable choices of $B$ and sequences
$(a_j)_{j\in B}$ taking values in $\{1,2\}$. The remainder of the proof then consists in applying \Cref{lem:tausv2}(B) to show that
\begin{equation}\label{eq:5.4bgoal}
\sum_{\ell=2}^{n}\sum_{k=1}^{\ell-1}\delta_{k,\ell} = o(I_n^2).  
\end{equation}

We argue that this is enough to conclude. Note that 
\be \label{eq:In2bound}
\sum_{\ell=2}^n \sum_{k=1}^{\ell-1}\frac{f(k-1)}{f(k)}\frac{f(\ell-1)}{f(\ell)}\leq \sum_{\ell=1}^{n+1}\sum_{k=1}^\ell \frac{f(k-1)}{f(k)}\frac{f(\ell-1)}{f(\ell)}=\bigg(\frac12+o(1)\bigg)I_n^2,
\ee 
\cb{where the last equality uses that $I_n=\omega(1)$ and
\[\sum_{\ell=1}^{n+1}\sum_{k=1}^\ell \frac{f(k-1)}{f(k)}\frac{f(\ell-1)}{f(\ell)} = \frac{1}{2}\bigg(I_n^2 + \sum_{\ell=1}^{n+1}\bigg(\frac{f(k-1)}{f(k)}\bigg)^2\bigg) = \frac{1}{2}I_n^2+o(I_n^2).\]}
Together with~\eqref{eq:doubleH}, this gives 
\begin{equation*}2\sum_{\ell=2}^n \sum_{k=1}^{\ell-1}\E{H_kH_\ell}\leq (1+o(1))I_n^2.\end{equation*}
Further, combined with~\eqref{ec:sec5_opensum} and~\eqref{eq:boundsquared}, this yields \Cref{prop:sec5_moments}(ii).

To show~\eqref{eq:5.4bgoal} note that, by applying \cb{\Cref{lem:tausv2}(A)} to $\delta^{(1)}_{k,\ell}$, we obtain
\begin{equation}\label{eq:delta1}
\delta_{k,\ell}^{(1)}\leq C\frac{f(k-1)}{f(k)}\frac{f(\ell-1)}{f(\ell)}\left(\frac{f(k)}{f(k+1)}+\frac{f(\ell)}{f(\ell+1)}\right).
\end{equation}

As a next step, we show that $\sum_{\ell=2}^{n}\delta_{\ell-1,\ell} = o(I_n^2)$; this is necessary since the bounds for $\delta_{k,\ell}$ for $k=\ell-1$ have to be treated slightly differently from the general case. 
When $k=\ell-1$, we observe that each of the products $\prod_{j\in B} G_j^{a_j}$ appearing within expectations in $\delta_{\ell-1,\ell}^{(2)}$ and $\delta_{\ell-1,\ell}^{(3)}$ contains a factor $G_{\ell}$, 
and either a factor $G_{\ell-1}$ (in the case of $\delta_{\ell-1,\ell}^{(2)}$), or a factor $G_{\ell-1}^2$ (in the case of $\delta_{\ell-1,\ell}^{(3)}$). 
Thus, applying \Cref{lem:tausv2}(B) with $d=2$ and $I=\{\ell-1,\ell\}$ shows that, for some constant $C_2'=C_2'(f)$ \cb{and all $\ell\in \N$,}
{\small\begin{align*}
    \delta_{\ell-1,\ell}^{(2)}&\leq C_2\frac{f(\ell-2)}{f(\ell-1)}\frac{f(\ell-1)}{f(\ell)}\left[\sum_{i=1}^{\ell-2}2^{i-\ell-1}+\sum_{j=1}^{\ell-1}2^{j-\ell-1}\right]\leq C_2'\frac{f(\ell-2)}{f(\ell-1)}\frac{f(\ell-1)}{f(\ell)},\\
    \delta_{\ell-1,\ell}^{(3)}&\leq C_2\bigg(\frac{f(\ell-2)}{f(\ell-1)}\bigg)^2\frac{f(\ell-1)}{f(\ell)}\left[2\sum_{i=1}^{\ell-1}\sum_{j=1}^{i-1}2^{j-\ell-1}+\sum_{i=1}^{\ell-1}2^{i-\ell-1}+\sum_{i=1}^{\ell-2}\sum_{j=1}^{i}2^{j-\ell-1}\right]\\
    &\leq C_2'\frac{f(\ell-2)}{f(\ell-1)}\frac{f(\ell-1)}{f(\ell)},
\end{align*}}
where we used that the sums in the square brackets are uniformly bounded in $\ell$ and also $f(n-1) = o(f(n))$.
Using further that $f(n-1)/f(n)$ is bounded, we deduce from~\eqref{eq:delta1} that, for large enough $C_2'$ \cb{and all $\ell\in\N$},
\begin{equation*}\delta_{\ell-1,\ell}^{(1)}\leq C_2'\frac{f(\ell-2)}{f(\ell-1)}\frac{f(\ell-1)}{f(\ell)}.\end{equation*}
Hence, using all three bounds, we arrive at
\begin{equation*}\delta_{\ell-1,\ell}\leq 3C_2'\frac{f(\ell-2)}{f(\ell-1)}\frac{f(\ell-1)}{f(\ell)},\end{equation*}
which yields
\begin{equation*}\sum_{\ell=2}^{n}\delta_{\ell-1,\ell}\leq 3C_2'\sum_{\ell=2}^n\frac{f(\ell-2)}{f(\ell-1)}\frac{f(\ell-1)}{f(\ell)}=o(I_n)=o(I_n^2),\end{equation*}
where we use Lemma~\ref{lemma:sums} with $c_j=f(j-1)/f(j)$, $\eps_j=c_{j-1}$, and $M_1=2, M_2=n$ to bound the sum from above. We thus conclude that the contribution of the terms of the form $\delta_{\ell-1,\ell}$ is negligible. 
Next, we assume that $\ell\geq k+2$ and study the contributions coming from $\delta^{(1)}_{k,\ell}$, $\delta^{(2)}_{k,\ell}$ and $\delta^{(3)}_{k,\ell}$ separately. 

For the contribution coming from $\delta_{k,\ell}^{(1)}$, we directly obtain from~\eqref{eq:delta1} that
\begin{align}
\sum_{k=1}^{n-2}\sum_{\ell=k+2}^{n} \delta_{k,\ell}^{(1)}
&\leq Ck_0 \sum_{\ell=1}^n \frac{f(\ell-1)}{f(\ell)}\nonumber\\
&\hspace{1cm}+ C\sum_{k=k_0+1}^{n-2}
\sum_{\ell=k+2}^{n}
\frac{f(k-1)}{f(k)}\frac{f(\ell-1)}{f(\ell)}\Big(\frac{f(k)}{f(k+1)}+\frac{f(\ell)}{f(\ell+1)}\Big)\label{eq:middle}\\
&\le C k_0 I_n+\cb{2\eps C}\sum_{\ell=2}^{n}\sum_{k=1}^{\ell-2}\frac{f(k-1)}{f(k)}\frac{f(\ell-1)}{f(\ell)}\le 2\eps C I_n^2\nonumber    
\end{align}
\cb{for all large enough $n$, where we use $k_0=k_0(\eps)$, as in~\eqref{eq:k0},} and~\eqref{eq:In2bound}. \cb{Note the crucial fact that the constant $C$ above does not depend on $\eps$ but only on the function $f$.}
 
    For the contribution coming from $\delta_{k,\ell}^{(2)}$, we observe first that the first sum 
     \begin{equation*}\sum_{i=1}^{k-1}\E{\Pi_{i,k}G_{\ell}}\end{equation*}
    is either zero (if $k=1$), or otherwise contains the factor $G_{k-1} G_{k} G_{\ell}$. Applying \Cref{lem:tausv2}(B) with $d=3$ and $I=\{k-1,k,\ell\}$ gives
     \begin{align*}\sum_{i=1}^{k-1}\E{\Pi_{i,k}G_{\ell}}
     &\leq C_3 \frac{f(k-2)}{f(k-1)}\frac{f(k-1)}{f(k)}\frac{f(\ell-1)}{f(\ell)}\sum_{i=1}^{k-1}2^{i-k-2}\\
     &\leq C_3'\frac{f(k-2)}{f(k-1)}\frac{f(k-1)}{f(k)}\frac{f(\ell-1)}{f(\ell)}.\end{align*}
     Similarly, as $k\leq \ell-2$, each of the products appearing in the other sums in $\delta_{k,\ell}^{(2)}$ contains the factor $G_kG_{\ell-1}G_{\ell}$. By \Cref{lem:tausv2}(B) with $d=3$ and $I=\{k,\ell-1,\ell\}$, the remainder of $\delta_{k,\ell}^{(2)}$ is at most
     \begin{align*}
     C_3\frac{f(k-1)}{f(k)}\frac{f(\ell-2)}{f(\ell-1)}\frac{f(\ell-1)}{f(\ell)}\bigg[\sum_{i=1}^k\sum_{j=k+1}^{\ell-1}2^{i-k+j-\ell-2}&+\sum_{j=1}^{k}2^{j-\ell-1}\bigg]\\
     &\leq C_3'\frac{f(k-1)}{f(k)}\frac{f(\ell-2)}{f(\ell-1)}\frac{f(\ell-1)}{f(\ell)}.
     \end{align*}
     By summing these bounds over $k, \ell$ \cb{and using a bound similar to~\eqref{eq:middle},} we obtain
     \begin{align*}
     \sum_{\ell=2}^{n}\sum_{k=1}^{\ell-2}\delta_{k,\ell}^{(2)}&\leq C_3'\sum_{\ell=2}^{n}\sum_{k=1}^{\ell-2}\frac{f(k-1)}{f(k)}\frac{f(\ell-1)}{f(\ell)}\bigg(\frac{f(k-2)}{f(k-1)}
     +\frac{f(\ell-2)}{f(\ell-1)}\bigg)\cb{\le 2\eps C_3' I_n^2},    
     \end{align*}
     \cb{where, crucially, $C_3'$ depends only on $f$ but not on $\eps$.} 
    
     Lastly, for $\delta_{k,\ell}^{(3)
     }$, using that $k\leq\ell-2$, each of the products in $\delta_{k,\ell}^{(3)}$ contains the factor $G_k^2G_{\ell-1}G_{\ell}$ (if $k=1$, the first and the third sums are zero). By \Cref{lem:tausv2}(B) with $d=3$ and $I=\{k,\ell-1,\ell\}$,
     \begin{align*}
     \delta_{k,\ell}^{(3)}
     &\leq C_3\left(\frac{f(k-1)}{f(k)}\right)^2\frac{f(\ell-2)}{f(\ell-1)}\frac{f(\ell-1)}{f(\ell)}\left[2\sum_{i=1}^{k}\sum_{j=1}^{i-1}2^{j-\ell-1}+\sum_{i=1}^{k}2^{i-\ell-1}+\sum_{i=1}^{k-1}\sum_{j=1}^{i}2^{j-\ell-1}\right]\\
     &\leq C_3'\left(\frac{f(k-1)}{f(k)}\right)^2\frac{f(\ell-2)}{f(\ell-1)}\frac{f(\ell-1)}{f(\ell)}.
     \end{align*}
     Once again, by summing these bounds over $k,\ell$, we obtain
     \begin{align*}
     \sum_{\ell=2}^{n}\sum_{k=1}^{\ell-2}\delta_{k,\ell}^{(3)}
     &\leq C_3'\sum_{\ell=2}^{n}\sum_{k=1}^{\ell-2}\frac{f(k-1)}{f(k)}\frac{f(\ell-1)}{f(\ell)}\Big(\frac{f(k-1)}{f(k)}\frac{f(\ell-2)}{f(\ell-1)}\Big)\cb{\le 2\eps C_3' I_n^2.}
     \end{align*}
\cb{By summing the contributions coming from $\delta_{k,\ell}^{(1)}$, $\delta_{k,\ell}^{(2)}$ and $\delta_{k,\ell}^{(3)}$, we obtain an upper bound of $(2C+4C_3')\eps I_n^2 + o(I_n^2)$. Since the latter holds for every $\eps > 0$ and the constants $C,C_3'$ depend only on $f$ but not on $\eps$,} \eqref{eq:5.4bgoal} follows, which completes the proof.
\end{proof}

\subsection{\texorpdfstring{Proof of Lemma \ref{lem:tausv2}}{Proof of Lemma 5.2}}\label{sec:appendix}
We conclude the section by proving Lemma~\ref{lem:tausv2}.
Recall the variables $G_k=f(k-1)F_k$ and the fact that the exponential variable $E_k$ with rate $f(k)$ has moments $\E{E_k^n}=n!/f(k)^n$ for every $n\in\N_0$. We proceed by establishing a couple of intermediate estimates which allow us to control the expectations in the statement of the lemma. 
\begin{proposition}\label{prop:appendix_upperF}
For every weight function $f$ satisfying Assumption~\eqref{ass:D}, there is $k_0=k_0(f)\ge 1$ and $C=C(f)>0$ such that, for every $k\ge k_0$ and $\ell\in\N_0$, as well as for every \cb{$1\leq k<k_0$} and \cb{$1\leq \ell\leq 2(k_0+1)$},
\begin{equation}\label{eq:appendix_boundF}\E{G_k^\ell}\leq \ell!\left(\frac{f(k-1)}{f(k)}\right)^{\ell}\left(1+C\frac{f(k)}{f(k+1)}\right).\end{equation}
\end{proposition}
\begin{proof}
Recall $k_0=k_0(\eps)$ from~\eqref{eq:k0}, and take $\eps=0.4$, so that $f(k-1)\leq 0.1f(k)$ for all $k\geq k_0$.
For positive integers $k,\ell$, define $(y_k,\ldots)_\ell\coloneq \{(y_i)_{i=k}^\infty \colon \sum_{i=k}^\infty y_i=\ell\}$ .
As $F_k=\sum_{j=k}^{\infty} E_j$, we have
\begin{equation*}F_k^{\ell}=\left(\sum_{j=k}^{\infty} E_j\right)^{\ell}=\sum_{(y_k,\ldots)_{\ell}} \frac{\ell!}{y_k!y_{k+1}!\cdots}\prod_{j=k}^\infty E_j^{y_j}.\end{equation*}
Multiplying by $f(k-1)^{\ell}$ and taking expectations yields
\begin{align*}\E{G_k^{\ell}}&=\sum_{(y_k,\ldots)_{\ell}} \frac{\ell!f(k-1)^{\ell}}{y_k!y_{k+1}!\cdots}\prod_{j=k}^\infty \E{E_j^{y_j}}=\sum_{(y_k,\ldots)_{\ell}} \frac{\ell!f(k-1)^{\ell}}{y_k!y_{k+1}!\cdots}\prod_{j=k}^\infty \frac{y_j!}{f(j)^{y_j}}\\&=\sum_{(y_k,\ldots)_{\ell}} \ell!f(k-1)^{\ell}\prod_{j=k}^\infty \frac{1}{f(j)^{y_j}}=\ell!\left(\frac{f(k-1)}{f(k)}\right)^{\ell}\sum_{(y_k,\ldots)_{\ell}} \prod_{j=k}^\infty \left(\frac{f(k)}{f(j)}\right)^{y_j},\end{align*}
where the last equality uses that $\sum_{j=k}^{\infty}y_j=\ell$.  Further, set
\begin{equation*}A_{k,\ell}\coloneq  \sum_{(y_k,\ldots)_\ell}\prod_{j= k}^{\infty}\bigg(\frac{f(k)}{f(j)}\bigg)^{y_j}.\end{equation*}
Our goal is to prove that there exists a constant $C=C(f) > 0$ such that, for all $k\ge k_0$ and $\ell\ge 0$, as well as for all $k<k_0$ and all $\ell\leq 2(k_0+1)$, 
\begin{equation}\label{eq:goalAkl}
A_{k,\ell}\leq 1+C\frac{f(k)}{f(k+1)}.    
\end{equation}
To this end, we first show that \cb{for all $\ell\in\N$},
\begin{equation} \label{eq:appendix_Aeq}
A_{k,\ell} = \sum_{j=k}^{\infty} \bigg(\frac{f(k)}{f(j)}\bigg)^\ell A_{j,\ell-1}.
\end{equation} 
Indeed, for any sequence $(y_k,\ldots)_\ell$, denote by $r$ the index of the first non-zero entry. 
That is, $y_k=y_{k+1}=\cdots =y_{r-1}=0$ and $y_r\ge 1$: we denote this family of sequences by $(y_k,\ldots)_{\ell,r}$. 
By splitting the sum over all sequences $(y_k,\ldots)_\ell$ into sums over $(y_k,\ldots)_{\ell,r}$ with $r\ge k$, we obtain
\begin{equation*}
A_{k,\ell}=\sum_{r=k}^\infty \sum_{(y_k,\ldots)_{\ell,r}} \prod_{j=k}^\infty \bigg(\frac{f(k)}{f(j)}\bigg)^{y_j}=\sum_{r=k}^\infty \sum_{(y_k,\ldots)_{\ell,r}} \prod_{j=r}^\infty \bigg(\frac{f(k)}{f(j)}\bigg)^{y_j}. 
\end{equation*}
Now, for fixed $r$ and any sequence in $(y_k,\ldots)_{\ell,r}$, we define a new sequence $y'_r,y'_{r+1},\ldots$ by setting $y'_r = y_r - 1$ and
$y'_j = y_j$ for all $j \neq r$. Observe that these sequences are precisely those counted in $(y'_r,\ldots)_{\ell-1}$.
Using this observation and the fact that $\sum_{j=r}^{\infty} y_j = \ell$, we deduce that
\begin{align*}
A_{k,\ell}&=\sum_{r=k}^\infty \sum_{(y_k,\ldots)_{\ell,r}} \left(\frac{f(k)}{f(r)}\right)^{\ell}\prod_{j=r}^\infty \bigg(\frac{f(r)}{f(j)}\bigg)^{y_j}\\
&=\sum_{r=k}^\infty \sum_{(y_r',\ldots )_{\ell-1}} \left(\frac{f(k)}{f(r)}\right)^{\ell} \prod_{j=r}^\infty \bigg(\frac{f(r)}{f(j)}\bigg)^{y_j'}=\sum_{r=k}^\infty \left(\frac{f(k)}{f(r)}\right)^{\ell} A_{r,\ell-1}, 
\end{align*}
which shows~\eqref{eq:appendix_Aeq}.
We now move to the proof of \eqref{eq:goalAkl} when $k\geq k_0$ and $\ell\in\N_0$. We observe that, by our choice of $k_0$, for all $\ell\geq 1$ and $k\geq k_0$, it holds that
\begin{equation}\label{eq:appendix_aux1} 
\sum_{j=1}^\ell \bigg(\frac{f(k)}{f(k+1)}\bigg)^j\leq\frac{10}{9}\frac{f(k)}{f(k+1)}\quad\text{ and }\quad
\sum_{j=k+1}^\infty \bigg(\frac{f(k)}{f(j)}\bigg)^{\ell} \leq  \frac{10}{9}\bigg(\frac{f(k)}{f(k+1)}\bigg)^\ell.
\end{equation}
We prove by induction on $\ell$ that, for every $k\ge k_0$ and $\ell\ge 0$,
\begin{equation}\label{eq:3}
A_{k,\ell}\leq 1+3\frac{f(k)}{f(k+1)}.
\end{equation}
For $\ell=0$, observe that $A_{k,0}=1$ for every $k\geq k_0$. 
Then, fix $\ell\in \N$ and assume that \eqref{eq:3} holds for every $A_{k,j}$ with $k\geq k_0$ and $j\leq\ell-1$, so in particular, by the choice of $k_0$, every such term is bounded by $3/2$. Using \eqref{eq:appendix_Aeq} and \eqref{eq:appendix_aux1}, we obtain the bound
\begin{align*}
A_{k,\ell}
&=A_{k,\ell-1}+\sum_{j=k+1}^{\infty} \bigg(\frac{f(k)}{f(j)}\bigg)^\ell A_{j,\ell-1}\\
&\leq A_{k,\ell-1}+\frac{3}{2}\sum_{j=k+1}^{\infty} \bigg(\frac{f(k)}{f(j)}\bigg)^\ell\leq A_{k,\ell-1}+2\bigg(\frac{f(k)}{f(k+1)}\bigg)^\ell.
\end{align*}
By iterating this procedure, we obtain
\begin{equation*} 
A_{k,\ell}\leq A_{k,\ell-1}+2 \Big(\frac{f(k)}{f(k+1)}\Big)^\ell\leq \cdots \leq A_{k,0}+2\sum_{j=1}^\ell \Big(\frac{f(k)}{f(k+1)}\Big)^j=1+2\sum_{j=1}^\ell \Big(\frac{f(k)}{f(k+1)}\Big)^j,
\end{equation*} 
so using \eqref{eq:appendix_aux1} once again gives
\begin{equation*}A_{k,\ell}\leq 1+3\frac{f(k)}{f(k+1)},\end{equation*}
which yields~\eqref{eq:goalAkl} in the first regime of interest when $C\geq 3$, as desired. To obtain the upper bound in~\eqref{eq:goalAkl} when $1\leq k<k_0$ and $\ell\leq 2(k_0+1)$, we observe that $A_{j,0}=1$ for any $j\in\N_0$  and that Assumption~\eqref{ass:D} is satisfied, so that the recursion in~\eqref{eq:appendix_Aeq} yields that 
\begin{equation*} 
A\coloneq \sup_{1 \le k<k_0, \ell\leq 2(k_0+1)}A_{k,\ell}<\infty.
\end{equation*} 
We can thus choose $C$ sufficiently large such that the inequality in~\eqref{eq:goalAkl} is satisfied for all $1\leq k\leq k_0$ and $\ell\leq 2(k_0+1)$.
\end{proof}

Recalling the $\sigma$-algebra $\mathcal{F}_j^{\infty}$ generated by the random variables $E_j,E_{j+1},\ldots$,
the following proposition relies on the previous one to provide a tractable bound for expressions of the form $\condE{F_k}{\mathcal{F}_j^{\infty}}$ with $k<j$, which in particular yield the first statement of Lemma~\ref{lem:tausv2}.

\begin{proposition}\label{prop:appendix_partb}
For $k_0$ and $C$ as in \Cref{prop:appendix_upperF} and all positive integers $k\ge k_0$ and $\ell,h\in\N$, as well as all $k<k_0$, $\ell\leq 2(k_0+1)$, and $h\in\N$, 
\begin{equation}\label{eq:appendix_partbsimpler}\condE{G_k^\ell}{\mathcal{F}_{k+h}^{\infty}}\leq \ell!\left(\frac{f(k-1)}{f(k)}\right)^{\ell}\left(1+C\frac{f(k)}{f(k+1)}\right)\sum_{j=0}^{\ell}\frac{1}{j!}\left(\frac{f(k)}{f(k+h-1)}\right)^{j}G_{k+h}^j.
\end{equation}
\end{proposition}

\begin{proof}
    Recall that $F_{k}=F_{k+h}+\sum_{j=k}^{k+h-1}E_j$. By combining the fact that $F_{k+h}$ is $\mathcal{F}_{k+h}^{\infty}$-measurable and the independence of the variables $(E_j)_{j=1}^{\infty}$, expanding $F_k^{\ell}$ gives
    \begin{equation}\label{eq:cond_form}
    \condE{F_k^\ell}{\mathcal{F}_{k+h}^{\infty}}=\sum_{j=0}^{\ell}\binom{\ell}{j}F_{k+h}^j\E{\left(\sum_{j=k}^{k+h-1}E_j\right)^{\ell-j}}\leq\sum_{j=0}^{\ell}\binom{\ell}{j}F_{k+h}^j\E{F_k^{\ell-j}},    
    \end{equation}
    where the inequality follows from the observation that $\sum_{j=k}^{k+h-1}E_j\leq F_k$. Multiplying both sides by $f(k-1)^{\ell}$ and applying \eqref{eq:appendix_boundF} to the right hand side gives
    \begin{align*}\condE{G_k^\ell}{\mathcal{F}_{k+h}^{\infty}}&\leq\sum_{j=0}^{\ell}\binom{\ell}{j}f(k-1)^{j}\E{G_k^{\ell-j}}F_{k+h}^j\\&\leq \sum_{j=0}^{\ell}\binom{\ell}{j}f(k-1)^{j}(\ell-j)!\left(\frac{f(k-1)}{f(k)}\right)^{\ell-j}\left(1+C\frac{f(k)}{f(k+1)}\right)F_{k+h}^j.
    \end{align*}
    As a result, the inequality \eqref{eq:appendix_partbsimpler} follows by rearranging the terms in the last sum and using that $G_{k+h}=f(k+h-1) F_{k+h}$. 
\end{proof}

We are now ready to prove Lemma~\ref{lem:tausv2}, which we split into two parts. 

\begin{proof}[Proof of Lemma~\ref{lem:tausv2}(A)]
\cb{To prove the first part, we recall that $G_k=f(k-1)F_k=f(k-1)\sum_{i=k}^\infty E_i$ and write 
\[ 
\E{G_kG_\ell}=f(k-1)f(\ell-1)\sum_{i=\ell}^\infty\sum_{j=k}^\infty \E{E_iE_j}.
\]
As $E_i$ and $E_j$ are independent when $i\neq j$  and $\E{E_i^2}=2/f(i)^2$, we obtain 
\[ \ba 
\E{G_kG_\ell}={}&f(k-1)f(\ell-1)\bigg[\sum_{i=\ell}^\infty \sum_{j=k}^\infty \frac{1}{f(i)f(j)}+\sum_{i=\ell}^\infty \frac{1}{f(i)^2}\bigg]\\
={}& \bigg(\sum_{i=\ell}^\infty \frac{f(\ell-1)}{f(i)}\bigg)\bigg(\sum_{j=k}^\infty \frac{f(k-1)}{f(j)}\bigg)+\frac{f(k-1)}{f(k)}\frac{f(\ell-1)}{f(\ell)} \sum_{i=\ell}^\infty \frac{f(k)f(\ell)}{f(i)^2}.
\ea \] 
By using~\eqref{eq:appendix_aux1}, we obtain for some constant $L=L(f)>0$ and any $k\in\N$, 
\[
\sum_{i=k}^\infty \frac{f(k-1)}{f(i)}  =\frac{f(k-1)}{f(k)}\bigg[1+\sum_{i=k+1}^\infty \frac{f(k)}{f(i)}\bigg]\leq \frac{f(k-1)}{f(k)}\Big[1+L\frac{f(k)}{f(k+1)}\Big],  
\] 
and for any $k\in\N$ and $\ell>k$, 
\[\ba 
\sum_{i=\ell}^\infty  \frac{f(k)f(\ell)}{f(i)^2}&\leq   \frac{f(k)}{f(\ell)}+\max\bigg\{\sum_{i=\ell+1}^\infty \!\!\Big(\frac{f(\ell)}{f(i)}\Big)^2,\sum_{i=k+1}^\infty \!\!\Big(\frac{f(k)}{f(i)}\Big)^2\bigg\}\\
&\leq  \frac{f(k)}{f(\ell)}+L \max\Big\{\Big(\frac{f(\ell)}{f(\ell+1)}\Big)^2,\Big(\frac{f(k)}{f(k+1)}\Big)^2\Big\}.
\ea \] 
As a result, since $\max\{x,y\}\leq x+y$, 
\[\ba 
\mathbb E[G_kG_\ell]\leq \frac{f(k-1)}{f(k)}\frac{f(\ell-1)}{f(\ell)}\bigg({}&\Big[1+L\frac{f(k)}{f(k+1)}\Big]\Big[1+L\frac{f(\ell)}{f(\ell+1)}\Big]\\
&+\frac{f(k)}{f(\ell)}+L \max\Big\{\Big(\frac{f(k)}{f(k+1)}\Big)^2,\Big(\frac{f(\ell)}{f(\ell+1)}\Big)^2\Big\}\bigg)\\
\leq \frac{f(k-1)}{f(k)}\frac{f(\ell-1)}{f(\ell)}\bigg({}&1+L\frac{f(k)}{f(k+1)}\Big(1+\frac{f(k)}{f(k+1)}+\frac{f(k+1)}{f(\ell)L} +L\frac{f(\ell)}{f(\ell+1)}\Big)\\
&+L\frac{f(\ell)}{f(\ell+1)}\Big(1+\frac{f(\ell)}{f(\ell+1)}\Big)\bigg).
\ea\]
By Assumption~\eqref{ass:D}, it follows $f(k+1)/f(\ell)$ is uniformly bounded over all $\ell>k$ and $f(k)/f(k+1)$ (resp.\ $f(\ell)/f(\ell+1)$) is uniformly bounded in $k$ (resp.\ in $\ell$). Hence,  we can take a large  constant $C=C(f)$  such that for all $k\in\N$ and $\ell>k$, 
\[ 
L\Big(1+\frac{f(k)}{f(k+1)}+\frac{f(k+1)}{f(\ell)L}+L\frac{f(\ell)}{f(\ell+1)}\Big)\leq C\qquad\text{and}\qquad L\Big(1+\frac{f(\ell)}{f(\ell+1)}\Big)\leq C, 
\]
so that we arrive at the desired result.
}
\end{proof}

We conclude by proving  Lemma~\ref{lem:tausv2}(B).

\begin{proof}[Proof of Lemma~\ref{lem:tausv2}(B)] 
Fix a positive integer $d$, a finite set $B\subseteq\N$ with $|B|\geq d$, and a sequence $(a_j)_{j\in B}$ with values in $\{1,2\}$. 
Recall that our goal is to find a constant $C_d=C_d(f)$ such that
\begin{equation}\label{eq:appendix_maingoal}\E{\prod_{j\in B} G_j^{a_j}}\leq C_d2^{-|B|}\prod_{j\in I}\left(\frac{f(j-1)}{f(j)}\right)^{a_{j}}.\end{equation}
If $d=1$ and $B=\{k\}$ for some $k\in\N$, the result follows by \Cref{prop:appendix_upperF}. Suppose $r\coloneq|B|\geq  d\geq 2$, and denote the elements of $B$ by  $j_1 < \cdots < j_r$. We distinguish between indices that are smaller than $k_0$ and at least $k_0$, and we note that there are at most $k_0$ many indices among the $j_1, \ldots, j_r$ that are smaller than $k_0$.  Let us suppose that $j_k<k_0$ for all $k<R$ and that $j_k\geq k_0$ for all $R\leq k\leq r$, for some $R\in\{1,\ldots, \min\{r,k_0\}\}$.
For all $k\le r$, define $B_k\coloneq B\setminus\{j_1,\ldots,j_k\}$.
Taking conditional expectations with respect to $\mathcal{F}_{j_2}^{\infty}$ and directly applying \Cref{prop:appendix_partb} gives 
\begin{align*}
&\E{\prod_{j\in B} G_j^{a_j}}=\E{\condE{G_{j_1}^{a_{j_1}}}{\mathcal{F}_{j_2}^{\infty}}\prod_{j\in B_1} G_j^{a_j}}\\&\leq a_{j_1}!\left(\frac{f(j_1-1)}{f(j_1)}\right)^{a_{j_1}}\Big(1+C\frac{f(j_1)}{f(j_1+1)}\Big)\sum_{t_2=0}^{a_{j_1}}\frac{1}{t_2!}\left(\frac{f(j_1)}{f(j_2-1)}\right)^{t_2}\E{G_{j_2}^{t_2}\prod_{j\in B_1} G_j^{a_j}}.
\end{align*}
Since $a_{j_1}\leq 2$ and as  $\frac{f(j_1)}{f(j_2-1)}$ and $1+C\frac{f(j_1)}{f(j_1+1)}$ are uniformly bounded from above by a constant $D =D(f)$   over all $j_1<j_2$  due to  Assumption~\eqref{ass:D}, we deduce that
\begin{equation}\label{eq:appendix_lemma1}\E{\prod_{j\in B} G_j^{a_j}}\leq 2 D\left(\frac{f(j_1-1)}{f(j_1)}\right)^{a_{j_1}}\sum_{t_2=0}^{a_{j_1}}\frac{D^{t_2}}{t_2!}\E{G_{j_2}^{t_2}\prod_{j\in B_1} G_j^{a_j}}.
\end{equation}
The expected value on the right-hand side now contains a term $G_{j_2}^{t_2+a_{j_2}}$. The exponent $t_2+a_{j_2}$ is at most $a_{j_1}+a_{j_2}\leq 4$. To analyse the expectations in the last expression, we again take conditional expectations with respect to $\mathcal{F}_{j_3}^{\infty}$ and apply \Cref{prop:appendix_partb}, as in~\eqref{eq:appendix_lemma1}, to obtain
\begin{align*}\E{G_{j_2}^{t_2}\prod_{j\in B_1} G_j^{a_j}}&=\E{\condE{G_{j_2}^{t_2+a_{j_2}}}{\mathcal{F}_{j_3}^{\infty}}\prod_{j\in B_2} G_j^{a_j}}\\&\leq (t_2+a_{j_2})!D\left(\frac{f(j_2-1)}{f(j_2)}\right)^{t_2+a_{j_2}}\,\sum_{t_3=0}^{t_2+a_{j_2}}\frac{D^{t_3}}{t_3!}\E{G_{j_3}^{t_3}\prod_{j\in B_2} G_j^{a_j}}.\end{align*}
Inserting this last bound into \eqref{eq:appendix_lemma1}, using that $\frac{(t_2+a_{j_2})!}{t_2!}=a_{j_2}!\binom{t_2+a_{j_2}}{t_2}$, and grouping terms gives
\begin{align*}
\E{\prod_{j\in B} G_j^{a_j}}
&\leq (2 D)^2\prod_{k=1}^2\left(\frac{f(j_k-1)}{f(j_k)}\right)^{a_{j_k}}\\
&\hspace{2cm}\cdot \sum_{t_2=0}^{a_{j_1}}\binom{t_2+a_{j_2}}{t_2}\left[D\frac{f(j_2-1)}{f(j_2)}\right]^{t_2}\sum_{t_3=0}^{t_2+a_{j_2}}\frac{D^{t_3}}{t_3!}\E{G_{j_3}^{t_3}\prod_{j\in B_2} G_j^{a_j}}.
\end{align*}
We now iterate this over all  indices $j_k$ with $k<R$ (so that $j_k<k_0$). Here, it is crucial that  the exponents of the random variables $G_j$ with $j\in B$ such that $j<k_0$ are at most $2(k_0+1)$, so that \Cref{prop:appendix_partb} can be applied. This is true, since the exponents increase at each step of the iteration by at most $2$, and we iterate at most $k_0$ times, as there are at most $k_0$ indices $j_k$ less than $k_0$. We thus arrive at
\begin{align*} \label{eq:intersmallindex}
    \E{\prod_{j\in B} G_j^{a_j}}
    &\leq (2 D)^{R-1}\prod_{k=1}^{R-1}\left(\frac{f(j_k-1)}{f(j_k)}\right)^{a_{j_k}}\\
    &\hspace{1cm}\cdot \sum_{(t_{2},\ldots,t_R)}\bigg[\prod_{k=2}^{R-1} \binom{a_{j_k}+t_k}{t_k} \bigg(D \frac{f(j_k-1)}{f(j_k)}\bigg)^{t_k}\bigg]\frac{D^{t_R}}{t_R!}\E{G_{j_R}^{t_R}\!\!\!\!\prod_{j\in B_{R-1}}\!\!\!\!G_j^{a_j}},
    \end{align*}
where the sum ranges over all non-negative integer tuples $(t_{2},\ldots,t_R)$ such that $0\leq t_{2}\leq a_{j_1}$ and, for each $3\leq k\leq R$,
\be\label{eq:tk}0\leq t_k\leq a_{j_{k-1}}+t_{k-1}.\ee
We now bound the expected value on the right-hand side in the same manner, but use that all indices $j_k$ with $k\geq R$ in the set $B_{R-1}$ are all at least $k_0$, so that we do not have to worry about the fact that the exponents of the random variables $G_j$ can become arbitrarily large, as we can still apply Proposition~\ref{prop:appendix_partb}. We can thus continue the iteration to obtain
\begin{align*}
\E{\prod_{j\in B} G_j^{a_j}}\leq &(2 D)^{r-1}\prod_{k=1}^{r-1}\left(\frac{f(j_k-1)}{f(j_k)}\right)^{a_{j_k}}\\
&\hspace{1.5cm}\cdot \sum_{(t_{2},\ldots,t_r)} \bigg[\prod_{k=2}^{r-1} \binom{a_{j_k}+t_k}{t_k} \bigg(D\frac{f(j_k-1)}{f(j_k)}\bigg)^{t_k}\bigg]\frac{D^{t_r}}{t_r!}\E{G_{j_r}^{t_r+a_{j_r}}},
\end{align*}
where  again the sum ranges over all non-negative integer tuples $(t_{2},\ldots,t_r)$ such that $0\leq t_{2}\leq a_{j_1}$ and $t_k$ satisfies~\eqref{eq:tk} for all $3\leq k\leq r$. Applying \Cref{prop:appendix_upperF} to the last expectation yields,  
\begin{equation*}\frac{D^{t_r}}{t_r!}\E{G_{j_r}^{t_r+a_{j_r}}}\leq \frac{D^{t_r+1}}{t_r!}(t_r+a_{j_r})!\left(\frac{f(j_r-1)}{f(j_r)}\right)^{t_r+a_{j_r}},\end{equation*}
which, inserted in the previous bounds, finally yields
\begin{equation}\label{eq:appendix_precise_bound}
\E{\prod_{j\in B} G_j^{a_j}}\leq (2 D)^{r}\prod_{k=1}^{r}\left(\frac{f(j_k-1)}{f(j_k)}\right)^{a_{j_k}}\sum_{(t_{2},\ldots,t_r)}\hspace{-0.3em} \bigg[\prod_{k=2}^{r} \binom{a_{j_k}+t_k}{t_k} \bigg(D \frac{f(j_k-1)}{f(j_k)}\bigg)^{t_k}\bigg].
\end{equation}
Our goal is to further bound the expression
\begin{equation}\label{eq:appendix_bigpart}\sum_{(t_{2},\ldots,t_r)}\hspace{-0.3em} \bigg[\prod_{k=2}^{r} \binom{a_{j_k}+t_k}{t_k} \bigg(D\frac{f(j_k-1)}{f(j_k)}\bigg)^{t_k}\bigg]\end{equation}
in \eqref{eq:appendix_precise_bound}. 
To do so, note that Assumption~\eqref{ass:D} implies the existence of some $m_0=m_0(f)\in\N$ such that, for all $j\geq m_0$, $D\frac{f(j-1)}{f(j)}\leq\frac{1}{2}$. Suppose that $|B\cap [m_0]| = \ell\le m_0$. Then, for each $k\leq\ell$, the corresponding index $t_k$ appearing in the sum satisfies
\begin{equation*}t_k\leq a_{j_{k-1}}+t_{k-1}\leq a_{j_{k-1}}+a_{j_{k-2}}+t_{k-2}\leq\cdots\leq\sum_{i=1}^{k-1}a_{j_{i}}\leq 2m_0,\end{equation*}
and hence, for each such $k$,
\begin{equation*}\binom{a_{j_k}+t_k}{t_k} \bigg(D\frac{f(j_k-1)}{f(j_k)}\bigg)^{t_k}\leq \binom{2m_0+2}{2m_0} \bigg(\max\bigg\{D\frac{f(j_k-1)}{f(j_k)}, 1\bigg\}\bigg)^{2m_0}.\end{equation*}
Moreover, since $j_k\leq m_0$, the whole term can be bounded by some constant $D'=D'(f)$. 
Thus, each index in $[m_0]$ contributes a uniformly bounded factor to \eqref{eq:appendix_bigpart}. 
As there are at most $m_0$ such indices, their total contribution can be absorbed into the constant $C_d$ in \eqref{eq:appendix_maingoal} by multiplying it by $D'^{m_0}$. 
Therefore, without loss of generality, we assume henceforth that $B$ has no elements in $[m_0]$. Under this assumption, we further bound \eqref{eq:appendix_bigpart} by allowing each $t_k$ to range over $\N_0$; note that this allows us to change order of the sum and the product \cb{(and it also allows to assume that $B$ has no elements in $[m_0]$, as otherwise changing the size of $B$ might affect the range of $t_k$)}, and it suffices to bound
\begin{equation*}\prod_{k=2}^{r}\sum_{t_{k}=0}^{\infty} \binom{a_{j_k}+t_k}{t_k} \bigg(D\frac{f(j_k-1)}{f(j_k)}\bigg)^{t_k}.\end{equation*}
To this end, we use the identity
\begin{equation*}\sum_{m=0}^{\infty} \binom{b+m}{m} x^m=\frac{1}{(1-x)^{b+1}}\qquad \text{for all}\qquad b\in \N_0,\, x\in [0,1)\end{equation*}
which follows from the probability density function of a negative binomial distribution with  $b+1$ successful trials  and success probability $1-x$. Since $D\tfrac{f(j_k-1)}{f(j_k)}\le \tfrac{1}{2}$,
\begin{align*}\sum_{(t_{2},\ldots,t_r)}\hspace{-0.3em} \bigg[\prod_{k=2}^{r} \binom{a_{j_k}+t_k}{t_k} \bigg(D\frac{f(j_k-1)}{f(j_k)}\bigg)^{t_k}\bigg]
&\leq \prod_{k=2}^{r}\left(1-D\frac{f(j_k-1)}{f(j_k)}\right)^{-a_{j_k}-1}\\
&\leq\prod_{k=2}^r2^{a_{j_k}+1}\leq 8^r,
\end{align*}
where we used that $a_j\le 2$ for the last inequality.
Inserting this back into \eqref{eq:appendix_precise_bound} gives
\begin{equation*}\E{\prod_{j\in B} G_j^{a_j}}\leq (16 D)^{|B|}\prod_{j\in B}\left(\frac{f(j-1)}{f(j)}\right)^{a_{j}}.\end{equation*}
Fix now $I\subseteq B$ with $|I|=d$, and write the right-hand side as
\begin{equation*}2^{-|B|}\prod_{j\in I}\left(\frac{f(j-1)}{f(j)}\right)^{a_{j}}\cdot\left[(32D)^{d}\prod_{j\in B\setminus I}32D\left(\frac{f(j-1)}{f(j)}\right)^{a_{j}}\right].\end{equation*}
To conclude the proof, we only need to find a bound for the expression within the square brackets which is independent of $B,I$ and the sequence $a_j$. 
To do so, note that $\lim_{j\to\infty}\frac{f(j-1)}{f(j)}=0$, so there is $j_0 = j_0(f)$ such that $32Df(j-1)/f(j)\leq 1$ for all $j\geq j_0$. We therefore have
\begin{align*}
\prod_{j\in B\setminus I} 32 D\left(\frac{f(j-1)}{f(j)}\right)^{a_{j}}
&\leq \prod_{j\in (B\setminus I)\cap [j_0]}32 D\left(\frac{f(j-1)}{f(j)}\right)^{a_{j}}\\
&\leq \prod_{j=1}^{j_0}\max\bigg\{D\bigg(\frac{f(j-1)}{f(j)}\bigg)^{2},1\bigg\},
\end{align*}
which depends on $f$ alone. 
Multiplying this last bound by $(32D)^{d}$ and $(D')^{m_0}$ gives the required constant $C_d$ in the statement of the lemma and finishes the proof.
\end{proof}

\textbf{Acknowledgements}
 Lichev was supported by the Austrian Science Fund (FWF)
grant number 10.55776/ESP624. Linker was supported by the ANID-FONDECYT regular grant 1252012. Lodewijks was financially supported by the CMM-Postdoc Visiting Programme and has received funding from the European Union’s Horizon 2022 research and innovation programme under
the Marie Sk\l{}odowska-Curie grant agreement no.\ 101108569. Mitsche was supported by ANID-FONDECYT regular grant 1261336. Linker would also like to thank his second baby for waiting with her birth before the paper was finished. For open access purposes, the authors have applied a CC BY public copyright license to any author-accepted manuscript version arising from this submission. The authors would like to sincerely thank the anonymous referees for their suggestions and comments, which helped to improve the presentation of the paper.

\bibliographystyle{abbrv}
	\bibliography{ref}

\providecommand{\vander}{van der }
\begin{thebibliography}{10}

\bibitem{Ald18}
D.~Aldous.
\newblock Random partitions of the plane via {P}oissonian coloring and a
  self-similar process of coalescing planar partitions.
\newblock {\em The Annals of Probability}, 46(4):2000--2037, 2018.

\bibitem{AS16}
N.~Alon and J.~Spencer.
\newblock {\em The probabilistic method}.
\newblock John Wiley \& Sons, 2016.

\bibitem{Omeretal}
O.~Angel, S.~Bhamidi, S.~Donderwinkel, N.~Maitra, and A.~Sakanaveeti.
\newblock Evolution of recursive trees with limited memory.
\newblock {\em arXiv preprint at arXiv:2510.18856}, 2025.

\bibitem{BBCS23}
A.-L. Basdevant, G.~Blanc, N.~Curien, and A.~Singh.
\newblock Fractal properties of the frontier in poissonian coloring, 2023.

\bibitem{BB-RKK23a}
{\'E}.~Bellin, A.~Blanc-Renaudie, E.~Kammerer, and I.~Kortchemski.
\newblock Uniform attachment with freezing.
\newblock {\em The {A}nnals of {A}pplied {P}robability}, 35(4):2882--2922,
  2025.

\bibitem{BB-RKK23b}
{\'E}.~Bellin, A.~Blanc-Renaudie, E.~Kammerer, and I.~Kortchemski.
\newblock Uniform attachment with freezing: {S}caling limits.
\newblock In {\em Annales de l'{I}nstitut {H}enri {P}oincare ({B})
  {P}robabilites et statistiques}, volume~61, pages 2679--2708. Institut Henri
  Poincar{\'e}, 2025.

\bibitem{Big76}
J.~D. Biggins.
\newblock The first- and last-birth problems for a multitype age-dependent
  branching process.
\newblock {\em Advances in Applied Probability}, 8(3):446--459, 1976.

\bibitem{BMMS24}
A.~Brandenberger, C.~Marcussen, E.~Mossel, and M.~Sudan.
\newblock Finding the root in random nearest neighbor trees.
\newblock {\em Random Structures \& Algorithms}, 68(2):e70057, 2026.

\bibitem{Cas25}
J.~Casse.
\newblock Siblings in {$d$}-dimensional nearest neighbour trees.
\newblock {\em Bernoulli}, 31(2):1300--1324, 2025.

\bibitem{Drm09}
M.~Drmota.
\newblock {\em Random trees: an interplay between combinatorics and
  probability}.
\newblock Springer, 2009.

\bibitem{FM10}
Q.~Feng and H.~M. Mahmoud.
\newblock On the variety of shapes on the fringe of a random recursive tree.
\newblock {\em Journal of applied probability}, 47(1):191--200, 2010.

\bibitem{FMP08}
Q.~Feng, H.~M. Mahmoud, and A.~Panholzer.
\newblock Phase changes in subtree varieties in random recursive and binary
  search trees.
\newblock {\em SIAM Journal on Discrete Mathematics}, 22(1):160--184, 2008.

\bibitem{Fuc08}
M.~Fuchs.
\newblock Subtree sizes in recursive trees and binary search trees:
  Berry--esseen bounds and poisson approximations.
\newblock {\em Combinatorics, Probability and Computing}, 17(5):661--680, 2008.

\bibitem{Gol19}
M.~Golosovsky.
\newblock {\em Citation analysis and dynamics of citation networks}.
\newblock Springer, 2019.

\bibitem{vdH17}
R.~{\vander H}ofstad.
\newblock {\em Random graphs and complex networks. {V}olume {O}ne.}, volume~43.
\newblock Cambridge university press, 2017.

\bibitem{vdHof17}
R.~{\vander H}ofstad.
\newblock Stochastic processes on random graphs.
\newblock {\em Lecture notes for the 47th Summer School in Probability
  Saint-Flour}, 2017, 2017.

\bibitem{vdH24}
R.~{\vander H}ofstad.
\newblock {\em Random graphs and complex networks. {V}olume {T}wo.}, volume~54.
\newblock Cambridge university press, 2024.

\bibitem{JLR11}
S.~Janson, T.~Luczak, and A.~Rucinski.
\newblock {\em Random graphs}.
\newblock John Wiley \& Sons, 2011.

\bibitem{Kam26}
E.~Kammerer.
\newblock Depth profile of depth-weighted trees with bounded weights.
\newblock {\em arXiv preprint arXiv:2607.05366}, 2026.

\bibitem{King75}
J.~F.~C. Kingman.
\newblock The first birth problem for an age-dependent branching process.
\newblock {\em The Annals of Probability}, pages 790--801, 1975.

\bibitem{Kom16}
J.~Komj{\'a}thy.
\newblock Explosive {C}rump-{M}ode-{J}agers branching processes.
\newblock {\em arXiv preprint arXiv:1602.01657}, 2016.

\bibitem{Kroll}
M.~Kroll.
\newblock Concentration inequalities for poisson point processes with
  application to adaptive intensity estimation.
\newblock {\em Stochastics}, pages 1--13, feb 2022.

\bibitem{LeGuvel2021}
R.~Le~Guével.
\newblock Exponential inequalities for the supremum of some counting processes
  and their square martingales.
\newblock {\em Comptes Rendus. Mathématique}, 359(8):969–982, Oct. 2021.

\bibitem{LeckMitWor20}
K.~Leckey, D.~Mitsche, and N.~Wormald.
\newblock The height of depth-weighted random recursive trees.
\newblock {\em Random Structures \& Algorithms}, 56(3):851--866, 2020.

\bibitem{LM24}
L.~Lichev and D.~Mitsche.
\newblock New results for the random nearest neighbor tree.
\newblock {\em Probability Theory and Related Fields}, 189(1):229--279, 2024.

\bibitem{Lod22}
B.~Lodewijks.
\newblock On joint properties of vertices with a given degree or label in the
  random recursive tree.
\newblock {\em Electronic Journal of Probability}, 27:1--45, 2022.

\bibitem{Mon26}
C.~M{\"o}nch.
\newblock Monotonicity of depth constants in general preferential attachment
  trees.
\newblock {\em arXiv preprint arXiv:2602.14741}, 2026.

\bibitem{Ner81}
O.~Nerman.
\newblock On the convergence of supercritical general ({C}-{M}-{J}) branching
  processes.
\newblock {\em Z. Wahrsch. Verw. Gebiete}, 57(3):365--395, 1981.

\bibitem{NBW11}
M.~Newman, A.-L. Barab{\'a}si, and D.~J. Watts.
\newblock {\em The structure and dynamics of networks}.
\newblock Princeton university press, 2011.

\bibitem{Nor98}
J.~R. Norris.
\newblock {\em Markov chains}.
\newblock Number~2. Cambridge university press, 1998.

\bibitem{PW08}
M.~D. Penrose and A.~R. Wade.
\newblock Limit theory for the random on-line nearest-neighbor graph.
\newblock {\em Random Structures \& Algorithms}, 32(2):125--156, 2008.

\bibitem{Pit94}
B.~Pittel.
\newblock Note on the heights of random recursive trees and random m-ary search
  trees.
\newblock {\em Random Structures \& Algorithms}, 5(2):337--347, 1994.

\bibitem{ReynaudBouret2003AdaptiveEO}
P.~Reynaud-Bouret.
\newblock Adaptive estimation of the intensity of inhomogeneous poisson
  processes via concentration inequalities.
\newblock {\em Probability Theory and Related Fields}, 126:103--153, 2003.

\bibitem{Rey06}
P.~Reynaud-Bouret.
\newblock Compensator and exponential inequalities for some suprema of counting
  processes.
\newblock {\em Statistics \& Probability Letters}, 76(14):1514--1521, 2006.

\bibitem{Ste04}
M.~Steele.
\newblock The {P}aley-{Z}ygmund argument and three variations.
\newblock {\em Course note}, 2004.

\bibitem{Tra25}
T.~Trauthwein.
\newblock Quantitative {CLT}s on the {P}oisson space via {S}korohod estimates
  and $p$-{P}oincar{\'e} inequalities.
\newblock {\em The Annals of Applied Probability}, 35(3):1716--1754, 2025.

\bibitem{Ver18}
R.~Vershynin.
\newblock {\em High-dimensional probability: An introduction with applications
  in data science}, volume~47.
\newblock Cambridge university press, 2018.

\bibitem{Wad09}
A.~R. Wade.
\newblock Asymptotic theory for the multidimensional random on-line
  nearest-neighbour graph.
\newblock {\em Stochastic Processes and their Applications}, 119(6):1889--1911,
  2009.

\end{thebibliography}

\end{document}